\documentclass{amsart}

\usepackage{amsthm,amsfonts,amsmath,amssymb}
\usepackage[all]{xy}
\usepackage{tensor}
\usepackage[utf8]{inputenc}
\usepackage{graphicx}
\usepackage{color}
\newcommand{\dtau}{\dot{\tau}}
\newcommand{\dt}{\dot{t}}
\newcommand{\de}{\dot{e}}

\newcommand{\at}{\vec{\alpha}}

\newcommand{\tS}{\widetilde{\Sigma}}
\newcommand{\ttS}{\widetilde{\Sigma}'}
\newcommand{\ti}{\widetilde{i}}
\newcommand{\tpr}{\widetilde{pr}}
\newcommand{\ssp}{\mathfrak{s}}
\newcommand{\gl}{\mathfrak{gl}}

\newcommand{\N}{\mathbb{N}}
\newcommand{\Z}{\mathbb{Z}}
\newcommand{\R}{\mathbb{R}}
\newcommand{\C}{\mathbb{C}}
\newcommand{\Cl}{\mathcal{C}}
\newcommand{\cyl}{Cyl}
\newcommand{\TC}{\underline{\C}}
\newcommand{\TR}{\underline{\R}}
\newcommand{\cpo}{\C P^1}
\newcommand{\rpo}{\R P^1}
\newcommand{\E}{\mathcal{E}}
\newcommand{\T}{\mathcal{T}}
\newcommand{\PS}{\mathcal{S}}
\newcommand{\GP}{\mathcal{G}}

\newcommand{\jo}{\mathcal{J}_{\Omega}^l(X)}
\newcommand{\rjo}{\R\jo}
\newcommand{\js}{\mathcal{J}_{S}^l}

\newcommand{\Tw}{\mathfrak{T}}

\newcommand{\M}{\mathcal{M}}
\newcommand{\mok}{\overline{\M}_{0,n}}
\newcommand{\rmok}{\R\overline{\M}_{2s,l}}
\newcommand{\rmook}{\R\overline{\M}_{2s,0}}
\newcommand{\rmk}{\R\M_{2s,0}}
\newcommand{\Uk}{\overline{\mathcal{U}}_{0,n}}
\newcommand{\TT}{\mathcal{T}}
\newcommand{\mdo}{\M^d(X,\Omega)}
\newcommand{\meo}{\M^e(X,\Omega)}
\newcommand{\mepo}{\M^{e'}(X,\Omega)}
\newcommand{\meeo}{\M^{e,e'}(X,\Omega)}
\newcommand{\rmeeo}{\R\meeo}
\newcommand{\rmddo}{\R\M^{d_1,d_2}(X,\Omega)}
\newcommand{\hmdo}{\widehat{\M}^d(X,\Omega)}
\newcommand{\rmdo}{\R\mdo}

\newcommand{\rmdod}{\R\M_{\delta}^d(X,\Omega)}

\newcommand{\crmdod}{\R\overline{\M}_{\delta}^d(X,\Omega)}
\newcommand{\hmdp}{\widehat{\M}_{\delta,\P}^d(X,\Omega)}
\newcommand{\rhmdp}{\R\hmdp}
\newcommand{\rchmdp}{\R_{c_S}\hmdp}
\newcommand{\rmdp}{\R\M_{\delta,\P}^d(X,\Omega)}
\newcommand{\crmdp}{\R\overline{\M}_{\delta,\P}^d(X,\Omega)}

\newcommand{\rhmdo}{\R\hmdo}
\newcommand{\rhmdoe}{\R\widehat{\M}^d_{\emptyset}(X,\Omega)}
\newcommand{\rhmdon}{\R\widehat{\M}^d_{S^1}(X,\Omega)}

\newcommand{\CS}{\mathcal{S}}
\newcommand{\PP}{\mathcal{P}}
\newcommand{\pdo}{\PP^d(X,\Omega)}

\newcommand{\pert}{\PP ert}
\renewcommand{\P}{\mathcal{P}}
\newcommand{\OO}{\mathcal{O}}

\newcommand{\mobs}{Mob(S,j_0)}
\newcommand{\mobsp}{Mob^+(S,j_0)}
\newcommand{\rmobsp}{\R Mob^+(S,j_0,c_S)}

\newcommand{\aut}{Aut}
\newcommand{\raut}{\R Aut(N,c_N)}

\newcommand{\rautt}{\R Aut(\TC^2,c_{\TC})}

\newcommand{\DB}{\overline{\partial}}
\newcommand{\uDB}{\underline{\overline{\partial}}}
\newcommand{\OB}{\overline{D}}
\newcommand{\rop}{\R Op(N)}
\DeclareMathOperator{\ddet}{Det}
\newcommand{\Det}{\ddet(N,c_N)}

\newcommand{\Dett}{\ddet(\TC^2,c_{\TC})}
\newcommand{\tD}{\widetilde{D}}
\newcommand{\bD}{\widehat{D}}
\newcommand{\Dn}{D^{\uh}}
\newcommand{\vD}{D^{\uh,\P}}
\newcommand{\uvD}{\underline{D}^{\uh,\P}}
\newcommand{\pD}{D^{\P}}
\newcommand{\kD}{\underline{D}^{\uk}}
\newcommand{\kDP}{\underline{D}^{\uk,P_0}}
\newcommand{\zD}{\underline{D}}
\newcommand{\zDP}{\underline{D}^{P_0}}
\newcommand{\kDPp}{\underline{D}^{\uk,P'_0}}

\newcommand{\kDa}{\underline{D}_{\at}^{\uk}}
\newcommand{\kDPa}{\underline{D}_{\at}^{\uk,P_{\at}}}
\newcommand{\zDa}{\underline{D}_{\at}}
\newcommand{\zDPa}{\underline{D}_{\at}^{P_{\at}}}
\newcommand{\kDPpa}{\underline{D}_{\at}^{\uk,P'_{\at}}}

\newcommand{\F}{\mathcal{F}}

\newcommand{\la}{l_{\at}}
\newcommand{\fa}{f_{\at}}
\newcommand{\lPKa}{l^{P_0,\uk}_{\at}}
\newcommand{\fPKa}{f^{P_0,\uk}_{\at}}
\newcommand{\lPpKa}{l^{P'_0,\uk}_{\at}}
\newcommand{\fPpKa}{f^{P'_0,\uk}_{\at}}
\newcommand{\lPHa}{l^{P,\uh}_{\at}}
\newcommand{\fPHa}{f^{P,\uh}_{\at}}

\newcommand{\ue}{\underline{e}}
\newcommand{\uh}{\underline{H}}
\newcommand{\up}{\underline{p}}
\newcommand{\uz}{\underline{z}}
\newcommand{\uk}{\underline{K}}
\newcommand{\uksi}{\underline{\xi}}
\newcommand{\uzeta}{\underline{\zeta}}
\newcommand{\ux}{\underline{x}}

\newcommand{\un}{\underline{n}}
\newcommand{\una}{\underline{n}_{\at}}
\newcommand{\uv}{\underline{v}}
\newcommand{\uw}{\underline{w}}

\newcommand{\fp}[2]{\tensor[_{#1}]{\times}{_{#2}}}
\newcommand{\fo}[2]{\tensor[_{#1}]{\otimes}{_{#2}}}

\newtheorem{prop}{Proposition}
\newtheorem{theo}{Theorem}
\newtheorem*{theon}{Theorem}
\newtheorem{lemma}{Lemma}
\newtheorem{definition}{Definition}
\newtheorem{remark}{Remark}

\DeclareMathOperator{\coker}{coker}
\DeclareMathOperator{\ind}{ind}
\DeclareMathOperator{\im}{im}
\DeclareMathOperator{\id}{id}
\DeclareMathOperator{\e}{e}
\DeclareMathOperator{\ex}{exp}
\DeclareMathOperator{\pt}{PT}
\DeclareMathOperator{\dd}{d}

\begin{document}

\title[Counting real rational curves on real symplectic $K3$ surfaces]{Counting real rational curves on real symplectic $4$-manifolds with vanishing first Chern class}
\author[R. Crétois]{Rémi Crétois}
\address{Matematiska institutionen\\
Uppsala universitet\\
Box 480\\
751 06 Uppsala, Sweden}
\email{remi.cretois@math.uu.se}
\date{\today}
\keywords{Real rational pseudo-holomorphic curves, Cauchy-Riemann operators, moduli spaces, real symplectic surfaces, $K3$ surfaces, Gromov-Witten invariants, Welschinger invariants}
\subjclass[2000]{Primary: 14N99; Secondary: 53D45, 14J28}
\thanks{Research supported in part by the ERC project TROPGEO}

\begin{abstract}
  We define a signed count of real rational pseudo-holomorphic curves appearing in a one-parameter family of real $Spin$ symplectic $K3$ surfaces. We show that this count is an invariant of the deformation class of the family. In the case of a real projective $K3$ surface, this invariant specializes to count real rational curves appearing in a linear system on the surface.
\end{abstract}

\maketitle

\tableofcontents

\section{Introduction}

Let $(X,\omega,c_X)$ be a real symplectic manifold of dimension $4$, i.e. $(X,\omega)$ is symplectic and $c_X$ is an anti-symplectic involution on $X$. Take a homology class $d\in H_2(X,\Z)$ such that $(c_X)_*d=-d$. Weschinger in \cite{wel1} defined enumerative invariants by counting real rational pseudo-holomorphic curves on $X$ appearing in the class $d$ and passing through a real configuration of points. However, his construction yields trivial invariants when the first Chern class of $(X,\omega)$ vanishes. Nonetheless, the algebraic counterpart of this problem that is counting real rational curves appearing in a linear system on a real $K3$ surface was studied by Kharlamov and R\u{a}sdeaconu in \cite{kharlr}, where the authors defined and computed a count invariant under deformation of the polarized surface when the linear system is primitive.

In the present article, we define what we believe is a natural generalization of both Welschinger's and Kharlamov-R\u{a}sdeaconu's work by studying the case of real symplectic $4$-manifolds with vanishing first Chern class.

To this end and as suggested by Kontsevich (\cite{konts}, \S 5.4), we consider a loop $\Omega =(\omega_t)_{t\in S^1}$ of symplectic structures on $X$ with vanishing first Chern class and such that $c_X$ is anti-symplectic for each structure. We let $\rjo$ be the space of almost-complex structures $J$ on $X$ which are tamed by an element of $\Omega$ and such that $d c_X\circ J = - J\circ d c_X$. The moduli space $\rmdo$ of real rational curves in the class $d$ which are pseudo-holomorphic for an element of $\rjo$ is a Banach manifold and comes with a Fredholm map $\pi : \rmdo\rightarrow \rjo$. This map is of index $-1$, which means in particular that there is no real rational $J$-holomorphic curve in the class $d$ for a generic $J$. This is the symplectic analogue of the algebraic statement that there is no holomorphic curve on a generic $K3$ surface. However, if in the algebraic case the moduli of the $K3$ surfaces admitting holomorphic curves is well-behaved, the image of the map $\pi$ is not.

Thus, rather than restricting our attention to the image of $\pi$, we will consider loops of almost-complex structures in order to track down the real rational curves appearing in the class $d$. Taking a loop $\gamma : S^1\rightarrow \rjo$ we will study the fiber product $\rmdo\fp{\pi}{\gamma}S^1$. When $\gamma$ is generic enough, this product is a manifold of dimension $0$, and when $d$ is a primitive class, it consists of a finite number of elements (see Lemma \ref{finite}). However, the cardinal of this set can vary along a homotopy of $\gamma$. In order to get an invariant out of it, we need to introduce an additional topological structure on $X$.

Namely, since the first Chern class of $(X,\Omega)$ vanishes, $X$ admits a $Spin$ structure. Suppose that $X$ admits a $Spin$ structure which is moreover invariant under the action of $c_X$; this is the case for example when $X$ is simply-connected. Then this structure admits two different orientations; we refer the reader to the \S \ref{spinpar} for the definitions, and simply mention here that when the real part of $X$ is non-empty, it is in fact an orientable surface and the data of an orientation for a real $Spin$ structure is equivalent to the choice of one of the two semi-orientations associated to the $Spin$ structure. Let $\ssp$ be an oriented real $Spin$ structure on $X$. Then it induces naturally an orientation of $\rmdo\fp{\pi}{\gamma}S^1$ (see Theorem \ref{first}). In other words, we can define a signed count $\chi_d^{\ssp}(X,\Omega,c_X;\gamma)$ of the elements of $\rmdo\fp{\pi}{\gamma}S^1$. We prove the following result (see Theorems \ref{first} and \ref{second}).

\begin{theon}
  Let $(X,\Omega = (\omega_t)_{t\in S^1},c_X)$ a one-parameter family of real symplectic manifolds of dimension $4$ with vanishing first Chern class, and suppose that $(X,c_X)$ admits an oriented real $Spin$ structure $\ssp$. Then for all primitive classes $d\in H_2(X,\Z)$ such that $(c_X)_* d = -d$, the signed count $\chi_d^{\ssp}(X,\Omega,c_X;\gamma)\in\Z$ is invariant under homotopy of $\gamma$ and $\Omega$.
\end{theon}

 Note that if $\Omega$ is constant, then $\rjo$ is contractible and hence the invariant given by the previous theorem vanishes. Taking one-parameter families of symplectic manifolds, the space $\rjo$ is not necessarily contractible anymore. For example, when $(X,\omega,c_X)$ is a $K3$ surface with a Kähler form $\omega$, it comes with a loop of Kähler structures (a subset of the twistor space) oriented by the choice of an oriented real $Spin$ structure on $(X,c_X)$. Applying the previous theorem we get an invariant $\chi_d(X,c_X)$ associated to the deformation class of the real $K3$ surface (see \S \ref{k3par}). Moreover, when $(X,c_X)$ admits a primitive linear system $h$ invariant by $c_X$ and is generic enough, then $\chi_h(X,c_X)$ actually counts the real rational curves in $h$, with appropriate signs (see Proposition \ref{twistortr}). In fact, we show that this count coincide with the one defined by Kharlamov and R\u{a}sdeaconu up to a sign (see Theorem \ref{samekr}).

 \begin{theon}
   Let $(X,c_X,h)$ be a real projective $K3$ surface with a non-empty $c_X$-invariant linear system which are generic enough. The absolute values of the count $\chi_h(X,c_X)$ and the count defined by Kharlamov and R\u{a}sdeaconu in \cite{kharlr} are equal.
 \end{theon}

We do not address the case of curves appearing in non-primitive homology classes, but we think that our approach extends to this case to provide a rational invariant and hope to show it in a future paper.

\section{Definition of the count}

\subsection{Oriented Fredholm maps}

In order to define the signed count of real rational pseudo-holomorphic curves we want, we need to introduce the notion of orientability for a map between two Banach manifolds. Let $M$ and $N$ be two Banach manifolds and $f : M\rightarrow N$ a smooth Fredholm map, i.e. for every point $x$ in $M$, the differential of $f$ at $x$ has finite dimensional kernel and cokernel. As is usual, we define the index of $f$, $\ind(f)$ to be $\dim(\ker(d_x f)) - \dim(\coker(d_x f)$ for any $x\in M$. Given such a map, one can also define a continuous real line bundle $\ddet(f)$ over $M$ (see e.g. \cite{zinger}), whose fiber over a point $x$ is the determinant of the differential $d_x f$
\[
\ddet(f) = \Lambda_{\R}^{\max}\ker(d_x f)\otimes \left(\Lambda_{\R}^{\max}\coker(d_x f)\right)^*.
\]

Since $M$ and $N$ are not necessarily finite dimensional, the notion of orientation on those manifolds is not clear. However, one can define a notion of relative orientation with respect to $f$ (see e.g. \cite{wang}).

\begin{definition}
  Let $f : M\rightarrow N$ be a smooth Fredholm map between two Banach manifolds. We say that $f$ is orientable (resp. oriented) if the line bundle $\ddet(f)$ is orientable (resp. oriented).
\end{definition}

On the other hand, if $f : M\rightarrow N$ is a smooth Fredholm map between two Banach manifolds, consider $g : L \rightarrow N$ a smooth map with $L$ a finite dimensional manifold with boundary. Suppose that $f$ and $g$ are transverse and write $M \fp{f}{g} L$ the fiber product along those two maps. According to the implicit function theorem, this is a smooth manifold with boundary $M\fp{f}{\partial g} \partial L$. It comes moreover with two smooth maps $\pi_M : M \fp{f}{g} L \rightarrow M$ and $\pi_L : M \fp{f}{g} L \rightarrow L$. Using the pullbacks by $\pi_M$ and $\pi_L$, we can consider the bundles $\ddet(f)$ and $\det(T L) = \Lambda_{\R}^{\dim L} T L$ over $M \fp{f}{g} L$. In the following, we will omit the pullback notation when the base of the bundles is clear.

We now have the following key fact.

\begin{prop}\label{fporient}
  Let $f : M\rightarrow N$ be a smooth Fredholm map between two Banach manifolds and let $g : L \rightarrow N$ be a smooth map with $L$ a finite dimensional manifold with boundary. Suppose that $f$ and $g$ are transverse.

Then there exist natural isomorphisms of line bundles
\begin{align*}
\det (T \left(M\fp{f}{g}L\right)) \overset{\mathfrak{o}}{=} \ddet(f)\otimes \det(T L),
\intertext{and}
\det (T \left(M\fp{f}{\partial g} \partial L\right)) \overset{\mathfrak{o}_{\partial}}{=} \ddet(f)\otimes \det(T \partial L),
\end{align*}

over  $M\fp{f}{g} L$ and $M\fp{f}{\partial g} \partial L$, such that when restricted over $M\fp{f}{\partial g} \partial L$, we have the following commuting square

\[
\xymatrix{
\det (T \left(M\fp{f}{\partial g}\partial L\right))  \otimes \mathcal{N}_{\partial (M\fp{f}{g} L)} \ar[r]^-{\mathfrak{o}_{\partial}\otimes d\pi_L} \ar[d] & \ddet(f)\otimes \det(T\partial L)\otimes \mathcal{N}_{\partial L} \ar[d]\\
\det (T \left(M\fp{f}{g} L\right))\ar[r]^{\mathfrak{o}} & \ddet(f)\otimes \det(T L) 
}
\]

where $\mathcal{N}_{\partial (M\fp{f}{g}L)} =  T \left(M\fp{f}{g}L\right) /T \left(M\fp{f}{\partial g} \partial L\right)$ and $\mathcal{N}_{\partial L} = T L / T\partial L$.

 In particular, when $f$ and $L$ are oriented, then so is $M \fp{f}{g} L$.
\end{prop}

 \begin{proof}
The fiber product $M\fp{f}{g} L$ is a smooth submanifold of the product $M \times L$ and at each point $(x,y) \in M\fp{f}{g} L$, its tangent space is the kernel of the surjective Fredholm map $d_xf\oplus -d_yg : T_xM \oplus T_yL \rightarrow T_{f(x)} N$. Thus, we have the equality $\det(T_{(x,y)} \left(M\fp{f}{g} L\right)) = \ddet(d_xf \oplus -d_yg)$.

On the other hand, for any $(x,y)\in M\fp{f}{g} L$, we have the following commutative diagram:
\[
\xymatrix{
0 \ar[r] & T_x M \ar[r] \ar[d]^{d_x f} & T_x M \oplus T_y L \ar[r] \ar[d]^{d_x f \oplus -d_y g} & T_y L \ar[r] \ar[d] & 0 \\
0 \ar[r] & T_{f(x)} N \ar[r] & T_{f(x)} N \ar[r] & 0 \ar[r] & 0.
}
\]
Thus, we have a natural isomorphism $\ddet(d_xf\oplus -d_yg) = \ddet(d_x f) \otimes \ddet(T_y L)$.

All in all, for all $(x,y)\in M\fp{f}{g} L$, we have an isomorphism 
\[
{\mathfrak{o}_{(x,y)} : \det(T_{(x,y)} \left(M\fp{f}{g} L\right)) \rightarrow \ddet(d_x f) \otimes \det(T_y L)}.
\]
 Using local trivializations for those three line bundles, one can then check that all the considered isomorphisms vary continuously.

We can then do the same construction for the boundary of $M\fp{f}{g} L$ to get the desired isomorphism $\mathfrak{o}_{\partial}$. The compatibility between those two isomorphisms is straightforward.
 \end{proof}

 \begin{remark}\label{imprem}
In the present paper, we will mainly use Proposition \ref{fporient} when we have $\dim(\partial L) + \ind(f) = 0$. In this case, the fiber product $M\fp{f}{g} L$ is a smooth manifold of dimension $1$ and the boundary is a discrete set of points. If $f$ and $L$ are oriented, then $\partial L$ is oriented using the inward normal convention. Thus $M\fp{f}{g} L$ and its boundary inherit an orientation that is again compatible with the inward normal convention. The orientation at a point $(x,y)$ of $\partial\left(M\fp{f}{g} L\right)$ is simply given by a sign which is the following : since $f$ and $\partial g$ are transverse at the point $f(x)=g(y)$, $d_y (\partial g) : T_y \partial L \rightarrow \coker(d_x f)$ is an isomorphism between two oriented vector spaces (since $\ker(d_x f) = \{0\}$ because $\ind(f)\leq 0$) so can either preserve of exchange the two orientations, which gives the desired sign.

In particular, if $f$ is a proper map and $L$ is compact, then the sum of the signs of the boundary points of $M\fp{f}{g} L$ gives zero. If $f$ is not proper, this is no longer true in general.
 \end{remark}

\subsection{Moduli space of real rational pseudo-holomorphic curves} 

We briefly recall the construction of the moduli space of real rational pseudo-holomorphic curves given in \cite{wel1}.

Let $(X,\Omega,c_X)$ a triple consisting of a smooth closed manifold of dimension $4$, of a smooth family $\Omega = (\omega_t)_{t\in S^1}$ of symplectic forms on $X$ and of an involutive diffeomorphism $c_X$ of $X$ such that for all $t\in S^1$, $c_X^* \omega_t = -\omega_t$. Fix an integer $l > 1$ large enough and define $\jo$ to be the set of almost-complex structures $J$ on $X$ which are of class $\Cl^l$ and tamed by some $\omega_t$, $t\in S^1$. Denote by $\rjo$ the subset of $\jo$ consisting of the almost-complex structures $J \in \jo$ such that $d c_X \circ J = -J\circ dc_X$. As proved in \cite{wel1}, the set $\rjo$ is a separable Banach manifold which is non-empty. However, it need not be contractible. It is nonetheless connected. Thus, the first Chern class $c_1(X,\Omega)\in H^2(X,\Z)$ of $(X,\Omega)$ is well-defined.

Let $S$ be an oriented $2$-sphere. Let $\js$ be the space of complex structures on $S$ of class $\Cl^l$ and compatible with the given orientation. Finally, fix an integer $1\leq k\leq l$ and a real number $p >2$.

\begin{definition}
  Let $d \in H_2(X,\Z)$ be a non-zero homology class. A parameterized rational pseudo-holomorphic curve on $(X,\omega,c_X)$ in the class $d$ is a triple $(u,j,J)\in L^{k,p}(S,X)\times \js \times \jo$ such that
\[
d u + J\circ d u \circ j = 0,\text{ and } u_*[S] = d.
\]

We say that $(u,j,J)$ is simple if it cannot be written as $u'\circ \phi$ where $(u',j',J')$ is a pseudo-holomorphic curve and $\phi : (S,j)\rightarrow (S,j')$ is a non-trivial ramified covering.
\end{definition}

The set $\pdo$ of simple parameterized rational pseudo-holomorphic curves in the class $d$ is a separable Banach manifold of class $\Cl^{l-k}$ (see \cite{MDS}, Proposition 3.2.1). Suppose that $(c_X)_* d = -d$. Then the group $Diff(S)$ of diffeomorphisms of $S$ of class $\Cl^{l+1}$ acts on $\pdo$ by reparameterization. Denote by $Diff^+(S)$ the subgroup of $Diff(S)$ consisting of orientation preserving diffeomorphisms. Then the quotient $\mdo$ of $\pdo$ by $Diff^+(S)$ can be seen as a Banach manifold of class $\Cl^{l-k}$ (see \cite{shev}, Corollary 2.2.3). Moreover, $\mdo$ comes with an action of $Diff(S)/Diff^+(S)\cong \Z/2\Z$. We denote by $\rmdo$ the fixed point set of this action, which is a separable Banach manifold of class $\Cl^{l-k}$ and $\pi : \rmdo \rightarrow \rjo$ the natural projection which is a smooth map. Then, any point of $\rmdo$ admits a lift $(u,j,J)\in\pdo$ such that there exists a unique involutive orientation-reversing diffeomorphism $c_S$ of $S$ with $u\circ c_S = c_X\circ u$ (see Lemma 1.3 in \cite{wel1}). We will say that an element of $\rmdo$ is a real rational pseudo-holomorphic curve in the class $d$.

Let us recall the following from Proposition 1.9 in \cite{wel1}.

\begin{prop}
  The map $\pi$ is Fredholm of index $c_1(X,\Omega)d~-~1$.
\end{prop}

\subsection{Oriented real $Spin$ structures}\label{spinpar}

Let $(X,c_X)$ be a smooth, closed connected manifold of dimension $n$ equipped with an involutive diffeomorphism $c_X$. Let $(E,g,c_E)$ be an oriented real vector bundle of rank $2r$ over $X$ equipped with a metric $g$ and an involutive automorphism $c_E$ lifting $c_X$ and being an orientation preserving (resp. reversing) isometry for $g$ in the fibers when $r$ is even (resp. odd). We say that $c_E$ is a real structure. Let us denote by $F_E$ the $O(2r)$-principal bundle of frames of $E$ and by $F_E^+$ the $SO(2r)$-principal bundle of positively oriented frames. The bundle $F_E$ admits a conjugate automorphism $\ue = (e_1,e_2,\ldots,e_{2r})\in F_E \mapsto \overline{\ue} = (e_1,-e_2,\ldots,e_{2r-1},-e_{2r})\in F_E$, i.e. if $\ue\in F_E$ and $M\in O(2r)$ then $\overline{\ue.M} = \overline{\ue}.\overline{M}$ where $\overline{M} = T M T$ and $T$ is the diagonal matrix $((-1)^{i}\delta_{ij})_{1\leq i,j\leq 2r}$. The involution $c_E$ induces one on $F_E^+$ defined by $\overline{c_E} : (e_1,\ldots,e_{2r})\in F_E^+\mapsto (c_E(e_1),-c_E(e_2),\ldots,c_E(e_{2r-1}),-c_E(e_{2r}))\in F_E^+$, which is a conjugate automorphism of $F_E^+$.

A $Spin$ structure on $E$ is a $Spin(2r)$-principal bundle $S_E$ over $X$ equipped with a double cover $S_E\rightarrow F_E^+$ which lifts the identity and restricts as the double cover $Spin(2r)\rightarrow SO(2r)$ on each fiber. A $Spin$ structure $S_E\rightarrow F_E^+$ is said to be real if $\overline{c_E}$ admits a lift $\sigma_E : S_E\rightarrow S_E$ as a conjugate morphism (see e.g. \cite{wangs}). Note that we do not require $\sigma_E$ to be an involution; it can be of order $2$ or $4$. On the other hand, if a $Spin$ structure is real, then there are exactly two lifts of $\overline{c_E}$, $\sigma_E$ and $-\sigma_E$.

\begin{definition}
A real $Spin$ structure $S_E\rightarrow F_E^+$ is oriented by the choice of a lift of $\overline{c_E}$.
\end{definition}

In fact, we consider oriented real $Spin$ structures only up to equivalence. Two oriented real $Spin$ structure $(S_E,\sigma_E)\rightarrow (F_E^+,\overline{c_E})$ and $(S'_E,\sigma'_E)\rightarrow (F_E^+,\overline{c_E})$ on $E$ are equivalent if there exists an isomorphism $f : S_E\rightarrow S'_E$ which lifts the identity on $F_E^+$ and satisfying $f\circ \sigma_E = \sigma'_E\circ f$.

Since the space of $c_E$-invariant metrics on $(E,c_E)$ is contractible, those definitions do not depend on the choice of $g$.

Suppose now that $J$ is a complex structure on $(E,c_E)$, i.e. $J$ is an endomorphism of $E$ which squares to minus the identity and satisfying $c_E\circ J = - J\circ c_E$. Then the fixed point set $\R E$ of $c_E$ is a real vector bundle of rank $r$ over the fixed point set $\R X$ of $c_X$ when it is not empty. Let us recall the following result (see e.g. \cite{wangs} Theorem 1).

\begin{prop}\label{orientspin}
  Let $(E,c_E,J)$ be a complex vector bundle on $(X,c_X)$ equipped with a real structure. Suppose that the real part of $X$ is non-empty and that $(E,c_E)$ admits a real $Spin$ structure. Then the bundle $\R E$ is orientable. Moreover, an orientation for the real $Spin$ structure on $(E,c_E)$ naturally induces an orientation on the real part of $(E,c_E)$. 

Finally, if $J'$ is a complex structure on $(E,c_E)$ which is homotopic to $J$ then the same orientation for the real $Spin$ structure give the same orientation of $\R E$ for $J$ and $J'$. 
\end{prop}

\begin{proof}
  The first part of the Proposition is proved in \cite{wangs}. The second part follows from the fact there exists an isomorphism between $(E,c_E,J)$ and $(E,c_E,J')$ which is homotopic to the identity as an automorphism of $(E,c_E)$.
\end{proof}

The existence of a real $Spin$ structure on a vector bundle $(E,c_E)$ is not an easy problem in general. The main case where we have the existence of such a structure is given by the following Lemma.

\begin{lemma}\label{existspin}
Let $(E,c_E)$ be an oriented real vector bundle of even rank equipped with a real structure over $(X,c_X)$. If $H^1(X,\Z/2\Z) = 0$ and $w_2(E) = 0$ then $(E,c_E)$ admits a real $Spin$ structure, which is unique. In particular, $(E,c_E)$ has exactly two oriented real $Spin$ structures.
\end{lemma}

\begin{proof}
  The condition $w_2(E)= 0$ guarantees the existence of a $Spin$ structure on $E$. The set of all $Spin$ structures on $E$ is then an affine space over $H^1(X,\Z/2\Z)$ which is trivial. Thus there is only one $Spin$ structure on $E$ and it must be real.
\end{proof}

\subsection{Main result}\label{mainsect}

\begin{definition}
  Let $(X,\Omega,c_X)$ be a smooth closed manifold of dimension $4$ equipped with a smooth family $\Omega = (\omega_t)_{t\in S^1}$ of symplectic forms and with an involutive diffeomorphism $c_X$ such that for all $t\in S^1$, $c_X^* \omega_t = -\omega_t$. We say $(X,\Omega,c_X)$ is a one-parameter family of real $Spin$ symplectic $K3$ surfaces if $c_1(X,\Omega)=0$ and if $(TX, dc_X)$ admits a real $Spin$ structure.

If $\Omega = \omega$ is constant, we say $(X,\omega,c_X)$ is a real $Spin$ symplectic $K3$ surface.
\end{definition}

 \begin{remark}\label{remktrois}
Let $(X,\omega,c_X)$ be a symplectic manifold of dimension $4$ equipped with an anti-symplectic involution such that $c_1(X,\omega) = 0$. Then Lemma \ref{existspin} implies that if $H^1(X,\Z/2\Z)=0$ there exist exactly two oriented real $Spin$ structures on $(TX,dc_X)$. In this case, Bauer proved in \cite{bauer}, Corollary 1.2 and Remark 1.3, that $X$ is homeomorphic to a $K3$ surface. It is not known if it is diffeomorphic to it but there is no known example of a simply-connected non-Kähler symplectic $4$-fold with vanishing first Chern class.
 \end{remark}

\begin{theo}\label{first}
Let $(X,\Omega,c_X)$ be a one-parameter family of real $Spin$ symplectic $K3$ surfaces. Then for all $d\in H_2(X,\Z)$ such that $(c_X)_* d = -d$, the map $\pi : \rmdo\rightarrow \rjo$ is orientable. Moreover, the choice of an oriented real $Spin$ structure on $(TX,dc_X)$ induces an orientation of $\pi$. The opposite oriented real $Spin$ structure induces the opposite orientation on $\pi$.
\end{theo}

 \begin{lemma}\label{finite}
  Let $(X,\Omega,c_X)$ be a one-parameter family of real $Spin$ symplectic $K3$ surfaces and $d\in H_2(X,\Z)$ such that $(c_X)_* d = -d$. Suppose that $d$ is primitive, i.e. that if $d = ke$ with $e\in H_2(X,\Z)$ and $k\in \Z$, then $k = \pm 1$. Then, a generic smooth map $\gamma : S^1\rightarrow \rjo$ is transverse to $\pi$ and the fiber product $\rmdo\fp{\pi}{\gamma} S^1$ is finite. Moreover, we can assume that all the real rational curves which are $\gamma(t)$-holomorphic for some $t\in S^1$ are irreducible.
 \end{lemma}

 \begin{proof}
First take a class $e\in H_2(X,\Z)$ such that $e' = -(c_X)_* e \neq e$. Then the moduli spaces $\meo$ and $\mepo$ of simple rational pseudo-holomorphic curves in degree $e$ and $e'$ are Banach manifolds of class $\Cl^{l-k}$ which come with Fredholm maps $\pi_e : \meo\rightarrow \jo$ and $\pi_{e'} : \mepo\rightarrow \jo$ of index $-2$ (see \cite{shev}, Corollary 2.2.3). Moreover, $\pi_e$ and $\pi_{e'}$ are transverse to each other (see \cite{wel1}, Proposition 2.9). Thus, the fiber product $\meeo = \meo\fp{\pi_{e}}{\pi_{e'}} \mepo$ is a Banach manifold of class $\Cl^{l-k}$ equipped with a Fredholm map $\pi_{e,e'} : \meeo \rightarrow \jo$ of index $-4$. The quotient $Diff(S)/Diff^+(S)$ then acts on $\meeo$ by reparameterization, and the fixed points set $\rmeeo$ of this action is a Banach manifold of class $\Cl^{l-k}$ with a Fredholm map $\pi_{e,e'}: \rmeeo \rightarrow \rjo$ of index $-2$. In particular, by Sard-Smale theorem, the image of a generic $\gamma : S^1 \rightarrow \rjo$ does not intersect $\im(\pi_{e,e'})$.

Secondly, take $d_1, d_2\in H_2(X,\Z)$ two different classes with $(c_X)_*d_i = -d_i$ for $i=1,2$. Then the projections $\pi_1 : \R\M^{d_1}(X,\Omega)\rightarrow \rjo$ and $\pi_2:\R\M^{d_2}(X,\Omega)\rightarrow \rjo$ are of index $-1$ and transverse to each other (see Proposition 2.9 in \cite{wel1}). Thus, the fiber product $\rmddo = \R\M^{d_1}(X,\Omega) \fp{\pi_{1}}{\pi_{2}} \R\M^{d_2}(X,\Omega)$ comes with a Fredholm map $\pi_{1,2} : \rmddo \rightarrow \rjo$ of index $-2$. In particular, by Sard-Smale theorem, the image of a generic $\gamma : S^1 \rightarrow \rjo$ does not intersect $\im(\pi_{1,2})$.

   On the other hand, also by Sard-Smale theorem, a generic $\gamma$ is transverse to $\pi$. 

Now take a $\gamma$ which is transverse to $\pi$ and such that for any $e\in H_2(X,\Z)$ with $e' = -(c_X)_* e \neq e$ we have $\im(\pi_{e,e'})\cap \im(\gamma) = \emptyset$ and for any $d_1,d_2\in H_2(X,\Z)$ with $d_1\neq d_2$ and $(c_X)_*d_i = -d_i$ for $i=1,2$, we have $\im(\gamma)\cap \im(\pi_{1,2}) = \emptyset$. Such a $\gamma$ is again generic. Moreover, the fiber product $\rmdo\fp{\pi}{\gamma} S^1$ is a $0$-dimensional manifold.

Assume by contradiction that it is infinite. Then one can find an injective sequence in it, and by Gromov's compactness theorem, there exist a subsequence converging to a stable pseudo-holomorphic curve. There are then three cases :
   \begin{itemize}
   \item the limit curve is irreducible and simple; this is not possible as $\rmdo\fp{\pi}{\gamma} S^1$ would then have an accumulation point.
     \item the limit curve is irreducible but not simple; this is ruled out by the assumption on $d$.
       \item the limit curve is reducible; this is not possible because of the conditions we have on $\gamma$.
   \end{itemize}

Thus $\rmdo\fp{\pi}{\gamma} S^1$ is finite.
 \end{proof}

On the other hand, under the same assumptions as in Lemma \ref{finite}, as mentioned in the Remark \ref{imprem}, the Theorem \ref{first} and the Proposition \ref{fporient} allow us to define for each element $\Cl \in \rmdo \fp{\pi}{\gamma} S^1$ and each oriented real $Spin$ structure $\ssp$ on $(TX, dc_X)$, a sign $\varepsilon_{\ssp}(\Cl)$. Then define :
\[
\chi^{\ssp}_d(X,\Omega,c_X;\gamma) = \sum_{\Cl\in \rmdo \fp{\pi}{\gamma} S^1}\varepsilon_{\ssp}(\Cl).
\]
However, since $\pi$ is not proper in general, it is not clear how $\chi^{\ssp}_d(X,\Omega,c_X;\gamma)$ vary along a homotopy of $\gamma$. This is answered by the following.

\begin{theo}\label{second}
  Let $(X,\Omega,c_X)$ be a one-parameter family of real $Spin$ symplectic $K3$ surfaces. Choose an oriented real $Spin$ structure $\ssp$ and let $d\in H_2(X,\Z)$ be a primitive class such that $(c_X)_* d = -d$. Then, for any generic $\gamma : S^1\rightarrow \rjo$, the integer $\chi^{\ssp}_d(X,\Omega,c_X;\gamma)$ depends only on the homotopy class of $\gamma$.
\end{theo}

Since $\rjo$ is connected, one can reformulate Theorem \ref{second} by saying that $\gamma\mapsto \chi^{\ssp}_d(X,\Omega,c_X;\gamma)$ defines a morphism $\chi^{\ssp}_d(X,\Omega,c_X;.)$ from $H_1(\rjo,\Z)$ into $\Z$.

On the other hand, one can choose $\gamma_{\Omega} : S^1\rightarrow\rjo$ generic such that for all $t\in S^1$, $\gamma_{\Omega}(t)$ is tamed by $\omega_t$. Moreover, such a $\gamma_{\Omega}$ is unique up to homotopy; thus, the integer $\chi^{\ssp}_d(X,\Omega,c_X; \gamma_{\Omega})\in \Z$ does not depend on the choice of $\gamma_{\Omega}$. We will denote it by $\chi^{\ssp}_d(X,\Omega,c_X)$. One can then easily check that $\chi^{\ssp}_d(X,\Omega,c_X)$ is invariant under deformation of $(X,\Omega,c_X)$, i.e. if $\Omega' = (\omega'_t)_{t\in S^1}$ is a family of real symplectic structures on $(X,c_X)$ which is homotopic to $\Omega$ through families of real symplectic structures, then $\chi^{\ssp}_d(X,\Omega,c_X) = \chi^{\ssp}_d(X,\Omega',c_X)$.

Theorem \ref{first} is a refinement of general results appearing in \cite{article} and \cite{georzing}. However, we will give a simpler proof in \S \ref{prooffirst} for the sake of completeness and also because we will need the explicit description of the orientation of $\pi$ in order to prove Theorem \ref{second}. The proof of Theorem \ref{second} appears in \S \ref{proofsecond}.

\subsection{K3 surfaces}\label{k3par}

We now explain how to get enumerative invariants associated to the deformation class of a real projective $K3$ surface from the invariants defined by Theorems \ref{first} and \ref{second}.

Let $(X,I,c_X)$ be a real $K3$ surface, i.e $(X,I)$ is a simply-connected complex surface with trivial canonical bundle and $c_X$ is an anti-holomorphic involution on $X$. Then $X$ admits a Kähler form $\omega$ for which $c_X$ is anti-symplectic (see Theorem 14.5 in \cite{barth}). Denote by $g = \omega(.,I.)$ the Kähler metric associated to the Kähler form $\omega$. 

The canonical bundle admits a real structure given by the transpose $(d c_X)^t$ of $d c_X$ and we denote by $H^0(X,K_X)_{+1}$ the $+1$-eigenspace of the map induced by $(dc_X)^t$ on $H^0(X,K_X)$. Taking a non-zero form $v\in H^0(X,K_X)_{+1}$, there is a unique endomorphism $J_1$ of $TX$ such that $g(J_1 ., .) = \Re(v(.,.))$. Moreover there exists a unique $\lambda>0$ such that $J =\lambda J_1$ is an almost-complex structure. Replacing $v$ by a positive multiple, we obtain the same almost-complex structure, and taking $-v$ gives $-J$.

 Moreover, as in \cite{barth} Lemma 13.1, $J$ is integrable, satisfies $IJ = - JI$ and $d c_X\circ J = J\circ d c_X$, and $g$ is Kähler for $J$. This implies that all the elements of the sphere
\[
\Tw(\omega,I) = \{aI + b J + c IJ \in End(TX), a,b,c\in \R^3, a^2+b^2+c^2 = 1\}
\]
are complex structures on $X$ making $g$ Kähler, and the only elements making $c_X$ anti-holomorphic form the circle
\[
\R \Tw(\omega,I) = \{aI + bJ + cIJ\in \Tw(\omega,I), b= 0\}.
\]
Note that the sets $\Tw(\omega,I)$ and $\R \Tw(\omega,I)$ do not depend on the choice of $v\in H^0(X,K_X)_{+1}$, but once we choose an orientation of $H^0(X,K_X)_{+1}$, $\R \Tw(\omega,I)$ is oriented as the boundary of the hemisphere of $\Tw(\omega,I)$ containing $J$. The opposite orientation of $H^0(X,K_X)_{+1}$ gives the opposite orientation of $\R \Tw(\omega,I)$. 

 On the other hand, as we mentioned in Remark \ref{remktrois}, $(TX, d c_X)$ admits exactly two oriented real $Spin$ structures which we will denote by $\ssp$ and $\ssp'$. 

More precisely, there is a bijection between the $Spin$ structure on $TX$ and the square roots of the canonical bundle $K_X$ of $X$ (see \cite{atiyah}, Proposition 3.2). In our case, since $K_X$ is trivial, this $Spin$ structure is given by the trivial line bundle $\OO_X$. Taking an isomorphism $\alpha : \OO_X^2\rightarrow K_X$, there are exactly two real structures $c_1$ and $c_2$ on $\OO_X$ such that their squares are equal to $\alpha^* (dc_X)^t$. The choice of one of those two real structures orients the real $Spin$ structure. It is also equivalent to the choice of a $(d c_X)^t$-invariant non-vanishing holomorphic $2$-form on $X$ up to a positive constant.

Thus, we get two loops of complex structure $\gamma_{\omega,I}^{\ssp},\gamma_{\omega,I}^{\ssp'} : S^1\rightarrow \R\Tw(\omega,I)$ which are obtained one from the other by reparameterizing using an orientation reversing diffeomorphism of $S^1$. We also get two loops of Kähler forms $\Omega_{\omega,I}^{\ssp} = (g(\gamma_{\omega,I}^{\ssp}(t).,.))_{t\in S^1}$ and $\Omega_{\omega,I}^{\ssp'} = (g(\gamma_{\omega,I}^{\ssp'}(t).,.))_{t\in S^1}$. The triples $(X,\Omega_{\omega,I}^{\ssp},c_X)$ and $(X,\Omega_{\omega,I}^{\ssp'},c_X)$ are one-parameter families of real $Spin$ symplectic $K3$ surfaces. The discussion following Theorem \ref{second} gives us $\chi_d^{\ssp}(X,\Omega_{\omega,I}^{\ssp},c_X)$ and $\chi_d^{\ssp'}(X,\Omega_{\omega,I}^{\ssp'},c_X)$ for any primitive class $d\in H_2(X,\Z)$ such that $(c_X)_* d = -d$.

 Since the structures $\ssp$ and $\ssp'$ are opposite and $\Omega_{\omega,I}^{\ssp}$ and $\Omega_{\omega,I}^{\ssp'}$ are the same family of symplectic forms travelled with opposite orientations, we have equality $\chi_d^{\ssp}(X,\Omega_{\omega,I}^{\ssp},c_X) = \chi_d^{\ssp'}(X,\Omega_{\omega,I}^{\ssp'},c_X)$. Since the space of Kähler forms $\omega$ on $(X,I,c_X)$ is convex, the integer $\chi_d^{\ssp}(X,\Omega_{\omega,I}^{\ssp},c_X)$ does not depend on the choice of $\omega$. We define 
\[
\chi_d(X,I,c_X) = \chi_d^{\ssp}(X,\Omega_{\omega,I}^{\ssp},c_X) = \chi_d^{\ssp'}(X,\Omega_{\omega,I}^{\ssp'},c_X).
\]

Let us underline that this number is now an invariant of the deformation class of the real $K3$ surface which does not depend on the choice of oriented real $Spin$ structure. 

\begin{prop}\label{twistortr}
  Let $(X,I,c_X)$ a real $K3$ surface and $\omega$ a Kähler form which makes $c_X$ anti-symplectic. Take a primitive class $d\in H_2(X,\Z)$ such that $(c_X)_*d= -d$ and $d^2\geq -2$. Then there is exactly one element $K\in\R \Tw(\omega,I)$ such that there exists a $K$-holomorphic curve in the class $d$.

Moreover, if $(X,I,c_X)$ is generic enough, then the loop $\gamma_{\omega,I}^{\ssp} : S^1 \rightarrow \R\mathcal{J}^l_{\Omega_{\omega,I}^{\ssp}}$ is transverse to the map $\pi : \R\mathcal{M}^d(X,\Omega_{\omega,I}^{\ssp})\rightarrow \R\mathcal{J}^l_{\Omega_{\omega,I}^{\ssp}}$, and the fiber product $\R\mathcal{M}^d(X,\Omega_{\omega,I}^{\ssp})\fp{\pi}{\gamma_{\omega,I}^{\ssp}} S^1$ contains only irreducible curves.
\end{prop}

\begin{proof}
Proposition 3.1 in \cite{bryanleung} shows that there is exactly one element $K\in\Tw(\omega,I)$ such that there exists a $K$-holomorphic curve in the class $d$. The structure $K$ must lie in $\R\Tw(\omega,I)$ by unicity.

If $(X,I,c_X)$ is generic enough, then it follows from Theorem 1.1 of Chen in \cite{chen} that all the curves appearing in the fiber product $\R\mathcal{M}^d(X,\Omega_{\omega,I}^{\ssp})\fp{\pi}{\gamma_{\omega,I}^{\ssp}} S^1$ are nodal and irreducible.

Let us suppose to simplify the notations that the structure $I$ is the one admitting a curve in the class $d$. Then the tangent to $\Tw(\omega,I)$ at $I$ gives a class $\mathfrak{c}$ in $H^1(X,TX)$ which is the tangent space to the space of deformations of $(X,I)$. Take a rational curve $u : \cpo\rightarrow X$ in the fiber product $\R\mathcal{M}^d(X,\Omega_{\omega,I}^{\ssp})\fp{\pi}{\gamma_{\omega,I}^{\ssp}} S^1$ and denote by $C\subset X$ its image. Then both the deformations of $X$ and of the pair $(X,C)$ are unobstructed (see \cite{fkps} Proposition 4.8). In particular, the image of $\mathfrak{c}$ in $H^1(S,N_u)$ is non-zero, where $N_u = u^*TX/T \cpo$ is the normal bundle of $u$; otherwise the curve $C$ would survive along the deformation given by $\Tw(\omega,I)$. 

Since the inclusion $\Tw(\omega,I)\subset \mathcal{J}^l_{\Omega_{\omega,I}^{\ssp}}$ is equivariant, the derivative of $\gamma_{\omega,I}^{\ssp}$ at $I$ gives a non-zero class in $H^1(S,N_u)_{+1}$. In other words, the derivative of $\gamma_{\omega,I}^{\ssp}$ at $I$ does not lie in the image of $d_{u}\pi$ (see \cite{shev} Corollary 2.2.3). Since $u$ is an immersion, $N_u$ is a line bundle of degree $-2$, so the image of $d_{u}\pi$ is codimension $1$. Thus $\gamma_{\omega,I}^{\ssp}$ is transverse to $\pi$.
\end{proof}

Thus, if $(X,I,c_X)$ is a real projective $K3$ surface generic enough and $d$ is the class of a non-empty linear system $h$ on $X$, then $\chi_d(X,I,c_X)$ counts the real rational curves appearing in $h$ with appropriate signs.

  Kharlamov and R\u{a}sdeaconu defined and computed an invariant signed count of real rational curves on real projective $K3$ surfaces (see \cite{kharlr}). The sign they give to a curve coincide with Welschinger's sign, i.e. it is given by the parity of the number of real isolated double points of the curve. In fact, the invariant they obtain is the same as the one we defined above, up to a sign coming from convention choices.

\begin{theo}\label{samekr}
Let $(X,I,c_X)$ be a real projective $K3$ surface generic enough and $d$ be the class of a non-empty linear system $h$ on $X$, invariant by $c_X$. The absolute values of the count $\chi_d(X,I,c_X)$ defined by Theorem \ref{first} and the count defined by Kharlamov and R\u{a}sdeaconu are equal.
\end{theo}

We obtain Theorem \ref{samekr} as a corollary of Proposition \ref{signk3} which appears in \S \ref{k3parc}.

\section{Proof of the invariance}

\subsection{Orientability of $\pi$ - Proof of Theorem \ref{first}}

\subsubsection{Complex vector bundles with real structure}\label{rec}

Let us recall some facts from \cite{art} which we will use in \S \ref{prooffirst}. Take an oriented $2$-sphere $S$ and consider a complex vector bundle $N$ over it. Equip $S$ with a real structure, i.e. with an orientation reversing involutive diffeomorphism $c_S$. A real structure on $N$ is an involutive automorphism $c_N$ of $N$ lifting $c_S$ and being $\C$-antilinear in the fibers. The fixed point set of $c_N$ is denoted by $\R N$. It is a real vector bundle over $\R S = Fix(c_S)$ of the same rank as $N$. The isomorphism classes of such pairs $(N,c_N)$ are classified by the rank of $N$, the degree of $N$ and the first Stiefel-Whitney class of $\R N$ (see \cite{bishurt} Propositions 4.1 and 4.2).

Now fix $j_0\in \js$ such that $d c_S\circ j_0 = -j_0 \circ d c_S$ and let us denote by $\rop$ the set of all generalized Cauchy-Riemann operators of class $\Cl^{l-1}$ on $(N,c_N)$ which commute with the action of $c_N$ (see Appendix C.1 in \cite{MDS}). The set $\rop$ of all such operators is an affine Banach space. In particular it is contractible.

 Note that the involution $c_N$ acts on the spaces $L^{k,p}(S,N)$ and $L^{k-1,p}(S,\Lambda^{0,1}S\otimes N)$. We denote by $L^{k,p}(S,N)_{+1}$ and $L^{k-1,p}(S,\Lambda^{0,1}S\otimes N)_{+1}$ the $+1$ eigenspaces of $c_N$. The elements of $\rop$ restrict as operators from $L^{k,p}(S,N)_{+1}$ to $L^{k-1,p}(S,\Lambda^{0,1}S\otimes N)_{+1}$, and this is how we will usually consider them. Moreover, those operators are Fredholm, so we can restrict the determinant line bundle $\Det$ over $\rop$. This line bundle is orientable.

On the other hand, let $\rmobsp$ be the group of automorphisms of $(S,j_0)$ commuting with $c_S$, and let
\begin{multline*}
\raut =  \left\{ f : N\rightarrow N\, |\, f \text{ lifts } \phi\in\rmobsp, \right. \\ \left. f \text{ is a }\C \text{-linear automorphism in the fibers}, f\circ c_N = c_N\circ f \right\} .
\end{multline*}
Then $\raut$ acts naturally on $\Det$, and in particular on its orientations. Moreover, this last action depends only on the homotopy classes of the elements of $\raut$.

Let us consider a particular example in detail. Take $\TC^2$ to be the trivial bundle of rank $2$ over $(S,j_0)$. Define a real structure $c_{\TC}$ on $\TC^2$ by $c_{\TC}(z,v) = (c_{S}(z),\overline{v})$. Recall from Lemma \ref{existspin} that $(\TC^2,c_{\TC})$ admits a unique real $Spin$ structure. When $\R S$ is non-empty, then the fixed points set of $c_{\TC}$ is the trivial bundle $\TR^2$ over $\R S$, and the two orientations of the real $Spin$ structure correspond naturally to the two orientations of $\TR^2$. Moreover, the real vector bundle $\TR^2$ over $\R S$ admits exactly two $Pin^{\pm}$ structures. We have the following Lemma (see Lemmas 2.2, 2.3 and 3.1 in \cite{art}).

\begin{lemma}\label{recall}
  Suppose that $\R S$ is empty. The group $\rmobsp$ is connected and the group $\rautt$ has two connected components : one containing the identity and the other containing the automorphism $r : (z,(v_1,v_2))\in \TC^2\mapsto (z,(-v_1,v_2))\in \TC^2$. The automorphism $r$ does not preserve the orientations of $\Dett$ and does not preserve the oriented real $Spin$ structures on $\TC^2$.

Suppose that $\R S$ is non-empty. The group $\rmobsp$ has two connected components : one containing the identity and the other containing an automorphism $h$ which exchanges the two hemispheres. The group $\pi_0(\rautt)$ is generated by three elements : $a : (z,v)\in \TC^2\mapsto (h(z),v)\in \TC^2$, $m : (z,(v_1,v_2))\in \TC^2\mapsto (z,(-v_1,v_2))\in \TC^2$, and $t$ is an automorphism of $\TC^2$ lifting the identity and such that its restriction to the real part of $\TC^2$ gives a generator of $\pi_1(SL_2(\R))$. The automorphisms $m$ and $t$ act non trivially on the orientations of $\Dett$ whereas $a$ acts trivially. Moreover, $a$ preserves the orientations and $Pin^{\pm}$ structures of $\TR^2$, $m$ preserves the $Pin^{\pm}$ structures but not the orientations of $\TR^2$, and $t$ preserves the orientations but not the $Pin^{\pm}$ structures of $\TR^2$.\qed
\end{lemma}

Let us fix once and for all an oriented real $Spin$ structure on $(\TC^2,c_{\TC})$ and a $Spin$ structure on $\TR^2$ when $\R S\neq \emptyset$ in the following way :
\begin{itemize}
\item if $\R S = \emptyset$ : taking the euclidean metric on the fibers, the bundle of oriented frame $F^+_{\TC^2}$ is the trivial bundle $S\times SO(4)$. The induced involution is $\overline{c_{\TC}} : (z,M)\in S\times SO(4)\mapsto (c_S(z),TMT^{-1})\in S\times SO(4)$, where $T\in SO(4)$ is the diagonal matrix $((-1)^i\delta_{i,j})_{1\leq i,j\leq 4}$. The oriented real $Spin$ structure is then given by the trivial bundle $S\times Spin(4)$ with the standard double cover to $S\times SO(4)$ and the lift of $\overline{c_{\TC}}$ is given by $\sigma_{\TC^2} : (z,p)\in S\times Spin(4)\mapsto (z,\tilde{T} p \tilde{T}^{-1}) \in S\times Spin(4)$, where $\tilde{T}\in Spin(4)$ is one of the two lifts of $T$.
\item if $\R S \neq \emptyset$ : the construction is the same as the previous one. Moreover, the orientation of $\TR$ induced by this oriented real $Spin$ structure (see  Proposition \ref{orientspin}) is the canonical one. The $Spin$ structure we fix on $\TR$ is given by the trivial bundle $\R S\times Spin(2)$ with the standard double cover to $\R S\times SO(2)$.
\end{itemize}

Finally, in both cases, we also fix an orientation on $\Dett$ as follows. Take the real Cauchy-Riemann operator $\DB_{\TC}\oplus \DB_{\TC}$ on $\TC^2$, where $\DB_{\TC} = \frac{1}{2} (d + i\circ d \circ j_0)$ is the standard Cauchy-Riemann operator on the complex-valued functions on $S$. It is surjective, and its kernel is of dimension $2$. Fix once and for all the orientation of $\ddet(\DB_{\TC}\oplus \DB_{\TC})$ by choosing $(1,0)\wedge (0,1)\in \Lambda_{\R}^2\ker(\DB_{\TC}\oplus \DB_{\TC})$ to generate this line. Now since the space of all real generalized Cauchy-Riemann operators on $(\TC^2,c_{\TC})$ is contractible, the orientation we chose induces an orientation of $\Dett$.

It is straightforward to check that if $(S',c_{S'},j_0')$ is another real sphere, isomorphic to $(S,c_S,j_0)$, then an isomorphism between the trivial bundles over those two spheres that preserves the fixed oriented real $Spin$ structure and $Spin$ structure on the real part also preserves the fixed orientations on the determinant bundles.

\subsubsection{Proof of Theorem \ref{first}}\label{prooffirst}

Let us first recall from \cite{shev} how $\rmdo$ inherits a structure of Banach manifold from $\pdo$. One chooses $j_0\in \js$ and considers $\hmdo = \pdo \cap L^{k,p}(S,X)\times\{j_0\}\times \jo$. It is a separable Banach manifold of class $\Cl^{l-k}$. The subgroup $\mobs\subset Diff(S)$ consisting of all $j_0$-holomorphic and $j_0$-anti-holomorphic diffeomorphisms of $S$ acts on $\hmdo$. Moreover, this action is of class $\Cl^{l-k}$. The only elements of $\mobs$ having fixed points are real structures on $(S,j_0)$, and a point of $(u,J)\in\hmdo$ can only be fixed by at most one real structure, which we will denote by $c_u$. We write $\rhmdo$ to be the reunion of all those fixed points. It is a disjoint union of separable Banach manifolds of class $\Cl^{l-k}$, each manifold being the fixed locus of a given real structure. The subgroup $\mobsp\subset \mobs$ consisting of automorphisms of $(S,j_0)$ acts freely on $\rhmdo$ in a $\Cl^{l-k}$-smooth way. Thus, the quotient $\rhmdo/\mobsp = \rmdo$ becomes a Banach manifold of class $\Cl^{l-k}$.

Moreover, define $\E$, $\E'$, $\T$ and $\T'$ to be the Banach bundles of class $\Cl^{l-k}$ over $\rhmdo$ whose fibers over a point $(u,J)\in\rhmdo$ are respectively $L^{k,p}(S,u^*TX)_{+1}$, $L^{k-1,p}(S,\Lambda^{0,1}S\fo{j_0}{J} u^*TX)_{+1}$, $L^{k,p}(S,TS)_{+1}$ and $L^{k-1,p}(S,\Lambda^{0,1} S\fo{j_0}{J} TS)_{+1}$. The indices $+1$ indicates that we take the $+1$-eigenspace of the linear map induced by $c_X$ and $c_u$ on each of those vector spaces. Then one has two bundle homomorphisms $\DB_{j_0} : \T\rightarrow \T'$ and $D : \E\rightarrow \E'$. The first one is the Cauchy-Riemann operator on the holomorphic line bundle $TS$. To define the second one, choose a metric $g_X$ on $X$ which is invariant by $c_X$ and let $\nabla$ be its associated Levi-Civita connection. Then the restriction $D_{u,J}$ of $D$ over $(u,J)\in\rhmdo$ is given by
\[
v\in \E_{(u,J)}\mapsto \nabla v + J\circ \nabla_{j_0} v + \nabla_v J\circ d u \circ j_0 \in \E'_{(u,J)}.
\]
Moreover, there is an injective morphism of Banach bundles from $\T$ to $\E$ and one from $\T'$ to $\E'$ given in the fibers over $(u,J)\in\rhmdo$ by $\xi\in L^{k,p}(S,TS)_{+1}\mapsto du(\xi)\in L^{k,p}(S,u^*TX)_{+1}$ for the first one and similarly for the second one. Then Lemma 1.4.2 in \cite{shev} states that the following diagram commutes
\begin{equation}\label{diagg}
\xymatrix{
0 \ar[r] & \T \ar[r] \ar[d]^{\DB_{j_0}} & \E \ar[r] \ar[d]^{D} & \E /\T \ar[r] \ar[d]^{\OB} & 0 \\
0 \ar[r] & \T'\ar[r]                 & \E' \ar[r]           & \E'/\T' \ar[r]           & 0
}
\end{equation}
where $\OB : \E/ \T \rightarrow \E'/\T'$ is induced by $D$ on the quotient. When restricted to fibers, the previous diagram is in fact a short exact sequence of Fredholm operators. Thus it induces a natural isomorphism of continuous line bundles over $\rhmdo$ between $\ddet(D)$ and $\ddet(\DB_{j_0})\otimes \ddet(\OB)$.

On the other hand, the action of $\mobsp$ on $\rhmdo$ lifts naturally to a continuous action on the three line bundles $\ddet(D)$, $\ddet(\DB_{j_0})$ and $\ddet(\OB)$. Corollary 2.2.3 in \cite{shev} and Proposition 1.9 in \cite{wel1} then state that there is a natural isomorphism between $\ddet(d\pi)$ and $\ddet(\OB)/\mobsp$.

Thus, to prove Theorem \ref{first}, it is enough to show that 
\begin{enumerate}
\item\label{one} $\ddet(\DB_{j_0})$ can be oriented in a canonical way,
\item\label{two} $\ddet(D)$ can be oriented by the choice of an oriented real $Spin$ structure on $(TX,dc_X)$,
\item\label{three} the action of $\mobsp$ preserves the orientations given in \ref{one} and \ref{two}.
\end{enumerate}
To this end, it will be useful to decompose the space $\rhmdo$ in the union of two open submanifolds : one, $\rhmdoe$, containing the real curves which have empty real part, the other, $\rhmdon$, containing the real curves which have non-empty real part. Note that the action of $\mobsp$ preserves those two submanifolds. Thus we can show \ref{one}, \ref{two} and \ref{three} for each of those two submanifolds independently.

\begin{prop}\label{dbj}
  The line bundle $\ddet(\DB_{j_0})$ admits a canonical orientation and the action of $\mobsp$ preserves it.
\end{prop}

\begin{proof}
Let us first choose for each $(u,J)\in\rhmdo$ an orientation of $\ddet(\DB_{j_0})_{(u,J)}$ in the following way. If the curve has empty real part, take an isomorphism $\psi : (S,c_u,j_0)\rightarrow (\cpo, c_{\emptyset},i)$, where $c_{\emptyset}(z) = -\frac{1}{\overline{z}}$; if the real part is non-empty, take an isomorphism $\psi : (S,c_u,j_0)\rightarrow (\cpo, c_{\rpo},i)$, where $c_{\rpo}(z) = \overline{z}$. In both cases, $\psi$ induces an isomorphism between $\ddet(\DB_{j_0})_{(u,J)}$ and $\ddet(\DB_{i})$. Now, fix once and for all an orientation of $\ddet(\DB_i)$ by requiring that $v_1(z) = (z-i)(z+i)$, $v_2(z) = (1-z)(1+z)$ and $v_3(z) = z$ is a positive basis of $\ker(\DB_i)$ when the real part is non-empty, and $v'_1(z) = (i-z)(i+z)$, $v'_2(z) = i(1-z)(1+z)$ and $v'_3(z) = iz$ is a positive basis of $\ker(\DB_i)$ when the real part is empty (those choices can appear to be arbitrary, but will be justified in the proof of Theorem \ref{second}, see Remark \ref{baserem}). This gives in both cases an orientation on $\ddet(\DB_{j_0})_{(u,J)}$. 

This orientation does not depend on the choice of $\psi$. Indeed, another isomorphism will differ from $\psi$ by an automorphism of $\cpo$ commuting with the appropriate real structure. But in the first case, the group of such automorphisms is connected; in the second case, it has two connected component : one containing the identity, and one containing $z\mapsto -z$. By a straightforward calculation, one can check that this last automorphism preserves the orientations of $\ddet(\DB_i)$. This also shows that the action of $\mobsp$ on these orientations is trivial.

Finally, it is clear that these orientations depend continuously on $(u,J)\in \rhmdo$.
\end{proof}

We now need two auxiliary technical results. Let $(\Sigma,g)$ be a compact and oriented Riemann surface. Let us denote by $F_{\Sigma}$ the oriented frame bundle of $\Sigma$. Then, given an immersed curve $a\subset \Sigma$, its tangential lift gives a class $\vec{a} \in H_1(F_{\Sigma},\Z/2\Z)$. Suppose moreover that the only singularities of $a$ are transverse double points, and let $m(a)\in \N$ be the number of such points.

\begin{lemma}\label{lifts}
  Let $(\Sigma,g)$ be a compact and oriented Riemann surface and let $a$ and $b$ be two immersed and connected curves on $\Sigma$ whose only singularities are transverse double points. Then
\[
\vec{a} = \vec{b} \text{ in } H_1(F_{\Sigma},\Z/2\Z) \Leftrightarrow a = b \text{ in } H_1(\Sigma,\Z/2\Z),\, \text{ and } m(a) = m(b) \mod 2.
\]
\end{lemma}

\begin{proof}
First, orient $a$ and $b$. Then for each node on $a$ or $b$, there is only one way to smoothen it while respecting the orientations. Moreover, if $a_1$ is obtained from $a$ by smoothening some nodes and $a_2$ is obtained from $a_1$ by smoothening one node $n$ on $a_1$, then on one hand, $\vec{a} = \vec{a_1} = \vec{a_2}$ in $H_1(F_{\Sigma},\Z/2\Z)$. On the other hand, if $n$ is at the intersection of two different components of $a_1$, then $a_2$ has one less component than $a_1$; if $n$ is an autointersection point of one of the components of $a_1$, then $a_2$ has one more component than $a_1$.

Thus, if $a'$ and $b'$ are the curves obtained from $a$ and $b$ after smoothening all the nodes, we have $\vec{a} = \vec{a'}$ and $\vec{b} = \vec{b'}$ in $H_1(F_{\Sigma},\Z/2\Z)$. Moreover, $a'$ (resp. $b'$) is the reunion of $m_1$ (resp. $m_2$) smooth simple closed curves and we have $m_1 = m(a)\mod 2$ and $m_2 = m(b)\mod 2$. So, using Theorem 1A in \cite{Johnson},
\[
\vec{a'} = \vec{b'} \text{ in } H_1(F_{\Sigma},\Z/2\Z) \Leftrightarrow a' = b'\text{ in } H_1(\Sigma,\Z/2\Z),\, \text{ and } m(a) = m(b)\mod 2,
\]
which ends the proof of this Lemma.
\end{proof}

Suppose $a:S^1\rightarrow \Sigma$ is a smooth immersed curve whose only singularities are transverse double points. The inclusion of vector bundles $d a : T S^1 \rightarrow a^* T \Sigma$ induces a $Spin$ structure on $a^* T \Sigma$. Indeed, fix an orientation of $T S^1$. This gives a framing of $a^*T \Sigma$ and thus an isomorphism between the oriented frame bundle of $a^* T\Sigma$ and the trivial bundle $SO(2)\times S^1$. The latter admits a natural $Spin$ structure that we can pullback on the former. The obtained $Spin$ structure does not depend on the choice of orientation on $S^1$ we made. Let us denote it by $\zeta_a$, and the other one by $-\zeta_{a}$. We have the following.

\begin{lemma}\label{pintriv}
  Let $(\Sigma,g)$ be a compact and oriented Riemann surface and $\gamma\in H_1(\Sigma,\Z/2\Z)$. Then there exists a $Spin$ structure $\zeta_{\gamma}$ on $\Sigma$ such that for any smooth immersed curve $a:S^1\rightarrow \Sigma$ in the class $\gamma$ whose only singularities are transverse double points, the pullback structure $a^*\zeta_{\gamma}$ coincides with $(-1)^{m(a)+1}\zeta_a$.

The structure $\zeta_{\gamma}$ is not unique, but if $\zeta'_{\gamma}$ satisfies the same conditions, then for any smooth curve $c: S^1\rightarrow \Sigma$ in the class $\gamma$, the pullbacks of those two structures by $c$ coincide.
\end{lemma}

\begin{proof}
  We can describe a $Spin$ structure on $\Sigma$ as an element of $H^1(F_{\Sigma},\Z/2\Z)$ with the only condition that it is non zero on the tangential framing of the boundary of a disc in $\Sigma$ (see \cite{Johnson}). Moreover, by definition, the $Spin$ structure takes value $1$ on a loop if and only if this loop cannot be lifted to the $Spin(2)$-principal bundle.

Fix an immersed curve $a:S^1\rightarrow \Sigma$ as in the statement. The $Spin$ structure $(-1)^{m(a)+1}\zeta_{a}$ is the structure where the tangential framing of $a^* T\Sigma$ can be lifted if and only if $m(a)+1$ is even. Thus, by Lemma \ref{lifts}, we can find a $Spin$ structure $\zeta_{\gamma}$ on $\Sigma$ which restricts to $(-1)^{m(a)+1}\zeta_a$ on $a$. Then, again by Lemma \ref{lifts}, this $Spin$ structure restricts to $(-1)^{m(b)+1}\zeta_b$ for all immersed curve smooth immersed curve $b:S^1\rightarrow \Sigma$ in the same class as $a$ and whose only singularities are transverse double points. 

The second part of the lemma is immediate.
\end{proof}

We can now resume our reasoning.

\begin{prop}\label{detd}
  The choice of a real oriented $Spin$ structure on $(TX,dc_X)$ naturally orients the line bundle $\ddet(D)$. Choosing the other real oriented $Spin$ structure gives the other orientation for $\ddet(D)$. Moreover, the action of $\mobsp$ on the orientations of $\ddet(D)$ is trivial.
\end{prop}

\begin{proof}
\textbf{\textbullet Curves with empty real part.} First, fix a curve $(u,J)\in\rhmdoe$ and let us orient $\ddet(D_{(u,J)})$. Consider as in \S \ref{rec} the trivial complex vector bundle $(\TC^2,c_{\TC})$ of rank $2$ on $(S,c_S)$. It is equipped with the oriented real $Spin$ structure defined at the end of \S \ref{rec}. Since the bundle $(u^*TX,d c_X)$ has rank two and vanishing first Chern class, it is isomorphic to $(\TC^2,c_{\TC})$. Choose such an isomorphism $f : \TC^2 \rightarrow u^*TX$ pulling back the oriented real $Spin$ structure given on $u^*TX$ to the one fixed on $\TC^2$ (using Lemma \ref{recall}). The orientation we fixed on $\Dett$ at the end of \S \ref{rec} gives an orientation of $\ddet(f^*D_{(u,J)})$ and thus an orientation of $\ddet(D)_{(u,J)}$ via $f$.

This orientation does not depend on the choice of $f$. Indeed another choice will differ from $f$ by an automorphism of $(\TC^2,c_{\TC})$ preserving the oriented real $Spin$ structure. But the group of all such automorphisms is connected (see Lemma \ref{recall}), so they act trivially on the orientations of $\Dett$. Moreover, the orientation we obtain depends continuously on $(u,J)$, hence orienting the bundle $\ddet(D)$ over $\rhmdoe$. From Lemma \ref{recall}, we note that taking the other oriented real $Spin$ structure on $(X,c_X)$ gives the other orientation of $\ddet(D)$.

\textbf{\textbullet Curves with non-empty real part.} We must now consider the case of the curves with non-empty real part. Fix an auxiliary metric on $\R X$, and for each class $\gamma\in H_1(\R X,\Z/2\Z)$ take a $Spin$ structure $\zeta_{\gamma}$ as given in Lemma \ref{pintriv}. Take a curve $(u,J)\in\rhmdon$. We proceed exactly as in the previous case, except that now $\TR^2$ is also equipped with the $Spin$ structure defined at the end of \S \ref{rec} and we require that $f$ pulls back the oriented real $Spin$ structure given on $u^*TX$ and the $Spin$ structure $\zeta_{u(\R S)}$ on $u^*T\R X$ to the corresponding fixed structures on $\TC^2$ and $\TR^2$ (using Lemma \ref{recall}). We then orient $\ddet(D)_{(u,J)}$ using $f$, and again by Lemma \ref{recall}, the resulting orientation does not depend on the choice of $f$. Lemma \ref{pintriv} also implies that other choices of $Spin$ structure on $\R X$ give the same $Spin$ structure on $u^*T\R X$, and hence the same orientation on $\ddet(D)_{(u,J)}$. Again, this gives an orientation of $\ddet(D)$ over $\rhmdon$, and taking the other oriented real $Spin$ structure on $(X,c_X)$ gives the other orientation for $\ddet(D)$.

In both cases, it follows from Lemma \ref{recall} that the group $\mobsp$ acts trivially on the fixed orientations.
\end{proof}

\begin{proof}[Proof of Theorem \ref{first}]
  The proof now follows from the Propositions \ref{dbj} and \ref{detd}.
\end{proof}

\subsection{K3 surfaces - continued}\label{k3parc}

We continue what we started in \S \ref{k3par}. Take a real $K3$ surface $(X,I, c_X)$ and $d\in H_2(X,\Z)$ a primitive class such that $(c_X)_* d = -d$. Suppose that $d$ is in fact the class of a linear system $h$ on $X$ such that all the rational curves in $h$ are nodal. As we saw in \S \ref{k3par}, the invariant $\chi_d(X,I,c_X)\in\Z$ counts the real rational curves in $h$ with some signs. Using the results of \S \ref{prooffirst} we can describe these signs more explicitely (see Proposition \ref{signk3}).

Take a Kähler form on $(X,I)$ such that $c_X$ is anti-symplectic for it and denote by $g = \omega(.,I.)$ the associated Kähler metric. Fix a non-zero element $v$ of $H^0(X,K_X)_{+1}$ and denote by $\ssp$ the associated oriented real $Spin$ structure on $(TX, d c_X)$. As we explained in \S \ref{k3par}, we can associate to all this data a loop $\Omega_{\omega,I}^{\ssp}$ of symplectic forms and a loop $\gamma_{\omega,I}^{\ssp} : S^1\rightarrow \R \mathcal{J}^l_{\Omega_{\omega,I}^{\ssp}}$ of complex structures which is transverse to the projection $\pi : \R\mathcal{M}^d(X,\Omega_{\omega,I}^{\ssp})\rightarrow \R \mathcal{J}^l_{\Omega_{\omega,I}^{\ssp}}$ and such that the fiber product $\R\mathcal{M}^d(X,\Omega_{\omega,I}^{\ssp})\fp{\pi}{\gamma_{\omega,I}^{\ssp}} S^1$ consists only of $I$-holomorphic curves (see Proposition \ref{twistortr}). Using Proposition \ref{fporient} and Theorem \ref{first} we obtain for each $u\in \R\mathcal{M}^d(X,\Omega_{\omega,I}^{\ssp})\fp{\pi}{\gamma_{\omega,I}^{\ssp}} S^1$ a sign $\varepsilon(u)\in\{-1,+1\}$ which does not depend on the choice of oriented real $Spin$ structure. On the other hand, we denote by $m(u)$ the number of real isolated double points of $u$.

\begin{prop}\label{signk3}
 There exists a sign $\epsilon\in \{-1,+1\}$ such that for all $u\in \R\mathcal{M}^d(X,\Omega_{\omega,I}^{\ssp})\fp{\pi}{\gamma_{\omega,I}^{\ssp}} S^1$, $\varepsilon(u) = \epsilon (-1)^{m(u)+1}$.
\end{prop}

We obtain Theorem \ref{samekr} as a corollary of Proposition \ref{signk3}. Note that we do not do the computation of $\epsilon$ as it boils down to a matter of conventions choices.

\begin{proof}
First, we define the sign $\epsilon$. To that end, consider the real curve $(\cpo,i,c)$, where $c$ can be $c_{\rpo}$ or $c_{\emptyset}$. The canonical bundle $K_{\cpo}$ comes with a real structure $(d c)^t$ and we denote by $H^1(\cpo,K_{\cpo})_{+1}$ the $1$-dimensional $+1$ eigenspace of $(d c_X)^t$. The holomorphic structure on $K_{\cpo}$ is given by an injective real Cauchy-Riemann operator which we denote by $\DB_i^*$, and we have $\coker(\DB_i) = H^1(\cpo,K_{\cpo})_{+1}$. Thus orienting the line $\ddet(\DB_i^*)$ is the same as orienting $H^1(\cpo,K_{\cpo})_{+1}$. Let us describe two ways of doing so.

The first way is to note that there is a linear isomorphism $H^1(\cpo,K_{\cpo})_{+1}\rightarrow i\R$ given by integrating the $(1,1)$-forms in $H^1(\cpo,K_{\cpo})_{+1}$ over $\cpo$.

The second way is to consider the exact sequence
\[
0\rightarrow T\cpo \rightarrow T\cpo\oplus K_{\cpo} \rightarrow K_{\cpo}\rightarrow 0
\]
of holomorphic bundles. This gives an isomorphism $\ddet(\DB_i^*) = \ddet(\DB_i)^*\otimes \ddet(\DB_i\oplus\DB_i^*)$. On the one hand we orient $\ddet(\DB_i)$ as in Proposition \ref{dbj}. On the other hand, the line bundle $\det(T\cpo\oplus K_{\cpo})$ admits a natural non-vanishing section which provides an oriented real $Spin$ structure on $(T\cpo\oplus K_{\cpo},d c\oplus (d c)^t)$. When the real part of the curve is non-empty, we take the $Spin$ structure $\zeta_{K_{\cpo}}$ on $\R (T\cpo\oplus K_{\cpo})$ to be the one induced by the direct sum decomposition. Thus, we can orient $\ddet(\DB_i\oplus\DB_i^*)$ in the same way as in Proposition \ref{detd}.

We let $\epsilon$ be $+1$ if the two previous orientations coincide and $-1$ otherwise.

Now, take $u\in \R\mathcal{M}^d(X,\Omega_{\omega,I}^{\ssp})\fp{\pi}{\gamma_{\omega,I}^{\ssp}} S^1$. By assumption, $u$ is a $\Z/2\Z$-equivariant holomorphic immersion, so the quotient $N_u = u^*TX/T\cpo$ is an holomorphic line bundle of degree $-2$, which is equipped with a real structure $c_N$. Let us denote by $\OB_u$ the associated real Cauchy-Riemann operator on $N_u$. As we mentioned in \S \ref{prooffirst}, it follows from Corollary 2.2.3 in \cite{shev} and Proposition 1.9 in \cite{wel1} that $\ker(d_u\pi) = \{0\}$ and $\coker(d_u\pi) = \coker(\OB_u)$. Thus, the sign $\varepsilon(u)$ is determined by looking at the isomorphism $d_0\gamma_{\omega,I}^{\ssp} : t \in T_0 S^1\mapsto t JI\circ d u\circ i\in \coker(\OB_u)$, where $\gamma_{\omega,I}^{\ssp}(0) = I$ and $J$ is the complex structure defined by $g$ and $v$ in \S \ref{k3par} (see the proof of Corollary 2.2.3 in \cite{shev}).

On the other hand, the map 
\[
f : [\xi]\in N_u = u^*TX/T\cpo \mapsto v(du(.),\xi)\in K_{\cpo}
\]
 is a $\Z/2\Z$-equivariant isomorphism between the holomorphic line bundles $(N_u,c_N)$ and $(K_{\cpo},(d c)^t)$. Moreover, the sequence
\[
0\rightarrow T\cpo \rightarrow u^*TX \rightarrow N_u\rightarrow 0
\]
splits. Let us choose a section $s : N_u\rightarrow u^*TX$ and define the isomorphism
\[
F : (x,s(y))\in u^*TX = T\cpo\oplus N_u \mapsto (x,f(s(y)))\in T\cpo\oplus K_{\cpo}.
\]
By construction, the isomorphism $F$ preserves the oriented real $Spin$ structures fixed on the two bundles; when the real part of the curve is non-empty, we use Lemma \ref{pintriv} to see that $F$ sends the $Spin$ structure $\zeta_{u(\rpo)}$ to the $Spin$ structure $(-1)^{m(u)+1}\zeta_{K_{\cpo}}$. Thus, pulling back by $f$ the first orientation on $\ddet(\DB^*_i)$, we obtain the orientation on $\ddet(\OB_u)$ defined by Propositions \ref{dbj} and \ref{detd} if and only if $\epsilon (-1)^{m(u)+1} = 1$. 

We can now explicitely evaluate the sign of $d_0\gamma_{\omega,I}^{\ssp}$. Indeed, we have
\begin{equation*}
\begin{split}
f\circ d_0\gamma_{\omega,I}^{\ssp}(t) & = t v(du(.),JI d u (i.))\\
 & = - t v(du(.), Jdu (.)),
\end{split}
\end{equation*}
and evaluating this $(1,1)$-form on a pair $(x,ix)$ of non-zero tangent vectors to $\cpo$, we obtain
\begin{equation*}
  \begin{split}
    t\Im(-v(du(x),Jdu(ix))) & =t\Im(iv(du(x),Jdu(x)))\\
 & = t\Re(v(du(x),Jdu(x)))\\
 & =t \frac{1}{\lambda}g(du(x),du(x)),
  \end{split}
\end{equation*}
where $\lambda>0$ was defined in \S \ref{k3par}. In particular, $f\circ d_0\gamma_{\omega,I}^{\ssp}$ always preserves the orientations. Thus, $d_0\gamma_{\omega,I}^{\ssp}$ preserves the orientations if and only if $\epsilon (-1)^{m(u)+1} = 1$, which proves the proposition.
\end{proof}

\subsection{Contribution of the reducible curves}\label{proofsecond}

\subsubsection{Black box around a reducible curve}\label{blackb}

Let $(X,\Omega,c_X)$ be a one-parameter family of real $Spin$ symplectic $K3$ surfaces and $d\in H_2(X,\Z)$ with $(c_X)_* d = -d$. Fix an oriented real $Spin$ structure $\ssp$ on $(TX,d c_X)$. Let $\gamma, \gamma' : S^1 \rightarrow \rjo$ two generic maps in the sense of Lemma \ref{finite} and which are homotopic. In \S \ref{mainsect} we defined the integers $\chi^{\ssp}_d(X,\Omega,c_X;\gamma)$ and $\chi^{\ssp}_d(X,\Omega,c_X;\gamma')$. Let $\cyl = S^1\times [0,1]$ and choose a smooth homotopy $\delta : \cyl\rightarrow \rjo$ such that $\delta(.,0) = \gamma$ and $\delta(.,1) = \gamma'$. Let us denote by $\rmdod$ the fiber product $\rmdo\fp{\pi}{\delta} \cyl$. If $\delta$ is generic, then Proposition \ref{fporient} and Theorem \ref{first} imply that $\rmdod$ is a smooth and oriented $1$-dimensional manifold with boundary $\rmdo\fp{\pi}{\gamma} S^1 \cup \rmdo \fp{\pi}{\gamma'} S^1$. Moreover, the orientation of $\rmdo \fp{\pi}{\gamma} S^1$ (resp. $\rmdo \fp{\pi}{\gamma'} S^1$) given by Proposition \ref{fporient} and Theorem \ref{first} agrees with (resp. in not the same as) the orientation it inherits from $\rmdod$. However, since $\rmdod$ is not compact in general, we cannot apply directly Stokes' Theorem to conclude on the equality of $\chi^{\ssp}_d(X,\Omega,c_X;\gamma)$ and $\chi^{\ssp}_d(X,\Omega,c_X;\gamma')$.

Denote by $\crmdod$ the Gromov compatification of $\rmdod$, that is the closure of $\rmdod$ in the space of stable real rational curves of degree $d$ for the Gromov topology. It is a compact topological space and the projection $\pi : \rmdod\rightarrow \cyl$ naturally extends to $\crmdod$ as a continuous map.

\begin{lemma}\label{fincomp}
  If the homotopy $\delta : \cyl\rightarrow \rjo$ between $\gamma$ and $\gamma'$ is generic enough, then the set of critical points of $\pi: \rmdod\rightarrow \cyl$ has no accumulation point.
\end{lemma}

\begin{proof}
Notice first that the critical points of $\pi$ are non-imersed curves. Indeed, it follows from Proposition 1.9 in \cite{wel1} that the kernel of $d \pi$ at a point $(u,J)\in\rmdod$ is isomorphic to $\ker(\OB_{(u,J)})$. Since $c_1(X).d = 0$, this kernel is trivial when the curve $u$ is immersed (see the sequence (2) in \cite{wel1}).

On the other hand, Proposition 2.7 in \cite{wel1} implies that if $\delta$ is generic enough, the subset of $\rmdod$ consisting of non-immersed curves has no accumulation point.
\end{proof}

Let $\Cl\in \crmdod\setminus\rmdod$ be a reducible curve. Take an open neighborhood $N_{\Cl}$ of $\Cl$ in $\crmdod$ such that its boundary as an open subset of $\crmdod$ contains only a finite number of points all of which are in $\rmdod$ and non-isolated in $\rmdod\setminus N_{\Cl}$. We will see later an example of such a neighborhood (see Lemma \ref{neig}). The orientation of $\rmdod$ gives a sign to each boundary point of $N_{\Cl}$, $+$ if by following the orientation of $\rmdod$ one goes from outside $N_{\Cl}$ to inside at the boundary point, $-$ otherwise. Let us write $m_{\ssp}(N_{\Cl})$ the sum of the signs of all the boundary points.

\begin{definition}
The open set $N_{\Cl}$ is called a black box around $\Cl$ and the integer $m_{\ssp}(N_{\Cl})$ is its contribution.
\end{definition}

Let us now show the existence of black boxes.

\begin{lemma}\label{neig}
  Let $\Cl\in \crmdod\setminus\rmdod$ be a reducible curve. Then for any open neighborhood $V\subset \crmdod$ of $\Cl$ there exists a black box $W\subset V$ around $\Cl$.
\end{lemma}

\begin{proof}
  Take an open neighborhood $V\subset\crmdod$ of $\Cl$. The space $\crmdod$ being metrizable, fix one such metric. Since $\crmdod\setminus \rmdod$ is countable, there exists $\epsilon>0$ such that the open ball $B(\Cl,\epsilon)$ centered at $\Cl$ and of radius $\epsilon$ is included in $V$ and its boundary contains only elements of $\rmdod$. 

On the other hand, $\pi^{-1}(\{t_0\})\cap \rmdod$ has no accumulation point (Lemma \ref{fincomp}), so the set $\pi^{-1}(\{t_0\})\cap\partial B(\Cl,\epsilon)$ is finite. For each $v\in\pi^{-1}(\{t_0\})\cap\partial B(\Cl,\epsilon)$, choose a small neighborhood $V_v\subset\rmdod$ of $v$ such that $\pi^{-1}(\{t_0\})\cap \overline{V_v} = \{v\}$. Then define $W = B(\Cl,\epsilon)\setminus\displaystyle\bigsqcup_{v\in\partial V\cap\pi^{-1}(\{t_0\})} \overline{V_v}$. It is an open neighborhood of $\Cl$ satisfying $\partial W\cap\pi^{-1}(\{t_0\}) = \emptyset$.

Then for all small enough disc $\Delta\subset\cyl$ containing $\pi(\Cl)$, we have $\partial (\pi^{-1}(\Delta)\cap W) = \partial(\pi^{-1}(\Delta))\cap \overline{W}$. Indeed, otherwise there would be a sequence $v_n\in\partial W$ such that $\pi(v_n)$ would converge to $t_0$, which contradicts the property of $W$. Then taking $\Delta$ with smooth boundary transverse to $\pi$, we get $\partial (\pi^{-1}(\Delta)\cap W) = \pi^{-1}(\partial \Delta)\cap \overline{W}\subset \pi^{-1}(\partial \Delta)$, and this last set is a finite subset of $\rmdod$. Thus $\pi^{-1}(\Delta)\cap W$ is a black box around $\Cl$ contained in $V$.
\end{proof}

Now, suppose that for each curve $\Cl\in\crmdod\setminus\rmdod$ we have a black box $N_{\Cl}$ around it. Then the reunion $\rmdod\displaystyle \bigcup_{\Cl\in\crmdod\setminus\rmdod}N_{\Cl}$ is an open cover of $\crmdod$. Since $\crmdod$ is compact, we can find $\Cl_1,\ldots,\Cl_N\in \crmdod$ such that $\rmdod\displaystyle \bigcup_{i=1}^N N_{\Cl_i}$ is still a cover of $\crmdod$. It follows from Stokes's Theorem that the difference $\chi^{\ssp}_d(X,\Omega,c_X;\gamma) - \chi^{\ssp}_d(X,\Omega,c_X;\gamma')$ is equal to $\displaystyle\sum_{i=1}^N m_{\ssp}(N_{\Cl_i})$. Thus, to prove Theorem \ref{second}, it is enough to prove that all the reducible curves in $\crmdod$ admit a black box with vanishing contribution.

\subsubsection{$\P$-thickened curves}\label{parthick}

Consider a reducible curve $\Cl = [\Sigma,j_{\Sigma},c_{\Sigma},u_0,t_0]\in\crmdod$ where $(\Sigma,j_{\Sigma},c_{\Sigma})$ is a nodal rational curve equipped with an anti-holomorphic involution $c_{\Sigma}$, and $u_0 : \Sigma \rightarrow X$ is a stable $\delta(t_0)$-holomorphic map such that $u_0\circ c_{\Sigma} = c_X\circ u_0$. To evaluate the contribution of a black box around $\Cl$, the idea is to use a gluing theorem to describe a neighborhood of $\Cl$ in $\crmdod$. Unfortunately, we cannot say much about the curve $\Cl$ except that if $\delta$ is generic enough it has only two component in its image. In particular, the curve $\Sigma$ could have many components, some of them on which $u_0$ could be multiple or constant. This implies that we cannot hope to have a nice description of a neighborhood of $\Cl$ in $\crmdod$.

Instead, we will consider a larger class of curves in which we will smoothen the nodes of $\Cl$. To this end, choose $n = 2s + l$ points $p_1,\ldots,p_n$ on $\Sigma$ away from the nodes and such that
\begin{itemize}
\item the marked curve $(\Sigma,j,p_1,\ldots,p_n)$ is stable,
\item for all $i = 1,\ldots,s$, $c_{\Sigma}(p_i) = p_{s+i}$,
\item for all $i = 2s +1,\ldots,2s +l$, $c_{\Sigma}(p_i) = p_{i}$,
\item for all $i = 1,\ldots,n$, $u_0$ is an immersion at $p_i$,
\item for all $i = 1,\ldots,n$ and $j = 1,\ldots,n$, $u_0(p_i)\neq u_0(p_j)$ if $i\neq j$.
\end{itemize}

\begin{remark}
  The curve $\Sigma$ is the reunion of $m +1$ spheres $\Sigma_1,\ldots,\Sigma_{m+1}$ which can be of three different types
  \begin{enumerate}
  \item $c_{\Sigma}(\Sigma_i) = \Sigma_i$ and $Fix(c_{\Sigma})\cap \Sigma_i \neq \emptyset$,
  \item\label{type} $c_{\Sigma}(\Sigma_i) = \Sigma_i$ and $Fix(c_{\Sigma})\cap \Sigma_i = \emptyset$,
  \item $c_{\Sigma}(\Sigma_i) = \Sigma_j$ for $i\neq j$.
  \end{enumerate}

The minimum number of points to add on $\Sigma$ so that it becomes stable is $m + 3$. However, we cannot always choose them to satisfy the second and third conditions. For example, if one the the components of $\Sigma$ is of type \eqref{type} and is attached to exactly two other components, then we need only add one point to it to get the stability, but this point would then have to be real in order to satisfy the third condition. All in all, we only have the inequality $n\geq m+3$.
\end{remark}

The stable real curve $(\Sigma,j_{\Sigma},c_{\Sigma},p_1,\ldots,p_n)$ has no automorphism and gives a point in the moduli space of $n$-marked real rational curves denoted by $\rmok$ as defined by Ceyhan, Definition 3.2 in \cite{ceyhan}. This space is the real part of the moduli space $\mok$ of $n$-marked rational curves equipped with an appropriate real structure $c_{2s,l}$ (see Theorem 3.4 in \cite{ceyhan}). The space $(\mok,c_{2s,l})$ is a smooth projective manifold of dimension $n-3$ and admits a universal curve $\Uk\rightarrow \mok$ which is equipped with a real structure $c_{\Uk}$ lifting $c_{2s,l}$ and such that its restriction to the fiber over $[S] \in \rmok$ gives a real structure turning the fiber into a real stable curve whose isomorphism class is $[S]$ (see Theorem 3.4 in \cite{ceyhan}).

Choose $n$ disjoint smooth open and orientable submanifolds $H_1,\ldots,H_n$ of codimension $2$ such that
\begin{itemize}
\item for all $i = 1\ldots,n$, $u_0(p_i)\in H_i$ and $u_0$ is transverse to $H_i$ at $p_i$,
\item for all $i = 1,\ldots,s$, $c_X(H_i) = H_{s+i}$ and for all $i = 2s+1,\ldots,2s+l$, $c_X(H_i) = H_{i}$.
\end{itemize}

Now consider the complex vector bundle $TX_{\delta}$ over $X\times \cyl$ which is the pullback of $TX$ with the complex structure given by $\delta : v \in T_{x,t} X_{\delta} \mapsto \delta(x)v \in T_{x,t}X_{\delta}$ for all $x\in X$ and $t\in \cyl$. There is also a complex line bundle $\Lambda^{0,1} T_{\rmok}\Uk$ over $\Uk$ such that its restriction to a fiber of the universal curve is the bundle of the $(0,1)$-forms on the fiber. Thus, we obtain a complex vector bundle $\Lambda^{0,1}T_{\rmok}\Uk \otimes TX_{\delta}$ over $\Uk\times X\times \cyl$. Let us denote by $\pert = \Cl^l(\Lambda^{0,1}T_{\rmok}\Uk \otimes TX_{\delta})_{+1}$ the space of section $e$ of this bundle of class $\Cl^{l}$ such that $e(c_{\Uk}(.),c_X(.),.) = d c_X \circ e(.,.,.)\circ d c_{\Uk}$.

Using the terminology of Pardon in \cite{pardon} (Definition 9.2.3), we define the following.
\begin{definition}\label{thick}
  Let $\P\subset \pert$ be a finite dimensional vector space such that all its elements are supported away from the nodes, i.e. if $Sing(\Uk)\subset\Uk$ is the set of nodes of the curves in the universal family, then there exists a neighborhood of $Sing(\Uk)\times X\times \cyl$ on which all the elements of $\P$ vanish. A $\P$-thickened pseudo-holomorphic curve of degree $d$ is a quadruple $(\PS,u,t,e)$ where
  \begin{enumerate}
  \item $\PS = (S,j_S,c_S,\up)$ is a stable $n$-marked real rational curve, $t\in \cyl$ and $e \in \P$,
  \item $u : S \rightarrow X$ is a smooth map with $u\circ c_S = c_X \circ u$ and $u_*[S] = d$,
  \item for all $i = 1,\ldots,n$, $u(p_i)\in H_i$ and $u$ is transverse to $H_i$ at $p_i$,
  \item for all $z \in S$, $d_z u + \delta(t)\circ d_z u \circ j_S = e(\phi(z) , u(z), t)\circ d_z\phi$, where $\phi : \PS\rightarrow (\Uk)_{|[\PS]}$ is the unique isomorphism between $\PS$ and the corresponding fiber of the universal curve.
  \end{enumerate}
We say that two $\P$-thickened pseudo-holomorphic curves $((S,j_S,c_{S},\up),u,t,e)$ and $((S',j_{S'},c_{S'},\up'),u',t',e')$ are equivalent if $t = t'$, $e = e'$ and there exists an isomorphism $\phi : (S,j_S,c_{S},\up) \rightarrow (S',j_{S'},c_{S'},\up')$ such that $u = u'\circ \phi$.
\end{definition}

Let $\P\subset \pert$ be as in Definition \ref{thick}. Let $\crmdp$ be the set of equivalence classes of $\P$-thickened pseudoholomorphic curves of degree $d$. It is endowed with the Gromov topology, for which the map $\pi_{\P} : \crmdp \rightarrow \cyl\times \P$ is continuous and proper. Note also that the curve $\Cl$ is in $\crmdp$. 

We consider also $\rmdp\subset \crmdp$ the subspace consisting of irreducible curves. We can describe it in the following way. 

First let us recall the construction of the Teichmüller spaces (following \cite{earle}). Take $S$ to be a smooth oriented $2$-sphere and fix $\uz = (z_1,\ldots,z_n)$ a $n$-tuple of disctinct points on $S$. Let $Diff(S;\uz)$ be the group of diffeomorphisms $\phi$ of $S$ of class $\Cl^{l+1}$ fixing all points of $\uz$ if $\phi$ preserves the orientation, inducing the permutation $\bigl(\begin{smallmatrix} 1 & \ldots & s & s+1 & \ldots & 2s & 2s+1 & \ldots & 2s + l \\ s+1 & \ldots & 2s & 1 & \ldots & s & 2s+1 & \ldots & 2s + l\end{smallmatrix}\bigr)$ on $\uz$ otherwise. Denote by $Diff_0(S;\uz)$ (resp. $Diff^+(S;\uz)$) the identity component of $Diff(S;\uz)$ (resp. the subgroup of $Diff(S;\uz)$ consisting of orientation preserving diffeomorphisms). The quotient $\TT_n = \js/Diff_0(S;\uz)$ is the Teichm\"uller space of $(S,\uz)$. It is a complex manifold of dimension $n-3$. The groups $\Gamma_{2s,l} = Diff(S;\uz)/Diff_0(S;\uz)$ and $\Gamma^+_n = Diff^+(S;\uz)/Diff_0(S;\uz)$ act on it. In fact, the elements of $\Gamma^+_{n}$ act as holomorphic automorphisms of $\TT_n$ and those of $\Gamma_{2s,l}\setminus\Gamma_n^+$ as anti-holomorphic automorphisms. There exists a holomorphic structure on $S\times \TT_n$ turning this product in the universal family over $\TT_n$. It is not unique, but we fix one once and for all. It induces a section $J_S : \TT_n\rightarrow \js$. Moreover, the group $\Gamma_{2s,l}$ acts on $S\times \TT_n$ as a group of holomorphic and anti-holomorphic automorphisms preserving the projection on $\TT_n$ (see \cite{earle}).

 On the other hand, let $\CS^d(\uh) = \{u\in L^{k,p}(S,X)\, |\, u_*[S] = d,\, \forall i = 1,\ldots ,n,\, u(z_i) \in H_i,\, u\pitchfork H_i\}$. Then we can define a Banach bundle $\E'$ over $\CS^d(\uh)\times \TT_n\times \cyl\times \P$ whose fiber over an element $(u,\tau,t,e)$ is $L^{k-1,p}(S,\Lambda^{0,1}S\otimes u^*TX)$ where the complex structures involved are given by $J_S(\tau)$ and $\delta(t)$, and a section $\Phi_{\DB}$ of $\E'$ given by $\Phi_{\DB}(u,\tau,t,e) = d u + \delta(t)\circ d u \circ J_S(\tau) - e(.,u,t)$. The last term in the previous operator is in fact given by $e(\phi(.),u(.),t)\circ d\phi$, where $\phi$ is the unique isomorphism between $(S,J_S(\tau),\uz)$ and the corresponding fiber of the universal curve; to simplify the notation, we forget about $\phi$ and consider $e$ directly as a section of the adequate bundle over the product of the universal curve over $\TT_n$ and $X\times\cyl$. 

The group $\Gamma_{2s,l}$ acts on the zero set $\hmdp$ of $\Phi_{\DB}$. The only elements of $\Gamma_{2s,l}$ having fixed points are orientation reversing involutions of $S$, and an element of $\hmdp$ can only be fixed by at most one element of $\hmdp$. For each orientation reversing involution $c_S\in\Gamma_{2s,l}$, let us write $\rchmdp$ the set of its fixed points. We denote by $\rhmdp$ the reunion of all the components $\rchmdp$. The group $\Gamma^+_n$ then acts freely on $\rhmdp$ and the quotient is $\rmdp$. 

Let us now discuss the smooth structure on $\rmdp$ (in the spirit of \cite{shev}). To this end, let us fix a riemannian metric $g_X$ on $X$ which is invariant under $c_X$ and denote by $\nabla$ its Levi-Civita connection. In order to simplify the notations, we will use $\nabla$ to denote all the induced connections on the various bundles. Take an orientation reversing involution $c_S\in\Gamma_{2s,l}$. Note that since $c_S$ is an anti-holomorphic automorphism of $S\times \TT_n$ preserving the projection on $\TT_n$, for each $\tau\in\TT_n$, $c_S$ induces a diffeomorphism of $S$ which we denote the same way. The involution acts on $\CS^d(\uh)\times\TT_n$ and we denote by $\R_{c_S}\CS^d(\uh)\times\R_{c_S}\TT_n$ its fixed points. This action lifts to the bundle $\E'$. The section $\Phi_{\DB}$ restricts to $\R_{c_S}\CS^d(\uh)\times\R_{c_S}\TT_n\times\cyl\times \P$ as a section of the Banach sub-bundle $\E'_{+1}$ of $\E'$ consisting of the $+1$ eigenspace of the involution induced by $c_S$ and $c_X$. The linearization $\tD_u$ of $\Phi_{\DB}$ at $(u,\tau,t,e)$ is given by
\[
\tD_u(v,\dtau,\dt,\de) = \Dn_u(v) + \delta(t)\circ d u \circ d_{\tau} J_S(\dtau) + d_t\delta(\dt) \circ d u \circ J_S(\tau) - \de - \nabla_{(\dtau,0,\dt)} e,
\]
for $(v,\dtau,\dt,\de)\in T_{(u,\tau,t,e)}\R_{c_S}\CS^d(\uh)\times\R_{c_S}\TT_n\times\cyl\times \P$, where $\Dn_u$ is the operator given by
\[
\Dn_u(v) = \nabla v + \delta(t)\circ \nabla v \circ J_S(\tau) + \nabla_v \delta(t) \circ d u \circ J_S(\tau) - \nabla_v e.
\]
Let us also denote by $\vD_{u}$ the restiction of $\tD_u$ to $T_{(u,e)}\R_{c_S}\CS^d(\uh)\times\P$.

\begin{lemma}\label{fredholm}
  The linearization of $\Phi_{\DB}$ at a point $(u,\tau,t,e)\in \rhmdp$ is a Fredholm operator of index $\dim(\P) +1$.
\end{lemma}

\begin{proof}
 Denote by $c_S$ the element of $\Gamma_{2s,l}$ fixing $(u,\tau,t,e)$. The operator $\Dn_u$ is the restriction of a real generalized Cauchy-Riemann operator on $u^*TX$ to the space $L^{k,p}(S,u^*TX;\uh)_{+1}$ consisting of sections $v$ of $TX$ such that for all $i$ from $1$ to $n$, $v_i\in T_{u(z_i)}H_i$ and $d c_X(v) = v\circ c_S$. In particular, it is Fredholm of index $2 - n$ given by Riemann-Roch theorem (see Lemma 1.6 in \cite{wel1}). Since the space $T_{(\tau,t,e)}\R_{c_S}\TT_n\times\cyl\times \P$ is finite dimensional, the operator $\tD_{u_0}$ is also Fredholm. Its index is given by $2-n + \dim(\R_{c_S}\TT_n) + \dim(\cyl) + \dim (\P) = 2-n + n-3 +2 + \dim(\P) = \dim(\P) + 1$.
\end{proof}

Coming back to our original reducible curve $\Cl = [\Sigma,j_{\Sigma},c_{\Sigma},u_0,t_0]\in\crmdod$, let $L^{k,p}(\Sigma,u_0^*TX;\uh)_{+1}$ be the subspace of $L^{k,p}(\Sigma,u_0^*TX)$ consisting of the sections $v$ of $u_0^*TX$ such that for all $i$ from $1$ to $n$, $v_i \in T_{u_0(p_i)}H_i$, and $d c_X(v) = v\circ c_{\Sigma}$. We can also define operators $\vD_{u_0}$ (resp. $\Dn_{u_0}$) from $L^{k,p}(\Sigma,u_0^*TX;\uh)_{+1} \times T_{t_0}\cyl \times T_0 \P$ (resp. from $L^{k,p}(\Sigma,u_0^*TX;\uh)_{+1}$) to  $L^{k-1,p}(\Sigma,\Lambda^{0,1}\Sigma\otimes u_0^* TX)_{+1}$ by
\begin{gather*}
\vD_{u_0}(v,\de) = \Dn_{u_0}(v) - \de.
\intertext{and}
\Dn_{u_0}(v) = \nabla v + \delta(t_0)\circ \nabla v \circ j_{\Sigma} + \nabla_v \delta(t_0) \circ d u_0 \circ j_{\Sigma},\\
\end{gather*}
Moreover, since the operator $\Dn_{u_0}$ is Fredholm, so is $\vD_{u_0}$.

\begin{definition}
  A finite dimensional vector space $\P\subset\pert$ is said to be $\Cl$-regularizing if the operator $\vD_{u_0}$ is surjective.
\end{definition}

Taking such a regularizing space $\P$ ensures that all the $\P$-thickened pseudo-holomorphic curves in a neighborhood of $\Cl$ are regular, i.e. that the section $\Phi_{\DB}$ vanishes transversally on a neighborhood of $\Cl$. More precisely, recall the following result (see for example Theorem B.1.1.i in \cite{pardon}).

\begin{prop}\label{surj}
  Let $\P\subset \pert$ be a $\Cl$-regularizing space. Then there exists an open neighborhood $V_{\P}$ of $\Cl$ in $\crmdp$ such that for any $\P$-thickened irreducible curve $(u,\tau,t,e)$ appearing in this neighborhood, the operator $\vD_{u}$ is surjective. In particular, $W_{\P} = V_{\P}\cap \rmdp$ is a Banach manifold of class $\Cl^{l-k}$ and of dimension $\dim(\P) +1$. Moreover, the map $\pi_{\P} : W_{\P} \rightarrow \cyl\times \P$ is smooth and transverse to $\cyl\times\{0\}$.\qed
\end{prop}

From now on, $\P$ will be a $\Cl$-regularizing space, $V_{\P}$ will be a neighborhood of $\Cl$ in $\crmdp$ as given by Proposition \ref{surj}, and $W_{\P}$ the intersection $V_{\P}\cap \rmdp$.

From Proposition \ref{surj}, we deduce that the fiber product $W_0 = W_{\P} \fp{\pi_{\P}}{\iota}\cyl$ is a smooth and oriented manifold of dimension $1$ included in $V_0= V_{\P} \fp{\pi_{\P}}{\iota}\cyl$. Note that both consist of honest pseudo-holomorphic curves. Moreover, up to shrinking $V_{\P}$, we can assume as in Lemma \ref{neig} that $V_{0}$ is a black box around $\Cl$. Thus, provided we can orient the map $\pi_{\P}$, which we will show in Proposition \ref{porient}, we can define a new contribution $m_{\ssp,\P}(V_0)$ of a black box around the $\P$-thickened curve $\Cl$ as the signed count of the elements in the boundary of $V_0$.

Let $\psi : V_{0} \rightarrow \crmdod$ be the forgetful map which consists in forgetting the marked points on the curves. Then its image is a black box around $\Cl$ so we can compare the integers $m_{\ssp,\P}(V_0)$ and $m_{\ssp}(\psi(V_0))$.

\begin{prop}\label{porient}
 The map $\pi_{\P} : W_{\P} \rightarrow \cyl\times \P$ is orientable, and is oriented by the choice of an oriented real $Spin$ structure on $(TX, dc_X)$.

In particular, given an oriented real $Spin$ structure on $(TX,d c_X)$, both $\rmdod$ and $W_0 = V_0\setminus\{\Cl\}$ are smooth and oriented $1$-dimensional manifolds. Moreover, the forgetful map $\psi : W_0\rightarrow \rmdod$ is orientation preserving.
\end{prop}

\begin{proof}
 We proceed as in the proof of Theorem \ref{first}. Let us denote by $\widehat{W}_{\P}\subset \rhmdp$ the preimage of $W_{\P}$ by the projection $\rhmdp\rightarrow \rhmdp/\Gamma^+_n = \rmdp$, and by $\widehat{\pi}_{\P} : \widehat{W}_{\P}\rightarrow\cyl\times\P$ the projection. The determinant of $\pi_{\P}$ is the the quotient of the determinant of $\widehat{\pi}_{\P}$ by the action of $\Gamma^+_n$. To prove the first part of the proposition, we show that $\ddet{\widehat{\pi}_{\P}}$ is orientable and oriented by the choice of an oriented real $Spin$ structure on $(TX, dc_X)$, and that the action of $\Gamma^+_n$ on $\ddet(\widehat{\pi}_{\P})$ is orientation preserving.

Take a $\P$-thickened curve $(u,\tau,t,e)\in\widehat{W}_{\P}$ and denote by $c_S$ the element of $\Gamma_{2s,l}$ fixing it. Let $\bD_u$ be the restriction of $\tD_u$ to $T_{(u,\tau)}\left(\R_{c_S}\CS^d(\uh)\times\R_{c_S}\TT_n\right)\times\{0\}\times \{0\}$. We then have canonical isomorphisms $\ker(d_{(u,\tau,t,e)}\widehat{\pi}_{\P}) = \ker(\bD_u)$ and $\coker(d_{(u,\tau,t,e)}\widehat{\pi}_{\P}) = \coker(\bD_u)$. The second one holds because $\tD_u$ is surjective (see Corollary 2.2.3 in \cite{shev}). Thus, we need to orient the determinant of $\bD_u$.

To this end, let us consider the following exact sequence
\begin{equation}
0\rightarrow \ker(\Dn_u) \rightarrow \ker(\bD_u)\rightarrow T_{\tau}\R_{c_S}\TT_n\rightarrow \coker(\Dn_u)\rightarrow \coker(\bD_u)\rightarrow 0. \label{exa}
\end{equation}
The first arrow comes from the inclusion of $T_{u}\R_{c_S}\CS^d(\uh)$ in $T_{(u,\tau)}\left(\R_{c_S}\CS^d(\uh)\times\R_{c_S}\TT_n\right)$; the second one is the restriction of the projection $T_{(u,\tau)}\left(\R_{c_S}\CS^d(\uh)\times\R_{c_S}\TT_n\right) \rightarrow T_{\tau}\R_{c_S}\TT_n$; the third one is given by the restriction of $\bD_u$ to $T_{\tau}\R_{c_S}\TT_n$; the last one is the projection. This sequence gives a canonical isomorphism $\ddet(\bD_u) = \Lambda_{\R}^{\max}(T_{\tau}\R_{c_S}\TT_n)^* \otimes \ddet(\Dn_u)$. We will now orient the two right-hand terms in the previous isomorphism.

Let us begin with $\ddet(\Dn_u)$. The operator $\Dn_u$ fits in the following commutating diagram
\begin{equation}\label{dn}
\xymatrix{0 \ar[r] & L^{k,p}(S,u^*TX;\uh)_{+1} \ar[d]^{\Dn_u} \ar[r] & L^{k,p}(S,u^*TX)_{+1} \ar[d]^{D_u} \ar[r]^{ev_{\uz}} & \R \left(T_{u(\uz)}X / T_{u(\uz)}\uh \right) \ar[d]^0 \ar[r] & 0 \\
0 \ar[r] & L^{k-1,p}(S,\Lambda^{0,1}S\otimes u^*TX)_{+1} \ar[r] & L^{k-1,p}(S,\Lambda^{0,1}S\otimes u^*TX)_{+1} \ar[r] & 0 \ar[r] & 0.
}
\end{equation}
The complex vector space $T_{u(\uz)}X / T_{u(\uz)}\uh$ is the direct sum $\oplus_{i}T_{u(z_i)}X / T_{u(z_i)}\uh$. The real structure $d c_X$ induces a $\C$-antilinear involution on it and $\R\left(T_{u(\uz)}X / T_{u(\uz)}\uh\right)$ denotes the fixed points set of this involution. On the other hand, the middle column of this diagram is a real generalized Cauchy-Riemann operator on $(u^*TX,d c_X)$. It can be seen as in Proposition \ref{detd} that its determinant is canonically oriented by the choice of an oriented real $Spin$ structure on $(TX,dc_X)$ and that the action of $\Gamma^+_n$ preserves this orientation. Thus the choice of an oriented real $Spin$ structure on $(TX,dc_X)$ canonically orients $\ddet(\Dn_u)\otimes \Lambda^{\max}_{\R}\R \left(T_{u(\uz)}X / T_{u(\uz)}\uh \right)$.

Let us now consider $\Lambda_{\R}^{\max}(T_{\tau}\R_{c_S}\TT_n)^*$. Denote by $\DB_{\tau} : L^{k,p}(S,TS)_{+1}\rightarrow L^{k-1,p}(S,\Lambda^{0,1}S\otimes TS)_{+1}$ the operator given by the holomorphic stucture induced by $J_S(\tau)$ on $TS$ and denote by $\DB_{\tau,\uz}$ its restriction to the space $L^{k,p}(S,TS;\uz)_{+1}$ consisting of real sections of $TS$ vanishing at all the $z_i$. The operator $\DB_{\tau,\uz}$ has trivial kernel and its cokernel is naturally isomorphic to the space $T_{\tau}\R_{c_S}\TT_n$ (see Lemma 1.8 in \cite{article}). In fact, the isomorphism is given by $\dtau\in T_{\tau}\R_{c_S}\TT_n\mapsto -J_S(\tau) d_{\tau}J_S(\dtau)\in \coker(\DB_{\tau,\uz})$. Thus, the line $\Lambda_{\R}^{\max}(T_{\tau}\R_{c_S}\TT_n)^*$ is naturally isomorphic to the determinant of $\DB_{\tau,\uz}$. On the other hand, those operators fit in the following commutating diagram
\begin{equation}\label{dtau}
\xymatrix{0 \ar[r] & L^{k,p}(S,TS;\uz)_{+1} \ar[d]^{\DB_{\tau,\uz}} \ar[r] & L^{k,p}(S,TS)_{+1} \ar[d]^{\DB_{\tau}} \ar[r]^{ev_{\uz}} & \R T_{\uz}S \ar[d]^0 \ar[r] & 0 \\
0 \ar[r] & L^{k-1,p}(S,\Lambda^{0,1}S\otimes TS)_{+1} \ar[r] & L^{k-1,p}(S,\Lambda^{0,1}S\otimes TS)_{+1} \ar[r] & 0 \ar[r] & 0.
}
\end{equation}
The determinant of $\DB_{\tau}$ is canonically oriented as in Proposition \ref{dbj} and the action of $\Gamma_n^+$ preserves this orientation. Thus, the tensor product $\Lambda_{\R}^{\max}(T_{\tau}\R_{c_S}\TT_n)^*\otimes \Lambda^{\max}_{\R}\R T_{\uz}S$ is canonically oriented. Moreover, there is an isomorphism $\Lambda^{\max}_{\R}\R T_{\uz}S = \Lambda^{\max}_{\R}\R \left(T_{u(\uz)}X / T_{u(\uz)}\uh \right)$ given by $d u$.

To sum up, the choice of an oriented real $Spin$ structure on $(TX,dc_X)$ canonically orients $\ddet(\widehat{\pi}_{\P})_u = \ddet(\bD_u) = \ddet(\Dn_u)\otimes \Lambda_{\R}^{\max}(T_{\tau}\R_{c_S}\TT_n)^*$, and this orientation is preserved by the action of $\Gamma_n^+$, which proves the first part of the proposition.

We now turn our attention to the forgetful map. Fix a curve $(u,\tau,t)\in W_0$ and choose a diffeomorphism $\varphi\in Diff^+(S)$ such that $(\varphi^{-1})^*J_S(\tau) = j_0$, where $j_0$ is the fixed complex structure on $S$ we used in \S \ref{prooffirst}. Then the curve $(u\circ \varphi^{-1},t)$ gives an element of $\rmdod$, which does not depend on the choice of $\varphi$ and is in fact $\psi(u,\tau,t)$. The differential of $\psi$ at the point $(u,\tau,t)$ sits in the following commutating diagram
\begin{equation}\label{diag}
\xymatrix{0 \ar[r] & \ker(d_{(u,\tau,t)} \pi_{\P}) \ar[r] \ar[d]^{\alpha} & T_{(u,\tau,t)} W_0 \ar[r] \ar[d]^{d \psi} & T_t \cyl \ar[r] \ar[d]^{id} & \coker(d_{(u,\tau,t)} \pi_{\P}) \ar[r] \ar[d]^{\beta} & 0 \\
0 \ar[r] & \ker(d_{(u\circ \varphi^{-1},t)} \pi) \ar[r] & T_{\psi(u,\tau,t)} \rmdod \ar[r]  & T_t \cyl \ar[r] & \coker(d_{(u\circ\varphi^{-1},t)} \pi) \ar[r] & 0.
}
\end{equation}

Since we used the first and second lines of this diagram to orient respectively $W_0$ and $\rmdod$, it is enough to check that the isomorphism between $\ddet(d_{(u,\tau,t)} \pi_{\P})$ and $\ddet(d_{(u\circ \varphi^{-1},t)} \pi)$ induced by $\alpha$ and $\beta$ is orientation preserving.

 As we have seen at the beginning of the proof, we have canonical isomorphism $\ker(d_{(u,\tau,t)} \pi_{\P}) = \ker(\bD_u)$ and $\coker(d_{(u,\tau,t)} \pi_{\P}) = \coker(\bD_u)$. We have also seen in \S \ref{prooffirst} that we have canonical isomorphisms $\ker(d_{(u\circ \varphi^{-1},t)} \pi) = \ker(\OB_{u\circ \varphi^{-1}})$ and $\coker(d_{(u\circ\varphi^{-1},t)} \pi) = \coker(\OB_{u\circ \varphi^{-1}})$. Using these isomorphisms, the first and fourth vertical arrows in (\ref{diag}) are given by

\begin{multline*}
\alpha :(v,\dtau) \in \ker(d_{(u,\tau,t)} \pi_{\P})\subset L^{k,p}(S,u^*TX;\uh)_{+1}\times T_{\tau}\R_{c_S}\TT_n \\ \mapsto [v\circ \varphi^{-1}]\in \ker(d_{(u\circ\varphi^{-1},t)} \pi)\subset L^{k,p}(S,(u\circ\varphi^{-1})^*TX)_{+1}/L^{k,p}(S,TS)_{+1}
\end{multline*}

and

\begin{multline*}
  \beta : [w]\in \coker(d_{(u,\tau,t)} \pi_{\P}) = L^{k-1,p}(S,\Lambda^{0,1}S\otimes u^*TX)_{+1}/ \im(\bD_u) \\ \mapsto [w\circ d\varphi^{-1}] \in \coker(d_{(u\circ\varphi^{-1},t)} \pi),
\end{multline*}

 where we recall that 

\[
\coker(d_{(u\circ\varphi^{-1},t)} \pi) = \left(L^{k-1,p}(S,\Lambda^{0,1}S \otimes(u\circ\varphi^{-1})^*TX)_{+1}/ L^{k-1,p}(S,\Lambda^{0,1}S\otimes TS)_{+1} \right) / \im(\OB_{u\circ\varphi^{-1}}).
\]

 Moreover, $\alpha$ and $\beta$ sit in the following commutating diagram
\begin{equation}\label{mast}
\xymatrix{
&0 \ar[r] & 0 \ar[r] \ar[d] & \ker(\Dn_u) \ar[r]\ar[d] & \ker(\bD_u) \ar[d]^{\alpha}\ar@{-}[r]& \\
&0 \ar[r] & \ker(\DB_{j_0}) \ar[r] & \ker(D_{u\circ\varphi^{-1}}) \ar[r] & \ker(\OB_{u\circ\varphi^{-1}}) \ar@{-}[r]& \\
\ar[r]&T_{\tau}\R_{c_S}\TT_n \ar[r] \ar[d] & \coker(\Dn_u) \ar[r] \ar[d] & \coker(\bD_u) \ar[r] \ar[d]^{\beta} & 0& \\
\ar[r]&\coker(\DB_{j_0}) \ar[r] & \coker(D_{u\circ\varphi^{-1}}) \ar[r] & \coker(\OB_{u\circ\varphi^{-1}}) \ar[r] & 0.&
}
\end{equation}
The first line is the exact sequence (\ref{exa}), the second one comes from the diagram (\ref{diagg}) in \S \ref{prooffirst} (see also Corollary 1.5.4 in \cite{shev}) which we used to orient $\ddet(\OB_{u\circ\varphi^{-1}})$. Note that the fourth vertical arrow is given by $\dtau\in T_{\tau}\R_{c_S}\TT_n\mapsto d\varphi\circ (-J_S(\tau) d_{\tau}J_S(\dtau))\circ d\varphi^{-1}\in\coker(\DB_{j_0})$. Using the diagrams (\ref{dn}) and (\ref{dtau}), we deduce from (\ref{mast}) a commutative square
\[
\xymatrix{
\ddet(\bD_u) \ar[r]^{\alpha,\beta}\ar[d] & \ddet(\OB_{u\circ\varphi^{-1}}) \ar[d]\\
\ddet(\Dn_u)\otimes \Lambda_{\R}^{\max}(T_{\tau}\R_{c_S}\TT_n)^* \ar[r] & \ddet(D_{u\circ\varphi^{-1}})\otimes \ddet(\DB_{j_0})^*.
}
\]
Seeing that to orient $\ddet(\OB_{u\circ\varphi^{-1}})$, we oriented the tensor product $\ddet(D_{u\circ\varphi^{-1}})\otimes \ddet(\DB_{j_0})^*$ and then used
 the rightmost arrow (see \S \ref{prooffirst}), and that to orient $\ddet(\bD_u)$ we used the leftmost arrow and the bottom arrow, we conclude that the top arrow, induced by $\alpha$ and $\beta$, preserves the orientations. Thus, as mentionned previously, we deduce from the diagram (\ref{diag}) that the forgetful map is orientation preserving.
\end{proof}

The forgetful map can be described in a different way as follows. Let $\aut(\Cl)$ be the automorphism group of the curve $\Cl$. Up to shrinking $V_{\P}$, the group $\aut(\Cl)$ acts on $V_0$ as is explained in Lemma 3.1 in \cite{litian}. In our case, this action is free on $W_0$ as all the curves appearing there are simple, and the forgetful map is the quotient map $V_0 \rightarrow V_0/\aut(\Cl)$.

In particular, the forgetful map is $\#(\aut(\Cl))$-to-one on $W_0$, so combined with Proposition \ref{porient}, we obtain that $m_{\ssp}(\psi(V_0)) = \frac{m_{\ssp,\P}(V_0)}{\#(\aut(\Cl))}$. In particular, if $m_{\ssp,\P}(V_0)$ vanishes, so does $m_{\ssp}(\psi(V_0))$.

\subsection{Gluing theorem for $\P$-thickened curves - Proof of Theorem \ref{second}}

In this section, we prove the a gluing theorem for the $\P$-thickened curves. We use the same notations as in the previous section. We will moreover assume that all the auxiliary marked points $p_i$ are not real, i.e that $l = 0$. This hypothese is not essential and can be forgotten in the \S\S \ref{pargl} and \ref{parlin} but simplifies some of the arguments that appear in \S \ref{orsection}.

\begin{theo}\label{gluing}
Let $\P\subset \pert$ be a $\Cl$-regularizing space and $V_{\P}\subset \crmdp$ a neighborhood of $\Cl$ given by Proposition \ref{surj}. Then there exists a neighborhood $U_{\P}\subset V_{\P}$ of $\Cl$, an open set $\mathcal{U}\subset \R^{\dim \P +1}$ and an homeomorphism $\gl : \mathcal{U} \rightarrow U_{\P}$ such that for any choice of an orientation on $U_{\P}\cap \rmdp$ given by the Proposition \ref{porient}, there exists an orientation of $\mathcal{U}$ such that the homeomorphism $\gl$ preserves the orientations.
\end{theo}

Let us first explain how to conclude the proof Theorem \ref{second}.

\begin{proof}[Proof of Theorem \ref{second}]
For each reducible curve $\Cl\in\crmdod$ we can apply Theorem \ref{gluing} and take $U_{\P}$ such that $U_{0} = \pi_{\P}^{-1}(\cyl)$ is a black box around the $\P$-thickened curve $\Cl$. Moreover, we can see $U_0$ as the zero set of the natural section $\sigma_{\P}$ of the trivial bundle $U_{\P}\times \P$ over $U_{\P}$ induced by the projection $\pi_{\P}$. 

Let $W_0 = U_0\cap\rmdod$. Then $\gl^{-1}(W_0)\subset \mathcal{U}$ is a $1$-dimensional topological submanifold of $\mathcal{U}$ which can also be seen as the zero set of the section $\sigma_{\P}\circ \gl$ of the bundle $\mathcal{U}\times \P$ over $\P$. Moreover, if we fix an oriented real $Spin$ structure on $(TX, d c_X)$ and an orientation of $\P$, this submanifold becomes oriented in, a priori, two different ways : as the preimage of $W_0$ and as the zero set of $\sigma_{\P}\circ \gl$. However, the Theorem \ref{gluing} implies that those two orientations coincide. In particular, counting the elements of the boundary of $\gl^{-1}(U_0)$ using the orientation coming from the section $\sigma_{\P}\circ \gl$ gives $m_{\ssp,\P}(U_0)$.

 Using for example a result of Kirby and Siebenman (\cite{kirbysieb}, Theorem 10), we can perturb the section $\sigma_{\P}\circ \gl$, keeping it unchanged near the boundary of $\gl^{-1}(U_0)$, to get a new section whose zero set is now a compact oriented $1$-dimensional submanifold of $\mathcal{U}$ which has the same boundary as $\gl^{-1}(U_0)$. In particular, the signed count of the elements of this boundary set is zero, i.e. $m_{\ssp,\P}(U_0) = 0 = m_{\ssp}(\psi(U_0))$.

In particular, every reducible curve admits a black box with vanishing contribution. As we explain at the end of \S \ref{blackb} this proves Theorem \ref{second}.
\end{proof}

Some parts of the proof of Theorem \ref{gluing} are well-known so we will not give the full proof of it. We will rather give references to proofs of those statements when we found them and give a proof when we did not.

\subsubsection{The gluing map}\label{pargl}

We will break down Theorem \ref{gluing} into two parts. The first one concerns the existence of the homeomorphism $\gl$ which we explain here. The second one concerns the statement about the orientations, which we will treat in the \S \ref{orsection}.

\begin{theo}\label{homeo}
  Let $\P\subset \pert$ be a $\Cl$-regularizing space and $V_{\P}\subset \crmdp$ a neighborhood of $\Cl$ given by Proposition \ref{surj}. Then there exists a neighborhood $U_{\P}\subset V_{\P}$ of $\Cl$, an open set $\mathcal{U}\subset \rmok\times \cyl\times \ker(\vD_{u_0})$ and an homeomorphism $\gl : \mathcal{U} \rightarrow U_{\P}$ such that the diagram
\begin{equation}\label{triangle}
\xymatrix{\mathcal{U} \ar[rr]^{\gl} \ar[rd]_{\pi_{\mathcal{U}}} & & U_{\P} \ar[ld]^{\pi_{\rmok\times\cyl}} \\
&\rmok \times \cyl&
}
\end{equation}
commutes, where $\pi_{\rmok\times \cyl}$ is the natural map obtained by keeping only the source of the curve and the almost-complex structure on $X$.
\end{theo}

We will not give a detailed proof of this theorem, but we will rather refer the reader to \cite{MDS} Chapter 10, and \cite{pardon} Appendix B. We will however recall the construction of the map $\gl$ in the spirit of \cite{pardon} because we will need it to prove the second part of Theorem \ref{gluing}.

Recall that the source of the curve $\Cl$ is a stable real rational curve $(\Sigma, j_{\Sigma}, c_{\Sigma}, p_1,\ldots, p_n)$. Let us choose once and for all a numbering $\Sigma_1,\ldots,\Sigma_{m+1}$ of the irreducible components of $\Sigma$ and a numbering $n_1,\ldots,n_m\in\Sigma$ of the nodes of $\Sigma$. A neighborhood of this curve in $\rmok$ can be described as follows.

 Around each node and on each branch choose a small disk centered at the node in such a way that the whole collection of disks is globally invariant by $c_{\Sigma}$ and does not contain any of the marked points. Moreover, we fix holomorphic parameterizations of those disks by the unit disk $\Delta\subset \C$ with $0$ sent to the node and such that the restriction of $c_{\Sigma}$ on each neighborhood is given on $\Delta$ by the standard conjugation on $\C$ (note that the way $c_{\Sigma}$ permutes the disks is hidden in the fact that we parameterized them all using the same $\Delta$). 

We now consider the polar coordinates $(r,\theta) \in [0,+\infty[\times S^1\mapsto \e^{-r-i\theta}\in \Delta$ on $\Delta\setminus\{0\}$. For each node, fix a gluing parameter $\alpha_j= \e^{-6R_j + i\Theta_j}\in \C$, $R_j\in [0,+\infty[$ and $\Theta_j\in S^1$. Construct the surface $\Sigma_{\at}$, $\at = (\alpha_1,\ldots,\alpha_m)$, by first removing the part $]6R_j,\infty[\times S^1$ on each branch and for each node $n_j$ and then gluing the remaining parts of the two branches of each node $n_j$ by using the map $(r,\theta)\mapsto (6R_j-r,\Theta_j-\theta)$.

The real structure $c_{\Sigma}$ induces a permutation $\sigma$ of $\{1,\ldots,m\}$ of order $2$ by acting on the nodes of $\Sigma$, and an involution on $\C^m$ given by $(\alpha_1,\ldots,\alpha_m)\in \C^m\mapsto (\overline{\alpha_{\sigma(1)}},\ldots,\overline{\alpha_{\sigma(m)}})\in \C^m$. The fixed points set $\GP$ of the latter involution corresponds to the gluing parameters $\at$ for which $c_{\Sigma}$ induces an orientation reversing involution $c_{\at}$ on $\Sigma_{\at}$. We will call this set the set of real gluing parameters. The situation can be summarized as in Figure \ref{nodesfig}. In particular, the set of real gluing parameters is a product of real lines, one for each node that is fixed by $c_{\Sigma}$, and of real planes, one for each pair of nodes which are exchanged by $c_{\Sigma}$.

\begin{figure}
  \centering
  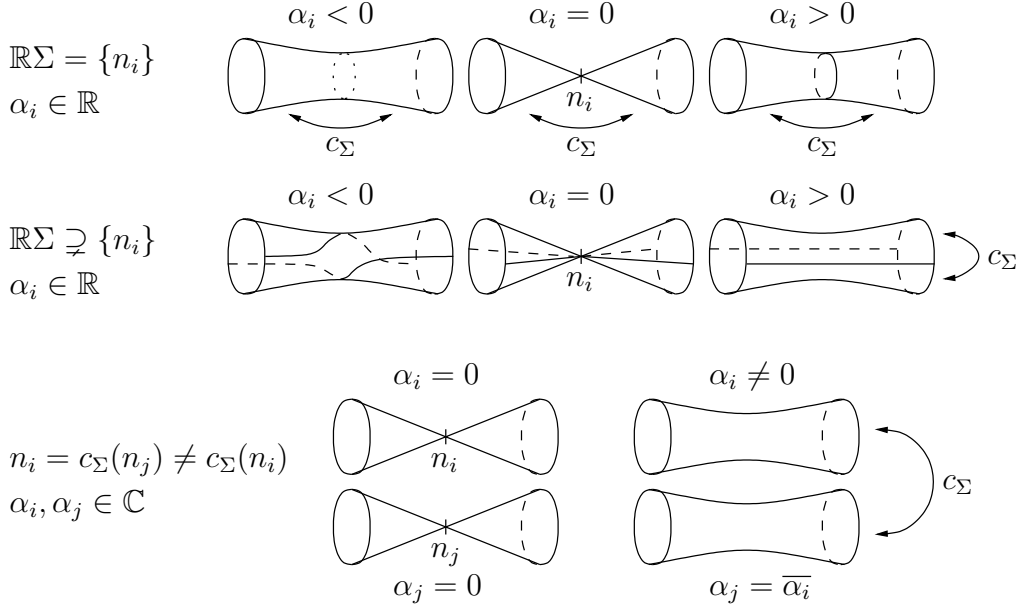
  \caption{The three types of nodes and their gluing}
  \label{nodesfig}
\end{figure}

Finally, we fix an $n-m-3$-dimensional family $(j_y)_{y\in \mathcal{Y}}$ of complex structures on $\Sigma$ parameterized by an open neighborhood $\mathcal{Y}$ of $0$ in $\R^{n-m-3}$ such that
\begin{itemize}
\item $j_0 = j_{\Sigma}$,
\item $d c_{\Sigma} \circ j_y = - j_y \circ d c_{\Sigma}$ for all $y\in \mathcal{Y}$,
\item all those structures are equal to $j_{\Sigma}$ on the previously considered disks around the nodes,
\item and the map $y \in \mathcal{Y}\mapsto (\Sigma,j_y,c_{\Sigma},p_1,\ldots,p_n) \in \rmok$ is a diffeomorphism onto its image, which is a neighborhood of $(\Sigma,j_{\Sigma},c_{\Sigma},p_1,\ldots,p_n)$ in the stratum of $\rmok$ parameterizing the curves with the same topological type as $\Sigma$.
\end{itemize}
The complex structures $(j_y)_{y\in\mathcal{Y}}$ all induce complex structures on the surfaces $\Sigma_{\at}$, again denoted by $j_y$ and for which the involution $c_{\at}$ is anti-holomorphic. The map $(\at,y)\in \GP \times \mathcal{Y}\mapsto (\Sigma_{\at},j_y,c_{\at},p_1,\ldots,p_n)\in \rmok$ is then a diffeomorphism from a neighborhood of $(0,0) \in \GP \times \mathcal{Y}$ onto a neighborhood of $(\Sigma,j_y,c_{\Sigma},p_1,\ldots,p_n)$ in $\rmok$.

Using a riemannian metric on $X$ which is $c_X$ invariant and for which $\uh$ is totally geodesic, for each real gluing parameter $\at\in \GP$ we can define, as in Definition B.3.2 of \cite{pardon}, the pregluing $u_{\at}: (\Sigma_{\at},c_{\at}) \rightarrow (X,c_X)$ of $u_0$. Taking $\nabla$ to be the Levi-Civita connection associated to the previous metric, we set $\nabla^t = \nabla - \delta(t)\circ \nabla\circ \delta(t)$ for each $t\in\cyl$ in a neighborhood of $t_0$. Using the parallel transports associated to those connections, for each $t$ in a neighborhood of $t_0$ and each $\at\in\GP$, we define the pregluing $\xi_{\at,t}$ of a real section $\xi$ of $u_0^*TX$ as in Definition B.3.3 of \cite{pardon}, which is a real section of $u_{\at}^*TX$.

Fix a real number $\epsilon\in ]0,1[$ and an integer $k\geq 2$. Let us consider for each real gluing parameter $\at\in \GP$, each complex structure $j_y$ and each $t\in \cyl$ in a neighborhood of $t_0$, the weighted Sobolev spaces $L^{k,2,\epsilon}(\Sigma_{\at},u_{\at}^*TX;\uh)_{+1}$ and $L^{k-1,2,\epsilon}(\Sigma_{\at}, \Lambda^{0,1}\Sigma_{\at} \fo{j_y}{\delta(t)} u_{\at}^*TX)_{+1}$ as in Definitions B.4.2 and B.4.3 of \cite{pardon}. Roughly speaking, the first space consists of real sections $\xi$ of $u_{\at}^*TX$ such that for all $i$ from $1$ to $n$, $\xi_{p_i} \in T_{u_0(p_i)}H_i$ and such that $\xi$ converges exponentially fast to its value at the middle of the necks where the gluing happens, and the second one consists of real $(0,1)$-forms with value in $u_{\at}^*TX$ converging exponentially fast to an exponentially small value at the middle of the necks.

One can express the perturbed Cauchy-Riemann equations using those spaces. More precisely, define, as in \S B.5 of \cite{pardon}, the function
\[
\begin{array}{c}
\F_{\at,y,t} : L^{k,2,\epsilon}(\Sigma_{\at},u_{\at}^*TX;\uh)_{+1}\oplus \P \rightarrow L^{k-1,2,\epsilon}(\Sigma_{\at}, \Lambda^{0,1}\Sigma_{\at} \fo{j_y}{\delta(t)} u_{\at}^*TX)_{+1} \\
(\xi,e)\mapsto \pt^t_{\ex_{u_{\at}} \xi\rightarrow u_{\at}}(\DB_{j_y,\delta(t)}(\ex_{u_{\at}} \xi) - e(.,\ex_{u_{\at}}\xi,t)),
\end{array}
\]
where $\ex$ is the exponential map associated to the previous riemannian metric, $\pt^t_{\ex_{u_{\at}} \xi\rightarrow u_{\at}}$ is the parallel transport associated to the connection $\nabla^t$, sending sections of $(\ex_{u_{\at}} \xi)^*TX$ to sections of $u_{\at}^*TX$ and $\DB_{j_y,\delta(t)}(u) = d u + \delta(t)\circ d u \circ j_y$ is the standard Cauchy-Riemann operator. We denote by $D_{\at,y,t}$ the derivative at $(0,0)$ of $\F_{\at,y,t}$.

Note that by the assumption on $\P$, the operator $D_{0,0,t_0}$ is surjective. In fact, as is proved in Proposition B.7.9 of \cite{pardon}, for all $(\at,y)$ small enough and all $t$ close enough to $t_0$, the operators $D_{\at,y,t}$ admit a right inverse which is bounded independently of $\at,y$ and $t$.

Then for each gluing parameter $\at\in\GP$ small enough, each complex structure $y\in \mathcal{Y}$ and each $t\in\cyl$ close to $t_0$, one constructs a map $\phi_{\at,y,t} : \ker(\vD_{u_0})\rightarrow L^{k,2,\epsilon}(\Sigma_{\at},u_{\at}^*TX;\uh)_{+1}\oplus \P$ which is defined in a neighborhood of $0$, and satisfies for all $(\xi,e)\in\ker(\vD_{u_0})$ small enough $\F_{\at,y,t}((\xi_{\at,t},e)+\phi_{\at,y,t}(\xi,e)) = 0$ (see Proposition B.9.2 of \cite{pardon}). Moreover, the map $\phi_{\at,y,t}$ is $C^1$ (see Proposition 24 in \cite{floermonop}). The gluing map is then given by
\[
\gl(\at,y,t,(\xi,e)) = ((\Sigma_{\at},j_y,c_{\at},\up), \ex_{u_{\at}}(\xi_{\at,t}+ pr_1(\phi_{\at,y,t}(\xi,e))) , t, e+ pr_2(\phi_{\at,y,t}(\xi,e))),
\]
where $pr_i$, $i=1,2$, are the projections from $L^{k,2,\epsilon}(\Sigma_{\at},u_{\at}^*TX;\uh)_{+1}\oplus \P$ to the respective factors. The fact that $\gl$ is an homeomorphism from a neighborhood of $((\Sigma,j_{\Sigma},c_{\Sigma},\up), t_0, 0)\in \rmok\times \cyl\times \ker(\vD_{u_0})$ to a neighborhood of $\Cl\in V_{\P}$ can be proved as in sections B.10, B.11 and B.12 of \cite{pardon}.

\begin{remark}\label{difffiber}
  Note that since the map $\phi_{\at,y,t}$ is $C^1$, the restriction of $\gl$ to a fiber of $\pi_{\mathcal{U}}$ is also $C^1$. 
\end{remark}

\subsubsection{The linear gluing}\label{parlin}

Before we tackle the orientation statement in Theorem \ref{gluing}, we need to recall the linear gluing procedure adapted to our setting. Let us use the notations of \S \ref{pargl}, that is $(\Sigma = \bigcup_{i=1}^{m+1} \Sigma_i,j_{\Sigma},c_{\Sigma},\up)$ is a stable real rational curve with nodes $\un = (n_1,\ldots,n_m)$, and for each real gluing parameter, we denote by $(\Sigma_{\at},c_{\at})$ the corresponding glued surface. We will also denote the normalization of the surface $\Sigma_{\at}$ by $\tS_{\at}$. Let us moreover fix for each node $n_i\in \Sigma$ an orientation, that is we decorate the two preimages of $n_i$ in the normalization $\tS$ with a sign to get $n_i^+$ and $n_i^-$; we do so while respecting the real structure $c_{\Sigma}$, i.e. if $n_i$ and $n_j$ are complex conjugated nodes, then $c_{\Sigma}(n_i^{\pm}) = n_j^{\pm}$.

Suppose now that we are given an hermitian vector bundle $E$ over $\tS$, equipped with a real structure $c_{E}$, which is the pullback of an hermitian vector bundle on $\Sigma$. For each node $n_j$ choose an hermitian vector bundle $E_j$ over $\Delta^2$ of the same rank as $E$ and $\C$-antilinear isomorphisms between $E_j$ and $E_{\sigma(j)}$ lifting the conjugation if $n_j$ is complex or real and non-isolated, and lifting the involution $(x,y)\in \Delta^2\mapsto (\overline{y},\overline{x})\in\Delta^2$ if $n_j$ is real and isolated. Then, using the maps $\eta_{\alpha_j} : (r,\theta)\in [0,6 R_j]\times S^1 \mapsto (\e^{-r-i\theta},\e^{-6R_j+r+i\theta+i\Theta_j})\in \Delta^2$ if $\alpha_j\neq 0$ (resp. $\eta^-_{0} : (r,\theta)\in [0, +\infty)\times S^1 \mapsto (\e^{-r-i\theta},0)\in\Delta^2$ on the negative side of the node $n_i$ and $\eta^+_{0} : (r,\theta)\in [0, +\infty)\times S^1 \mapsto (0,\e^{-r-i\theta})\in\Delta^2$ on the positive side), we assume moreover that the restriction of $\eta_{\alpha_j}^*E_j$ to $[0,R-1]\times S^1$ is exactly $E$ when $\alpha_j\neq 0$ (resp. $(\eta_{0}^{\pm})^*E_j = E$). Using those, we can define a family of hermitian vector bundles $E_{\at}$ over $\tS_{\at}$ equipped with real structures $c_{E_{\at}}$, for all $\at\in\GP$. We also choose trivializations of the bundles $E_j$ over the bidisks, in such a way that they are compatible with the real stuctures. Those induce trivializations of the $E_{\at}$ on the gluing region.

Choose a finite dimensional vector subspace $P_0$ of $L^{k-1,2}(\tS,\Lambda^{0,1}\tS \otimes E)_{+1}$ such that all the elements of $P_0$ vanish on the gluing disks. Then $P_0$ naturally induces a family a finite dimensional vector subspaces $P_{\at}$ of $L^{k-1,2}(\tS_{\at},\Lambda^{0,1}\tS_{\at} \otimes E_{\at})_{+1}$, all of the same dimension as $P_0$. Here, the complex structure on $\tS_{\at}$ is the one coming from $j_{\Sigma}$.

For each $i\in\{1,\ldots,n\}$, choose a complex vector subspace $K_i$ of the fiber $E_{p_i}$ in such a way that for $i\in \{1,\ldots,s\}$, $c_{E}(K_i) = K_{s+i}$ and for $i\in \{2s+1,\ldots,2s+l\}$, $c_{E}(K_i) = K_{i}$, and let us write $\uk = (K_1,\ldots,K_n)$. Note that $\uk$ is still a family of vector subspaces of the fibers of $E_{\at}$ over the points $p_i$. Let us denote by $L^{k,2}(\tS_{\at},E_{\at};\uk)_{+1}$ the subspace of $L^{k,2}(\tS_{\at},E_{\at})$ consisting of real sections $v$ which satisfy $v_{p_i}\in K_i$ for each $i\in\{1,\ldots,n\}$. We also set $E_{\up}/\uk$ to be the direct sum of the quotient spaces $E_{p_i} / K_i$. It comes with a conjugation and we denote by $\R E_{\up}/\uk$ its fixed points set.

Fix a real generalized Cauchy-Riemann operator $\zD : L^{k,2}(\tS,E)_{+1}\rightarrow L^{k-1,2}(\tS,\Lambda^{0,1}\tS\otimes E)_{+1}$ and denote by $\kD$ its restriction to $L^{k,2}(\tS,E;\uk)_{+1}$. Let $\zDP :L^{k,2}(\tS,E)_{+1}\oplus P_0\rightarrow L^{k-1,2}(\tS,\Lambda^{0,1}\tS\otimes E)_{+1}$ be the operator $\zD+\id$ and let $\kDP$ be its restriction to $L^{k,2}(\tS,E;\uk)_{+1}\oplus P_0$. Note that there is a natural isomorphism
\begin{equation}
  \label{detPz}
  \ddet(\zD) \otimes \det(P_0) = \ddet (\kDP) \otimes \Lambda^{\max}_{\R} \R E_{\up}/\uk.
\end{equation}

If $\ttS$ is an intermediate normalization of $\Sigma$ obtained by normalizing the nodes in the subset $\un'\subset\un$, we can restrict $\zD$ (resp. $\kDP$) on $\ttS$ to obtain an operator $\zD_{\un'} : L^{k,2}(\ttS,E)_{+1}\rightarrow L^{k-1,2}(\tS,\Lambda^{0,1}\tS\otimes E)_{+1}$ (resp. $\kDP_{\un'}:L^{k,2}(\ttS,E;\uk)_{+1}\oplus P_0 \rightarrow L^{k-1,2}(\tS,\Lambda^{0,1}\tS\otimes E)_{+1}$). Similarly, there is a natural isomorphism
\begin{equation}
  \label{detPnz}
  \ddet(\zD_{\un'}) \otimes \det(P_0) = \ddet (\kDP_{\un'}) \otimes \Lambda^{\max}_{\R} \R E_{\up}/\uk.
\end{equation}

On the other hand, the sum of the fibers of $E$ over the points of $\un'$ comes with a conjugation induced by $c_E$ and we denote by $\R E_{\un'}$ the set of fixed points of this conjugation. Note here that if $n_i$ is a real isolated node, then the induced conjugation on $E_{n_i}$ is given by $-c_E$; this is due to the fact that $c_{\Sigma}$ exchanges the branches of the node $n_i$. Using the map $ev_{\un'} :\ker(\kDP)\rightarrow \R E_{\un'}$ given by the difference of the evaluations of a section on both sides of the nodes $\un'$, we have an exact sequence
\begin{equation}
  \label{normal}
  0\rightarrow \ker(\kDP_{\un'})\rightarrow \ker(\kDP) \xrightarrow{ev_{\un'}} \R E_{\un'} \rightarrow \coker(\kDP_{\un'}) \rightarrow \coker(\kDP) \rightarrow 0,
\end{equation}
which combined with the isomorphisms \eqref{detPz} and \eqref{detPnz} gives the isomorphism
\begin{equation}
  \label{isonormal}
  \ddet(\zD) = \ddet(\zD_{\un'})\otimes \Lambda^{\max}\R E_{\un'}.
\end{equation}
The isomorphism \eqref{isonormal} does not depend on $\uk$ or $P_0$, and is invariant under the homotopy of $\zD$; i.e. if $(\zD_t)_{t\in [0,1]}$ is a homotopy of real generalized Cauchy-Riemann operators, then there is a commutative square
\[
\xymatrix{
 \ddet(\zD_0)\ar[r]\ar[d] & \ddet(\zD_{\un',0})\otimes \Lambda^{\max}\R E_{\un'}\ar[d]\\
 \ddet(\zD_1)\ar[r] & \ddet(\zD_{\un',1})\otimes \Lambda^{\max}\R E_{\un'},
}
\]
where the horizontal maps are the isomorphisms \eqref{isonormal} and the vertical come from the trivialization of the determinant bundles over the homotopy.

Take a real gluing parameter $\at\in\GP$. We say that a real Cauchy-Riemann operator $\zDa : L^{k,2}(\tS_{\at},E_{\at})_{+1}\rightarrow L^{k-1,2}(\tS_{\at},\Lambda^{0,1}\tS_{\at}\otimes E_{\at})_{+1}$ is the gluing of $\zD$ if the two operators coincide on the complement of the gluing region and on the cylinders $[0,R_i-1]\times S^1$ and $[5R_i+1,6R_i]\times S^1$ for each $i\in\{1,\ldots,m\}$, and if $\zDa$ restricts to the standard Cauchy-Riemann operator on the cylinders $[R_i,5R_i]$ (in the fixed trivialization of $E_i$). Restricting $\zDa$ to $L^{k,2}(\tS_{\at},E_{\at};\uk)_{+1}$ we obtain operators $\kDa$ which we call the gluing of $\kD$. Let us also denote by $\zDPa : L^{k,2}(\tS_{\at},E_{\at})_{+1}\oplus P_{\at}\rightarrow L^{k-1,2}(\tS_{\at},\Lambda^{0,1}\tS_{\at}\otimes E_{\at})_{+1}$ the operator $\zDa+\id$ and $\kDPa$ its restriction to $L^{k,2}(\tS_{\at},E_{\at};\uk)_{+1}\oplus P_{\at}$. Note that there is a natural isomorphism
\begin{equation}
  \label{detPa} \ddet(\zDa) \otimes \det(P_{\at})= \ddet(\kDPa) \otimes \R E_{\up}/\uk.
\end{equation}

Let us give two examples. The first one comes directly from \S \ref{pargl}. The bundles $E_{\at}$ are given by the pullback bundles $u_{\at}^*TX$ (with the complex structure induced by $\delta(t_0)$) and the operators $\kDPa$ are induced by the operators $D_{\at,0,t_0}$.

For the second one, the bundles $E_{\at}$ are the trivial bundles $\TC^2$ over the surfaces $\tS_{\at}$ and the operators are all the standard ones. We will study this particular case in more detail in Lemma \ref{trivgl}.

For $\at\in\GP$, we will denote by $\una$ the subset of $\un$ consisting of the nodes $n_i$ for which the corresponding $\alpha_i$ is non-zero, i.e. the set of nodes which do not survive in $\Sigma_{\at}$.

The following proposition is an adaptation of an already known result to our particular setting (see Theorem 2.4.1 in \cite{ww}).

\begin{prop}\label{lingl}
Suppose that the restriction $\kDP_{\un}$ of $\kDP$ to $L^{k,2}(\Sigma,E;\uk)_{+1}\oplus P_0$ is surjective. Then, for $\at$ small enough, the operators $\kDPa$ are also surjective and there exists an exact sequence
  \begin{equation}
    \label{linseq}
0 \rightarrow \ker(\kDPa)\xrightarrow{\lPKa} \ker(\kDP) \xrightarrow{\fPKa}  \R E_{\una} \rightarrow 0,
  \end{equation}
which gives rise together with \eqref{detPz}, \eqref{isonormal} and \eqref{detPa} to an isomorphism
\begin{equation}\label{liniso}
\ddet(\zDa) = \ddet(\zD_{\un_{\at}}).
\end{equation}
This isomorphism up to homotopy does not depend on the choice of the space $P_0$ and on the subspaces $\uk$ and is invariant under the homotopy of the data $(\zD,\zDa)$.
\end{prop}

Note that thanks to the sequence \eqref{normal}, the condition that $\kDP_{\un}$ is surjective is equivalent to the fact that both the map $ev_{\un}$ and the operator $\kDP$ are surjective.

\begin{proof}
We use Sobolev spaces with negative exponential weight $-\epsilon\in ]-1,0[$ on the gluing regions. More precisely, for all $\at\in\GP$, take $w_{\at} : \tS_{\at}\rightarrow \R$ a smooth $\Z/2\Z$-invariant function taking value $1$ outside of the gluing regions, being equal to $\e^{-\epsilon r} + \e^{-\epsilon (6R_i - r)}$ on the cylinders $[R_i,5R_i]\times S^1$ for which $\alpha_i\neq 0$ and equal to $\e^{-\epsilon r}$ on the cylinders $[0,+\infty[\times S^1$ corresponding to the nodes of $\Sigma_{\at}$. We denote by $L^{k,2,-\epsilon}(\tS_{\at},E_{\at};\uk)_{+1}$ and $L^{k-1,2,-\epsilon}(\tS_{\at},\Lambda^{0,1}\tS_{\at}\otimes E_{\at})_{+1}$ the spaces of real sections $\xi$ of $E_{\at}$ (resp. $\Lambda^{0,1}\tS_{\at}\otimes E_{\at}$) such that the section $w_{\at}\xi$ is in $L^{k,2}(\tS_{\at},E_{\at};\uk)_{+1}$ (resp. $L^{k-1,2}(\tS_{\at},\Lambda^{0,1}\tS_{\at}\otimes E_{\at})_{+1}$). Note that this change of Sobolev spaces does not modify the kernels and cokernels of the operators we consider because $\epsilon \in ]0,1[$ (see \cite{lockhartm}). We also fix once and for all a scalar product on $P_0$.

The first map in the sequence \eqref{linseq} is obtained as follows. For each $\at\in\GP$ fix a $\Z/2\Z$-invariant cutoff function $\varphi_{\at}:\Sigma_{\at}\rightarrow [0,1]$ which is supported outside the cylinders $[3R_i/4,21R_i/4]\times S^1$ for each $i\in\{1,\ldots,m\}$ such that $\alpha_i\neq 0$, equal to $1$ outside the cylinders $[R_i/4,23R_i/4]\times S^1$ for the same values of $i$, and such that its derivatives on the cylinders $[0,6R_i]\times S^1$ are bounded by a constant times $1/R_i$. Those cutoff functions induce maps $p_{\at} : (\xi,e)\in\ker(\kDP) \mapsto (\varphi_{\at}\xi,e)\in L^{k,2,-\epsilon}(\tS_{\at},E_{\at};\uk)_{+1}\oplus P_{\at}$ which are injective provided $\at$ is small enough. The map $\lPKa$ is the orthogonal projection on the image of $p_{\at}$ composed with the inverse of $p_{\at}$. 

One can prove in the same way as in Theorem 2.4.1 of \cite{ww} that there exists a constant $C>0$ such that for all $\at\in\GP$ small enough and for all elements $(\xi,e)\in p_{\at}(\ker(\kDP))^{\perp}$,
\begin{equation}\label{ineq}
\lVert (\xi,e) \rVert_{k,2,-\epsilon} \leq C \lVert \kDPa (\xi,e)\rVert_{k-1,2,-\epsilon},
\end{equation}
where the norms are the exponential ones we defined at the beginning. We do not give the proof of this inequality as it is an easy adaptation of the one given in \cite{ww}. The fact that $\lPKa$ is injective for $\at$ small enough follows.

To construct the second map, we first fix $\Z/2\Z$-invariant cutoff functions $\psi_{\at}: \Sigma_{\at}\rightarrow [0,1]$ which are supported on the cylinders $[R_i,5R_i]\times S^1$ for each $i\in\{1,\ldots,m\}$ such that $\alpha_i\neq 0$, equal to $1$ on the cylinders $[2R_i,4R_i]\times S^1$ for the same values of $i$, and such that their derivatives on the cylinders $[0,6R_i]\times S^1$ are bounded by a constant times $1/R_i$. We then include the space $\R E_{\una}$ in  $L^{k-1,2,-\epsilon}(\tS_{\at},\Lambda^{0,1}\tS_{\at}\otimes E_{\at})_{+1}$ by sending an element $\ux = (x_j)_{j|\alpha_j\neq 0}\in \R E_{\una}$ to the form $q_{\at}(\ux) = \psi_{\at}\displaystyle\sum_{j | \alpha_j\neq 0} x_j (\dd r_j - i \dd \theta_j)$. Here, the coordinate $r_j$ is taken to increase from the $-$ side of the node $n_j$ to the $+$ side. Again, when $\at$ is small enough, the map $q_{\at}$ is injective. From the orthogonal decompositions
\begin{gather*}
L^{k,2,-\epsilon}(\tS_{\at},E_{\at};\uk)_{+1}\oplus P_{\at} = p_{\at}(\ker(\kDP))^{\perp}\oplus p_{\at}(\ker(\kDP)) \\
\intertext{ and }
L^{k-1,2,-\epsilon}(\tS_{\at},\Lambda^{0,1}\tS_{\at}\otimes E_{\at})_{+1} = q_{\at}(\R E_{\una})^{\perp}\oplus q_{\at}(\R E_{\una}),
\end{gather*}
we obtain four operators $D_{\at}^{ij}$, $i,j=1,2$, by restricting and projecting on each of the components the operator $\kDPa$. It follows then from the bounds on the derivatives of $\psi_{\at}$ and from the inequality \eqref{ineq} that $D_{\at}^{11} : p_{\at}(\ker(\kDP))^{\perp}\rightarrow q_{\at}(\R E_{\una})^{\perp}$ is an isomorphism with uniformly bounded inverse as soon as $\at$ is small enough, as in \cite{ww}. Moreover, in our case, the choices of cutoff functions imply that the operator $D_{\at}^{22} : p_{\at}(\ker(\kDP)) \rightarrow q_{\at}(\R E_{\una})$ vanishes. The map $\fPKa : \ker(\kDP)\rightarrow \R E_{\una}$ is given by $q_{\at}^{-1}\circ D_{\at}^{12}\circ (D_{\at}^{11})^{-1}\circ D_{\at}^{21} \circ p_{\at}$. It is straightforward to check that the kernel of $\fPKa$ is exactly the image of $\lPKa$. Moreover, as mentioned in \cite{ww}, multiplying $\at$ by a parameter $\varepsilon > 0$ and taking the limit when $\varepsilon$ goes to zero this map converges to the map $ev_{\una}$ given by the difference of the evaluations of a section on both sides of the nodes $n_i$ for which $\alpha_i\neq 0$. In particular, since by assumption this latter map is surjective, the map $\fPKa$ is also surjective when $\at$ is small enough. This shows that the sequence \eqref{linseq} is exact and that the operators $\kDPa$ are surjective.

To show that the isomorphism \eqref{liniso} does not depend on the choice of $P_0$, we take another finite dimensional vector subspace $P'_0$ of $L^{k-1,2}(\tS,\Lambda^{0,1}\tS \otimes E)_{+1}$ such that all the elements of $P'_0$ vanish on the gluing disks and suppose that $P_0$ is included in $P'_0$. We take $\at$ small enough and consider the following diagram
\begin{equation}\label{bigdiag}
\xymatrix{  & 0 \ar[d]              & 0 \ar[d]                 &                         &  \\
0 \ar[r] & \ker(\kDPa) \ar[d]_{i_{\at}}\ar[r]^{\lPKa} & \ker(\kDP) \ar[d]_{\ti_{\at}} \ar[r]^{\fPKa} & \R E_{\una} \ar[d]^{\id}\ar[r] & 0 \\
0 \ar[r] & \ker(\kDPpa) \ar[d]_{pr_{\at}}\ar[r]^{\lPpKa} & \ker(\kDPp) \ar[d]_{\tpr_{\at}} \ar[r]^{\fPpKa} & \R E_{\una} \ar[r]       & 0 \\
         & P_0'/ P_0 \ar[r]^{\id}\ar[d]   & P_0'/ P_0 \ar[d]    &                         &  \\
         & 0                        & 0                        &                         &
}
\end{equation}
The first two rows are the sequences \eqref{linseq} for $P_0$ and $P'_0$ respectively, giving rise to the isomorphisms \eqref{liniso}. The first column is given by the natural inclusion of $\ker(\kDPa)$ in $\ker(\kDPpa)$ and then the projection on the $P'_0$ factor. To make the diagram commutative, we cannot take the maps in the second column to be the inclusion and the projection as for the first column. We proceed as follows. First, fix a vector subspace $F$ of $\ker(\kDP)$ such that the evaluation map $ev_{\una}$ is an isomorphism. Then, for $\at$ small enough, we get two direct sum decompositions
\begin{gather*}
\ker(\kDP) = \lPKa(\ker(\kDPa)) \oplus F \\ \intertext{and}
\ker(\kDPp) = \lPpKa(\ker(\kDPpa)) \oplus F.
\end{gather*}
We define $\ti_{\at}$ to be $\lPpKa\circ i_{\at}\circ (\lPKa)^{-1}$ on $\lPKa(\ker(\kDPa))$ and $(\fPpKa)^{-1}\circ \id\circ \fPKa$ on $F$. The map $\tpr_{\at}$ is given by $pr_{\at}\circ (\lPpKa)^{-1}$ on $\lPpKa(\ker(\kDPpa))$ and the zero map on $F$. The obtained diagram is commutative. We now show that the second column of diagram \eqref{bigdiag} induces the same isomorphism $\ddet(\kDPp) = \ddet(\kDP)\otimes \Lambda^{\max}_{\R} P'_0/P_0$ up to homotopy as the sequence
\[
0\rightarrow \ker(\kDP) \xrightarrow{i_0} \ker(\kDPp) \xrightarrow{pr_0} P'_0/P_0 \rightarrow 0,
\]
which will show that the isomorphism \eqref{liniso} is the same if we use $P_0$ or $P'_0$. To that end, we define an orientation-preserving automorphism $g_{\at}$ of $\ker(\kDPp)$ such that the diagram
\[
\xymatrix{
0 \ar[r] & \ker(\kDP) \ar[r]^{i_0} \ar[d]_{\id} & \ker(\kDPp) \ar[r]^{pr_0} \ar[d]_{g_{\at}} & P'_0/P_0 \ar[r] \ar[d]_{\id} & 0 \\
0 \ar[r] & \ker(\kDP) \ar[r]^{\ti_{\at}}  & \ker(\kDPp) \ar[r]^{\tpr_{\at}}  & P'_0/P_0 \ar[r]  & 0 
}
\]
commutes. Denote by $G\subset \ker(\kDPp)$ the orthogonal complement of $\ker(\kDP)$. For $\at$ small enough, $\ti_{\at}$ is close to the inclusion $i_0$. Indeed, take an element $\lPKa(\xi,e) + (\zeta,e')$ of $\ker(\kDP) = \lPKa(\ker(\kDPa)) \oplus F$, and write
\begin{multline}\label{decoup}
\ti_{\at}(\lPKa(\xi,e) + (\zeta,e')) - (\lPKa(\xi,e) + (\zeta,e'))\\ = (\lPpKa(\xi,e) - \lPKa (\xi,e)) + ((\fPpKa)^{-1}\circ \fPKa((\zeta,e')) - (\zeta,e')).
\end{multline}
The second term on the right-hand side is small compared to the norm of $(\zeta,e')$ because of the convergence of $(\fPpKa)^{-1}\circ \fPKa$ to the identity. For the first term, first note that there exists a constant $C_1>0$ such that for all $\at$ small enough and all $w\in\ker(\kDPp)$,
  \begin{equation}
    \lVert w\rVert_{k,2,-\epsilon}\leq C_1\rVert p_{\at}(w)\lVert_{k,2,-\epsilon},
  \end{equation}
which follows from the fact that the $L^{k,2,-\epsilon}$ norm of an element of $\ker(\kDPp)$ outside the gluing region controls the norm over the whole curve $\Sigma$. Applying this to our case, we obtain
\begin{multline*}
\lVert \lPpKa(\xi,e) - \lPKa (\xi,e)\rVert_{k,2,-\epsilon} \leq C_1 \lVert p_{\at}\circ \lPpKa(\xi,e) - (\xi,e)\rVert_{k,2,-\epsilon}\\
 + C_1\lVert (\xi,e) - p_{\at}\circ \lPKa (\xi,e))\rVert_{k,2,-\epsilon}.
\end{multline*}
Then use the inequality \eqref{ineq} so we find a positive constant $C_2$ such that for $\at$ small enough,
\begin{gather*}
\lVert p_{\at}\circ \lPpKa(\xi,e) - (\xi,e)\rVert_{k,2,-\epsilon} \leq C_2 \lVert \kDPpa (p_{\at}\circ \lPpKa(\xi,e)) \rVert_{k-1,2,-\epsilon} \\ \intertext{and}
\lVert (\xi,e) - p_{\at}\circ \lPKa (\xi,e))\rVert_{k,2,-\epsilon}\leq C_2 \lVert \kDPa (p_{\at}\circ \lPKa (\xi,e)) \rVert_{k-1,2,-\epsilon}.
\end{gather*}
It then follows from the bounds on the cutoff functions $\varphi_{\at}$ that there is a positive constant $C_3$ such that
\begin{gather*}
 \lVert \kDPpa (p_{\at}\circ \lPpKa(\xi,e)) \rVert_{k-1,2,-\epsilon} \leq C_3 \sum_{i|\alpha_i\neq 0}\frac{1}{R_i}\lVert \lPpKa(\xi,e)\rVert_{k,2,-\epsilon} \\ \intertext{and}
 \lVert \kDPa (p_{\at}\circ \lPKa (\xi,e)) \rVert_{k-1,2,-\epsilon} \leq C_3 \sum_{i|\alpha_i\neq 0}\frac{1}{R_i}\lVert \lPKa(\xi,e)\rVert_{k,2,-\epsilon}.
\end{gather*}
 Thus, both terms on the right-hand side of \eqref{decoup} are small compared to the norm of $\lPKa(\xi,e) + (\zeta,e')$ and $\ti_{\at}$ is close to the inclusion map. In particular, it is transverse to $G$. The map $g_{\at}$ is then given by $\ti_{\at}$ on $\ker(\kDP)$ and by $\tpr_{\at}^{-1}\circ pr_{0}$ on $G$. This map is in fact close to the identity. Indeed, if $(\xi,e) + (\zeta,e')$ is an element of $\ker(\kDPp) = \ker(\kDP)\oplus G$, then
\[
g_{\at}((\xi,e) + (\zeta,e')) - ((\xi,e) + (\zeta,e')) = (\ti_{\at}(\xi,e) - (\xi,e)) + (\tpr_{\at}^{-1}\circ pr_{0}(\zeta,e') - (\zeta,e')).
\]
Then the norm of the first term is small compared to the norm of $(\xi,e)$ because $\ti_{\at}$ is close to the inclusion. The norm of the second term is small compared to the norm of $(\zeta,e')$ because $\tpr_{\at}$ converges to $pr_0$ when $\at$ goes to $0$.

This shows that $g_{\at}$ preserves the orientations of $\ker(\kDPp)$.

Thus the isomorphism \eqref{liniso} does not depend on the choice of $P_0$. The proof that the isomorphism \eqref{liniso} does not depend on the choice of $\uk$ is similar so we do not reproduce it here. The homotopy invariance follows readily from this independence (see Lemma 9.6 \cite{hutchtaubes}).
\end{proof}

\begin{remark}\label{remlin}
  Let $E'$ be another hermitian vector bundle over $\tS$ given as the pullback of a bundle on $\Sigma$, and equipped with a real structure $c_{E'}$. Suppose we have the same data as for $E$ and thus construct a family $E'_{\at}$ of hermitian vector bundles over $\tS_{\at}$ also equipped with real structures $c_{E'_{\at}}$. If $\Phi$ is an isomorphism between $(E',c_{E'})$ and $(E,c_E)$ such that it is the identity in the fixed trivializations on the gluing disks, then it naturally induces a family of isomorphisms $\Phi_{\at}$ between $(E'_{\at},c_{E'_{\at}})$ and $(E_{\at},c_{E_{\at}})$. Moreover, it follows directly from the definition of the sequence \eqref{linseq} that the following diagram commutes
\[
\xymatrix{
0 \ar[r] & \ker(\Phi_{\at}^*\kDPa)\ar[rr]^{l^{\Phi^{-1}(P_0),\Phi^{-1}(\uk)}_{\at}} \ar[d]_{\Phi_{\at}} & & \ker(\Phi^*\kDP) \ar[rrr]^{f^{\Phi^{-1}(P_0),\Phi^{-1}(\uk)}_{\at}}\ar[d]_{\Phi} & & & \R E'_{\una} \ar[r] \ar[d]_{\Phi} & 0\\
0 \ar[r] & \ker(\kDPa)\ar[rr]^{\lPKa} & & \ker(\kDP) \ar[rrr]^{\fPKa} & & & \R E_{\una} \ar[r] & 0\\
}
\]
\end{remark}

Take $\at\in \GP$ a small gluing parameter with no vanishing coordinate and let us now describe more precisely the case of the standard Cauchy-Riemann operators $\uDB_{\TC^2,\at}$ on the trivial complex vector bundles of rank $2$ over $\Sigma_{\at}$. Those operators are all surjective, their kernels consisting of the constant sections with value in $\R^2$. In particular, the determinants of these operators all come with a natural orientation induced by the orientation of $\R^2$. On the curve $\Sigma$, the standard Cauchy-Riemann operator $\DB_{\TC^2}$ on the rank 2 trivial complex vector bundle is also surjective and its kernel is naturally isomorphic to $\R^2$, so its determinant is oriented.

\begin{lemma}\label{trivgl}
The isomorphism \eqref{liniso} applied to the operators $\uDB_{\TC^2,\at}$ and $\DB_{\TC^2}$ preserves the natural orientations described above.
\end{lemma}

\begin{proof}
 On the curve $\tS$, the standard Cauchy-Riemann operator $\uDB_{\TC^2}$ on the rank 2 trivial complex vector bundle is still surjective. Its kernel is oriented by the sequence \eqref{normal} on the one hand,
\[
0\rightarrow \ker(\DB_{\TC^2}) \rightarrow \ker(\uDB_{\TC^2})\xrightarrow{ev_{\un}} \R \TC^2_{\un}\rightarrow 0,
\]
 and by the sequence \eqref{linseq} on the other hand,
\[
0\rightarrow \ker(\uDB_{\TC^2,\at}) \xrightarrow{l_{\at}} \ker(\uDB_{\TC^2})\xrightarrow{f_{\at}} \R \TC^2_{\un}\rightarrow 0,
\]
 where the real vector space $\R \TC^2_{\un}$ is oriented using the numbering of the nodes.

  Choose a vector space $F\subset \ker(\uDB_{\TC^2})$ such that the restriction of $ev_{\un}$ on $F$ is an isomorphism. Then, since the map $\fa$ of Proposition \ref{lingl} converges to the map $ev_{\un}$, for $\at$ small enough, the restriction of $\fa$ to $F$ is an isomorphism and given a basis of $F$ its images by $ev_{\un}$ and by $\fa$ give the same orientation of $\R \TC_{\un}$. Thus, to prove the statement, we only need to check that the two orientations we have on $\ker(\uDB_{\TC^2})$ coincide.

Fix a norm on $\ker(\uDB_{\TC^2})$. Using inequality \eqref{ineq} one sees that there exists a positive constant $K>0$ such that for $\at$ small enough and for all $c\in \R^2 = \ker(\uDB_{\TC^2,\at})$
\[
\lVert \la(c) - c\rVert \leq K \sum_{i=1}^m\frac{1}{R_i}\lVert c \rVert,
\]
where we see $c$ in $\ker(\uDB_{\TC^2})$ as the constant section with value $c$. Thus, for $\at$ small enough, one can homotope the image of a positive basis of $\ker(\uDB_{\TC^2,\at})$ by $\la$ to a positive basis of $\ker(\DB_{\TC^2})$ without crossing $F$, which concludes the proof.
\end{proof}

\subsubsection{Gluing orientations}\label{orsection}

We now tackle the orientation statement appearing in Theorem \ref{gluing}. To that end, let us fix an orientation on $\cyl$ and $\P$, and let us also fix an oriented real $Spin$ structure on $(TX, d c_X)$. Then, Proposition \ref{porient} shows that the open set $U_{\P}\cap \rmdp$ is oriented. 

We will first define an orientation on the open set $\mathcal{U}\subset \rmook\times \cyl \times \ker(\vD_{u_0})$ given by Theorem \ref{homeo}, then check that the map $\gl : \mathcal{U}\rightarrow U_{\P}$ preserves the orientations.

\textbf{Orientation of $\rmook\times \cyl \times \ker(\vD_{u_0})$ :} We have already fixed an orientation on $\cyl$. On $\rmook$, we pick the orientation defined by Ceyhan, \S 5.4.1 in \cite{ceyhan}. Here, we use the assumption we made that none of the points $p_i$ is real to be sure that $\rmook$ is orientable (Theorem 5.7 in \cite{ceyhan}).

To orient $\ker(\vD_{u_0})$, which is the same as orienting $\ddet(\vD_{u_0})$, we first use the diagram
\[
\xymatrix{
0 \ar[r] & L^{k,p}(\Sigma,u_0^*TX;\uh)_{+1} \oplus \P \ar[d]^{\vD_{u_0}} \ar[r] & L^{k,p}(\Sigma,u_0^*TX)_{+1}\oplus \P \ar[d]^{\pD_{u_0}} \ar[r]^{ev_{\up}} & \R \left(T_{u_0(\up)}X / T_{u_0(\up)}\uh \right) \ar[d]^0 \ar[r] & 0 \\
0 \ar[r] & L^{k-1,p}(\tS,\Lambda^{0,1}\tS\otimes u_0^*TX)_{+1} \ar[r] & L^{k-1,p}(\tS,\Lambda^{0,1}\tS\otimes u_0^*TX)_{+1} \ar[r] & 0 \ar[r] & 0,
}
\]
which gives the isomorphism $\ddet(\pD_{u_0}) = \ddet(\vD_{u_0})\otimes \Lambda_{\R}^{\max}\R \left( T_{u_0(\up)} X /  T_{u_0(\up)} \uh\right)$. Since all the points $p_i$ are complex, their numbering orients the space $\R \left( T_{u_0(\up)} X /  T_{u_0(\up)} \uh\right) = \R T_{\up}\Sigma$ (where the last isomorphism is given by $d u_0$). To orient $\ddet(\pD_{u_0})$, we will construct below an isomorphism $\Phi$ between the trivial complex vector bundle of rank $2$ over $\Sigma$ and $u_0^*T X$. We then use the orientation of $\P$ and of the determinant of the standard operator $\DB_{\TC^2}$ on the trivial bundle we fixed at the end of \S \ref{parlin} in order to orient $\ddet(\Phi^*\pD_{u_0})$. Using $\Phi$, we get the desired orientation on $\ddet(\pD_{u_0})$ and thus on $\ker(\vD_{u_0})$.

Let us now explain how we choose $\Phi$. If $\R \Sigma\neq \emptyset$, take a $Spin$ structure $\zeta_{u_0(\R\Sigma)}$ on $\R X$ given by Lemma \ref{pintriv}. Choose trivializations of $(TX,d c_X)$ in a $\Z/2\Z$-invariant neighborhood of $u_0(n_i)$, $i=1,\ldots,m$, such that for each real node $n_i$, the orientation on $T_{u(n_i)}\R X$ induced by those trivializations coincide with the one given by the oriented real $Spin$ structure. Using those trivializations, we get an isomorphism between the trivial complex vector bundle of rank $2$ over $\Sigma$ and $u_0^*T X$ on a neighborhood of the nodes. Extend it over the whole curve $\Sigma = \bigcup_{i=1}^m \Sigma_i$ as follows : 
\begin{itemize}
\item if the component $\Sigma_i$ is stable by $c_{\Sigma}$ and has non-empty real part, then we take $\Phi$ on $\Sigma_i$ to pullback the $Spin$ structure $\zeta_{u_0(\R\Sigma)}$ on $u_0^* T\R X$ to the natural one on the trivial bundle $\TR^2$ over the real part of $\Sigma_i$,
\item if the component $\Sigma_i$ is stable by $c_{\Sigma}$ but has empty real part, then we take $\Phi$ on $\Sigma_i$ to pullback the oriented real $Spin$ structure on $(u_0^*TX,d c_X)$ to the natural one on $(\TC^2,c_{\TC})$ over $\Sigma_i$,
\item on all the other components, we extend $\Phi$ arbitrarily.
\end{itemize}

\begin{remark}\label{remspin}
  As we mentionned in Remark \ref{remlin}, the isomorphism $\Phi$ induces a family of isomorphisms $\Phi_{\at} : (\TC^2,c_{\TC})\rightarrow (u_{\at}^*TX,d c_X)$, $\at\in\GP$. The choices we made for $\Phi$ together with Lemma \ref{pintriv} imply that when $\Sigma_{\at}$ is irreducible, then the pullback by $\Phi_{\at}$ of the real oriented $Spin$ structure given on $(u_{\at}^*TX,d c_X)$ coincides with the natural one on $(\TC^2,c_{\TC})$, and in the case where $\R\Sigma_{\at}$ is non-empty, the pullback of $\zeta_{u_{0}(\R\Sigma)}$ coincides with the natural $Spin$ structure on $\TR^2$ over $\R \Sigma_{\at}$.
\end{remark}

\textbf{The map $\gl$ preserves the orientations :}

We described above the orientation we put on $\mathcal{U}$. Let us briefly recall in the current setting how we oriented $U_{\P}\cap \rmdp$ in Proposition \ref{porient}. Consider the forgetful map $\pi_{\rmook,\cyl} : U_{\P}\cap \rmdp\rightarrow \rmk\times \cyl$. It is a submersion because the operators $\vD_u$ are surjective for all the curves in a neighborhood of $u_0$ (see Proposition \ref{surj}). Thus, to orient $U_{\P}\cap \rmdp$, we orient the product $\rmk\times \cyl$ and the fiber of $\pi_{\rmook,\cyl}$.

 The product is oriented using the orientation of $\cyl$ on one hand. On the other hand, recall that each component of $\rmk$ that intersects the image of $\pi_{\rmook,\cyl}$ can be seen as follows. Using the notations we introduced in \S \ref{parthick}, it is the quotient of the real part $\R_{c_S}\TT_n$ of the Teichmüller space $\TT_n$ by the subgroup $\R_{c_S}\Gamma^+_n$ of $\Gamma_n^+$ of elements commuting with $c_S$, for some real structure $c_S$ on $S$ such that all the marked points $\uz$ are complex conjugated. Now, the determinant bundle of $\R_{c_S}\TT_n$ is canonically isomorphic to the line bundle $\ddet(\DB_{\tau,\uz})_{\tau\in\R_{c_S}\TT_n}$ where $\DB_{\tau,\uz}$ is the restriction to $L^{k,p}(S,TS;\uz)_{+1}$ of the operator $\DB_{\tau} : L^{k,p}(S,TS)_{+1}\rightarrow L^{k-1,p}(S,\Lambda^{0,1}S\otimes TS)_{+1}$ given by the holomorphic structure on $TS$ equipped with $J_S(\tau)$ (see for example Lemma 1.8 in \cite{article}). We then orient this latter line bundle by using the isomorphism $\ddet(\DB_{\tau}) = \ddet(\DB_{\tau,\uz})\otimes \Lambda^{\max}_{\R}\R T_{\uz}S$. Since none of the marked points is real, the line $\Lambda^{\max}_{\R}\R T_{\uz}S$ is naturally oriented using the numbering of $\uz$, and we fixed an orientation of $\ddet(\DB_{\tau})$ by using an isomorphism of $(S,J_S(\tau),c_S)$ with $(\cpo,i,c_{\emptyset})$ or $(\cpo,i,c_{\rpo})$ as in Proposition \ref{dbj}. The orientation of $\R_{c_S}\TT_n$ we thus obtain descends to an orientation of the component of $\rmk$ we consider since the action of $\R_{c_S}\Gamma^+_n$ preserves it.

 The fiber of $\pi_{\rmook,\cyl}$ is oriented by orienting $\ker(\vD_u)$. This is done using the orientation of $\P$, the natural orientation on $\R \left( T_{u(\uz)} X /  T_{u(\uz)} \uh\right)$ coming from the one on $\R T_{\uz}S$, and an isomorphism between $(u^*TX,dc_X)$ and $(\TC^2,c_{\TC})$ preserving the appropriate structures.

This orientation does not exactly coincide with the one we chose in Proposition \ref{porient} because of the order in which we orient each term, but one can check that it differs from it only when $\P$ is odd-dimensional, and then it differs for all the curves. Thus, we will forget about this difference, and work with the orientation we just defined.

Now, to check that $\gl$ preserves the orientations, we use the triangle \ref{triangle} of Theorem \ref{homeo}. Indeed, it implies that is is enough to check two things : 
\begin{enumerate}
\item that the orientation we put on $\rmk$ coincides with the one given on $\rmook$ by Ceyhan in \S 5.4.1 in \cite{ceyhan} and that we used to orient $\mathcal{U}$,
\item that for all $\at\in\GP$ small enough and with no vanishing coordinate, the differential at $(0,0)\in\ker(\vD_{u_0})$ of the restriction of $\gl$ to $\{(\at,0,t_0)\}\times\ker(\vD_{u_0})$ is an orientation preserving isomorphism onto $\ker(\vD_{\gl(\at,0,t_0,(0,0))})$.
\end{enumerate}

Those two points imply that $\gl$ is orientation preserving at all points $(\at,0,t_0,(0,0))\in\rmk\times\cyl\times\ker(\vD_{u_0})$ and thus that $\gl$ is everywhere orientation preserving.

We treat the first point in Lemma \ref{orbase} and the second in Proposition \ref{orfiber}.

\begin{lemma}\label{orbase}
  The orientation of the components of $\rmk$ intersecting the image of $\pi_{\rmook,\cyl}$ coming from the orientation of the corresponding real Teichmüller spaces coincides with the orientation defined by Ceyhan in \S 5.4.1 in \cite{ceyhan}.
\end{lemma}

\begin{proof}
  Let $\OO\subset\rmk$ denote a component of $\rmk$ intersecting the image of $\pi_{\rmook,\cyl}$. There are two cases depending on whether the curves in $\OO$ have non-empty real part or no. Since the proof is identical in both cases, we will assume that the curves in $\OO$ have non-empty real part and omit the other case. 

Every element in $\OO$ has a unique representative $(\cpo,i,c_{\rpo},\ux)$ where
\[
\ux = (x_1,\ldots,x_{k-2},\varepsilon \lambda i, i, \overline{x_1},\ldots, \overline{x_{k-2}}, -\varepsilon \lambda i, -i),
\]
with $x_1,\ldots,x_{k-2}\in\cpo$, $\lambda\in ]0,1[$ and $\varepsilon = \pm 1$ depending on the component $\OO$. This induces coordinates on $\OO$, and in those, the volume form 
\[
\omega_{\OO} =\displaystyle \left(\frac{i}{2}\right)^{k-2}\bigwedge_{j=1}^{k-2} (d x_i\wedge d \overline{x_i}) \wedge d \lambda
\]
 gives the orientation defined by Ceyhan on the component $\OO$ (see \S\S 4.2.2 and 5.4.1 in \cite{ceyhan}).

Take an orientation reversing involution $c_S$ on $S$ and $\tau\in \R_{c_S}\TT_n$ such that the curve $(S,J_S(\tau),c_S,\uz)$ in an element of $\OO$. Take the unique isomorphism $\psi : (S,J_S(\tau),c_S,\uz)\rightarrow (\cpo,i,c_{\rpo},\ux)$, where $\ux$ is defined as above. Suppose $\uksi\in \R T_{\ux}\cpo$ is a tangent vector to $\OO$ written in Ceyhan's coordinates. Take a path $(\cpo,i,c_{\rpo},\ux_t)_{t\in\R}$ in $\OO$ such that $\ux_0 = \ux$, and $\frac{d\ux_t}{dt}_{|t=0} = \uksi$. Let $\uz_t = \psi^{-1}(\ux_t)$ and choose an isotopy $\phi_t\in Diff^+(S,c_S)$, $t\in\R$, such that for all $t\in\R$, $\uz_t = \phi_t(\uz)$. The path $(S,\phi_t^*J_S(\tau),c_S,\uz)_{t\in\R}$ coincides with $(\cpo,i,c_{\rpo},\ux_t)_{t\in\R}$. The tangent vector to this path at $t=0$ in $\coker(\DB_{\tau,\uz})$ is given by $-J_S(\tau)\frac{d}{dt}(\phi_t^*J_S(\tau))_{|t=0} = \DB_{\tau}(\frac{d}{dt}(\phi_t)_{|t=0})$ (see the proof of Lemma 1.8 in \cite{article}).

Now, recall that to orient $\coker(\DB_{\tau,\uz})$, we used the sequence
\[
0\rightarrow \ker(\DB_{\tau})\xrightarrow{ev_{\uz}} \R T_{\uz}S \xrightarrow{\delta} \coker(\DB_{\tau,\uz})\rightarrow 0.
\]
Note that if $\uzeta\in\R T_{\uz}S$, then $\delta(\zeta) = \DB_{\tau}v$, where $v\in L^{k,p}(S,TS)_{+1}$ satisfies $ev_{\uz}(v) = \uzeta$. In particular, $\delta(d\psi^{-1}(\uksi)) = \DB_{\tau}(\frac{d}{dt}(\phi_t)_{|t=0})$. Thus, we only have to check that if we take a basis of $\R T_{\uz}S$ given by a positive basis of $\ker(\DB_{\tau})$ followed by the image by $d\psi^{-1}$ of a positive basis of the tangent space of $\OO$ at $\ux$ in Ceyhan's coordinates, this basis is positive ($\R T_{\uz}S$ being oriented by the numbering of $\uz$).

For each point $x_j$, $j=1,\ldots k-2$, choose a non-zero vector $e_j\in T_{x_j}\cpo$. Then the family 
\begin{multline*}
e_1+d_{x_1} c_{\rpo}(e_1),ie_1-id_{x_1} c_{\rpo}(e_1),\ldots,e_{k-2}+d_{x_{k-2}} c_{\rpo}(e_{k-2}),\\
ie_{k-2} -id_{x_{k-2}} c_{\rpo}(e_{k-2}),(0,\ldots,0,i,0,0,\ldots,0,-i,0)
\end{multline*}
 of vectors of $T_{\ux}\cpo$ is a positive basis of the tangent space of $\OO$ at $(\cpo,i,c_{\rpo},\ux)$ in Ceyhan's coordinates. On the other hand, recall that we defined a basis $(v_1,v_2,v_3)$ of $\ker(\DB_i)$ with $v_1(z) = (z-i)(z+i)$, $v_2(z) = (1-z)(1+z)$ and $v_3(z) = z$. Thus, a positive basis of $\ker(\DB_{\tau})$ is given by $d\psi^{-1}(v_1),d\psi^{-1}(v_2),d\psi^{-1}(v_3)$. Reorganizing the vectors in the concatenated family, we indeed obtain a positive basis of $\R T_{\uz}S$, which concludes the proof.
\end{proof}

\begin{remark}\label{baserem}
  The choice of the sections $v_j$ which seemed arbitrary in Proposition \ref{dbj} is justified in the proof of Lemma \ref{orbase}. In fact, we use the following facts about the $v_j$:
  \begin{itemize}
  \item $v_1$ vanishes at $i$ and $-i$ and is real and positive between those two points,
\item $(v_2(i),v_3(i))$ is a positive basis of $T_i\cpo$.
  \end{itemize}
Any other choice of sections satisfying the above two conditions gives the same orientation of $\ker(\DB_i)$. The choice for the sections $v'_j$ was also done so that Lemma \ref{orbase} would stand.
\end{remark}

We now go on to the second point. For $\at\in\GP$, we denote by $\gl_{\at} : \ker(\vD_{u_0})\rightarrow U_{\P}$ the restriction of $\gl$ to the fiber $\{(\at,0,t_0)\}\times \ker(\vD_{u_0})$.

\begin{prop}\label{orfiber}
 Let $\at\in \GP$ be a gluing parameter with no vanishing coordinate. If $\at$ is small enough, the differential $d_{(0,0)} \gl_{\at}$ is an orientation preserving isomorphism between $\ker(\vD_{u_0})$ and $\ker(\vD_{\gl_{\at}(0,0)})$.
\end{prop}

 Using the notation of \S \ref{pargl}, define the function $h : [0,1]\times \ker(\vD_{u_0})\rightarrow L^{k,2,\epsilon}(\Sigma_{\at},X)\times \P$ by
\[
h_T(\xi,e) = (\ex_{u_{\at}}(\xi_{\at,0}+ T pr_1(\phi_{\at,0,t_0}(\xi,e))) , e+ T pr_2(\phi_{\at,0,t_0}(\xi,e))).
\]
Notice that $h_1 = \gl_{\at}$. Let $\zeta_T = T pr_1(\phi_{\at,0,t_0}(0,0)))$ and $u_T=\ex_{u_{\at}}(\zeta_T)$. Consider the function
\[
\begin{array}{c}
\F_{\at}^T : L^{k,2,\epsilon}(\Sigma_{\at},u_{T}^*TX;\uh)_{+1}\oplus \P \rightarrow L^{k-1,2,\epsilon}(\Sigma_{\at}, \Lambda^{0,1}\Sigma_{\at} \fo{j_0}{\delta(t_0)} u_{T}^*TX)_{+1} \\
(\xi,e)\mapsto \pt^{t_0}_{\ex_{u_{T}} \xi\rightarrow u_{T}}(\DB_{j_0,\delta(t_0)}(\ex_{u_{T}} \xi) - e(.,\ex_{u_{T}}\xi,t_0)),
\end{array}
\]
and denote by $D_{T}$ its derivative at the point $(0,0)$, so that $D_0 = \vD_{\at,0,t_0}$ and $D_1 = \vD_{\gl_{\at}(0,0)}$. The operators $D_T$ are Fredholm and it follows from the inequality (B.8.8) in \cite{pardon} that they are surjective and that they admit uniformly bounded right inverses $Q_T$ so that $H_T = (\id-Q_TD_T)\circ d_{(0,0)} h_T : \ker(\vD_{u_0})\rightarrow \ker(D_T)$ gives a morphism between the trivial bundle $\ker(\vD_{u_0})$ over $[0,1]$ and the bundle $\ker(D_T)_{T\in [0,1]}$. Moreover, $H_1 = d_{(0,0)} \gl_{\at}$.

In fact, let us show as in Lemma 10.6.3 in \cite{MDS} the following Lemma.

\begin{lemma}\label{ltoprove}
 There exists a constant $C>0$ such that for all $\at$ small enough, all $T\in [0,1]$ and all $(\xi,e)\in\ker(\vD_{u_0})$, $H_T$ satisfies 
\begin{equation}\label{toprove}
\lVert H_T(\xi,e)\rVert_{k,2,\epsilon} \geq C \lVert (\xi,e)\rVert_{k,2,\epsilon}.
\end{equation}
\end{lemma}

\begin{proof}
First, we can differentiate $h_T$ to get
\[
d_{(0,0)}h_T(\xi,e) = (E(u_{\at},\zeta_T)(\xi_{\at,t_0} + Tpr_1(d_{(0,0)}\phi_{(\at,0,t_0)}(\xi,e))),e+Tpr_2(d_{(0,0)}\phi_{(\at,0,t_0)}(\xi,e))),
\]
where $E(x,\zeta):T_xX\rightarrow T_{\ex_x(\zeta)}X$ is a smooth family of uniformly invertible linear maps for $\zeta\in T_xX$ small enough, obtained by differentiating the exponential map. Denote by $(\xi,e)_T = (\xi_{\at,t_0} + Tpr_1(d_{(0,0)}\phi_{(\at,0,t_0)}(\xi,e)),e+Tpr_2(d_{(0,0)}\phi_{(\at,0,t_0)}(\xi,e)))$. Then, as remarked by McDuff-Salamon, (10.6.22) in \cite{MDS}, and Pardon (B.8.8) in \cite{pardon}, there exists a constant $C_1>0$ such that 
\[
\lVert D_Td_{(0,0)}h_T(\xi,e) - \pt^{t_0}_{\ex_{u_{T}} \zeta_T\rightarrow u_{T}} D_0(\xi,e)_T\rVert_{k-1,2,\epsilon}\leq C_1 \lVert \zeta_T \rVert_{k,2,\epsilon}\lVert (\xi,e)\rVert_{k,2,\epsilon}.
\]
In particular, there is a constant $C_2>0$ such that we have
\begin{equation}\label{inegun}
\lVert D_Td_{(0,0)}h_T(\xi,e)\rVert_{k-1,2,\epsilon}\leq C_1 \lVert \zeta_T \rVert_{k,2,\epsilon}\lVert (\xi,e)\rVert_{k,2,\epsilon} + C_2\lVert D_0(\xi,e)_T\rVert_{k-1,2,\epsilon}.
\end{equation}
It follows from the construction of $\phi_{(\at,0,t_0)}$ (see Proposition 24 in \cite{floermonop}) and from the estimates in Lemma B.6.1 in \cite{pardon} that there is a constant $C_3>0$ with
\begin{gather}
\lVert \zeta_T \rVert_{k,2,\epsilon}\leq C_3 \sum_{i=1}^m\e^{-(1-\epsilon)R_i}\label{inegzeta} \\ \intertext{and}
\lVert d_{(0,0)}\phi_{(\at,0,t_0)}(\xi,e)\rVert_{k,2,\epsilon} \leq C_3 \sum_{i=1}^m\e^{-(1-\epsilon)R_i} \lVert (\xi,e)\rVert_{k,2,\epsilon}.\label{inegphi}
\end{gather}
Combining the inequalities \eqref{inegzeta} and \eqref{inegphi} with \eqref{inegun} we get a positive constant $C_4$ with
\begin{equation}
  \label{inegdeu}
  \lVert D_Td_{(0,0)}h_T(\xi,e)\rVert_{k-1,2,\epsilon}\leq C_4 \sum_{i=1}^m\e^{-(1-\epsilon)R_i}\lVert (\xi,e)\rVert_{k,2,\epsilon} + C_2\lVert D_0(\xi_{\at,t_0},e)\rVert_{k-1,2,\epsilon}.
\end{equation}
Finally, using the estimates in Lemma B.6.2 in \cite{pardon}, we obtain a positive constant $C_5$ with
\begin{equation}
  \label{inegtroi}
  \lVert D_Td_{(0,0)}h_T(\xi,e)\rVert_{k-1,2,\epsilon}\leq C_4 \sum_{i=1}^m\e^{-(1-\epsilon)R_i}\lVert (\xi,e)\rVert_{k,2,\epsilon}.
\end{equation}
  Thus, since the operators $Q_T$ are uniformly bounded, it follows from \eqref{inegtroi} that there is a positive constant $C_6$ with
  \begin{equation}
    \label{eq:1}
\lVert H_T(\xi,e)\rVert_{k,2,\epsilon} \geq  \lVert d_{(0,0)}h_T(\xi,e)\rVert_{k,2,\epsilon} - C_6 \sum_{i=1}^m\e^{-(1-\epsilon)R_i}\lVert (\xi,e)\rVert_{k,2,\epsilon}.
  \end{equation}
Using inequality \eqref{inegphi} again, we get two positive constants $C_7$ and $C_8$ with
\begin{equation}
  \label{eq:2}
  \lVert H_T(\xi,e)\rVert_{k,2,\epsilon} \geq C_7 \lVert (\xi_{\at,t_0},e)\rVert_{k,2,\epsilon} - C_8 \sum_{i=1}^m\e^{-(1-\epsilon)R_i}\lVert (\xi,e)\rVert_{k,2,\epsilon}.
\end{equation}
Since $\ker(\vD_{u_0})$ is finite dimensional and the norm of $(\xi_{\at,t_0},e)$ is bounded below by the $L^2$ norm of $(\xi,e)$ (see (10.5.15) in \cite{MDS}), we have a positive constant $C_9$ such that
\begin{equation}
  \label{eq:3}
  \lVert (\xi_{\at,t_0},e)\rVert_{k,2,\epsilon}\geq C_9\lVert (\xi,e)\rVert_{k,2,\epsilon}.
\end{equation}
Combining the inequalities \eqref{eq:2} and \eqref{eq:3} proves the inequality \eqref{toprove}.
\end{proof}

Notice that the operator $\vD_{u_0}$ is the restriction of a surjective operator $\uvD_{u_0} : L^{k,p}(\tS,u_0^*TX;\uh)_{+1}\oplus \P\rightarrow L^{k-1,p}(\tS,\Lambda^{0,1}\tS\otimes u_0^*TX)_{+1}$. Applying Proposition \ref{lingl} to the operators $\uvD_{u_0}$ and $\vD_{\at,0,t_0}$, we have the sequence
\begin{equation}
  \label{uorientt}
  0\rightarrow \ker(\vD_{\at,0,t_0})\xrightarrow{\lPHa} \ker(\uvD_{u_0}) \xrightarrow{\fPHa} \R (u_0^*TX)_{\un}\rightarrow 0.
\end{equation}

\begin{lemma}\label{linegl}
There exists a constant $C>0$ such that for $\at$ small enough and for all $v\in\ker(\vD_{u_0})$,
\begin{equation}
  \label{inegl}
  \lVert \lPHa(H_0(v)) - v \rVert \leq C\sum_{i=1}^m\frac{1}{R_i}\lVert v \rVert,
\end{equation}
where the norm on $\ker(\vD_{u_0})$ is fixed arbitrarily.
\end{lemma}

\begin{proof}
We use the notations we introduced during the proof of Proposition \ref{lingl}. First note that there exists a constant $C_1>0$ such that for all $\at$ small enough and all $w\in\ker(\vD_{u_0})$,
  \begin{equation}
    \label{cutting}
    \lVert w\rVert\leq C_1\rVert p_{\at}(w)\lVert_{k,2,-\epsilon},
  \end{equation}
which follows from the fact that the norm $L^{k,2,-\epsilon}$ of an element of $\ker(\vD_{u_0})$ outside the gluing region controls the norm over the whole curve $\Sigma$. In our case, we obtain
\begin{equation}
  \label{firststep}
  \lVert \lPHa(H_0(v)) - v \rVert \leq C_1\lVert p_{\at}(\lPHa(H_0(v))) - p_{\at}(v) \rVert_{k,2,-\epsilon}.
\end{equation}
We rewrite
\[
p_{\at}(\lPHa(H_0(v))) - p_{\at}(v) = (p_{\at}(\lPHa(H_0(v))) - H_0(v)) + (H_0(v)-v_{\at,t_0}) + (v_{\at,t_0} - p_{\at}(v)),
\]
and we bound each term on the right-hand side.

 Using inequality \eqref{ineq}, we have positive constants $C_2,C_3$ such that
\begin{equation}
  \label{inone}
\begin{split}
  \lVert p_{\at}(\lPHa(H_0(v))) - H_0(v)\rVert_{k,2,-\epsilon}& \leq C_2 \lVert \vD_{\at,0,t_0}(p_{\at}(\lPHa(H_0(v))) - H_0(v))\rVert_{k-1,2,-\epsilon}\\
& \leq C_3 \sum_{i=1}^m\frac{1}{R_i}\lVert v\rVert,
\end{split}
\end{equation}
where the second inequality follows on one hand from the fact that $H_0(v)$ is in the kernel of $\vD_{\at,0,t_0}$ and on the other from the fact that the form $\vD_{\at,0,t_0}(p_{\at}(\lPHa(H_0(v))))$ vanishes outside of the gluing region, and its norm on the gluing region is controlled by that of the cutoff functions $\varphi_{\at}$.

The bound on the second term follows from the inequality \eqref{inegtroi}. The third term is easily bounded, as it is non zero only on the cylinders $[R_i,5R_i]\times S^1$ so that the exponential weight makes its norm exponentially small compared to that of $v$. Putting together the bounds for each of the three terms, we obtain the inequality \eqref{inegl}.
\end{proof}

\begin{proof}[Proof of Proposition \ref{orfiber}]
Lemma \ref{ltoprove} implies that $H_T$ is in fact a trivialization of $\ker(D_T)$. In particular, $d_{(0,0)}\gl_{\at}$ is an isomorphism and to check that it preserves the orientations we fixed, we can in fact check that $H_0$ preserves the orientations, where the one given on $\ker(D_0) = \ker(D_{\at,0,t_0})$ is obtained by following the trivialization $H_T$. More precisely, $\ker(D_0)$ is oriented in the same way as we oriented $\ker(D_1)=\ker(\vD_{\gl_{\at}(0,0)})$, that is using the orientation of $\P$, of $\R(T_{u_{\at}(\up)} X/ T_{u_{\at}(\up)}\uh)$ and an isomorphism between $(u_{\at}^*TX,d c_X)$ and $(\TC^2,c_{\TC})$ preserving the appropriate structures.

Thus, we focus on $H_0$. Fix once and for all a numbering $\un=(n_1,\ldots,n_m)$ of the nodes, as well as a sign for each branch of each node, as we did in \S \ref{parlin}. Using this data, the sequence \eqref{normal} in this case gives
\begin{equation}\label{uorient}
0\rightarrow \ker(\vD_{u_0})\rightarrow \ker(\uvD_{u_0}) \xrightarrow{ev_{\un}} \R (u_0^*TX)_{\un}\rightarrow 0,
\end{equation}
which orients $\ker(\uvD_{u_0})$. On the other hand, applying Proposition \ref{lingl} to the operators $\uvD_{u_0}$ and $D_{\at,0,t_0}$, $\ker(\uvD_{u_0})$ is also oriented by
\begin{equation}
  \label{uorientt}
  0\rightarrow \ker(\vD_{\at,0,t_0})\xrightarrow{\lPHa} \ker(\uvD_{u_0}) \xrightarrow{\fPHa} \R (u_0^*TX)_{\un}\rightarrow 0.
\end{equation}
We claim that $H_0$ preserves the orientations if and only if the orientations on $\ker(\uvD_{u_0})$ coming from the sequences \eqref{uorient} and \eqref{uorientt} coincide. Indeed, take a basis of $\ker(\uvD_{u_0})$ given by the concatenation of a positive basis $\uv$ of $\ker(\vD_{u_0})$ and a basis $\uw$ of the orthogonal of $\ker(\vD_{u_0})$ for some fixed scalar product, such that $ev_{\un}(\uw)$ is a positive basis of $\R (u_0^*TX)_{\un}$. The basis $(\uv,\uw)$ is positive for the orientation induced by the sequence \eqref{uorient}. Now, if $\at$ is small enough, the family $(\lPHa(H_0(\uv)),\uw)$ is a basis of $\ker(\uvD_{u_0})$. Moreover, since $\fPHa$ converges to the map $ev_{\un}$ when $\at$ goes to zero, $\fPHa(\uw)$ is also a positive basis of $\R (u_0^*TX)_{\un}$, so that the basis $(\lPHa(H_0(\uv)),\uw)$ is positive for the orientation coming from the sequence \eqref{uorientt} if and only if $H_0$ preserves the orientations. But Lemma \ref{linegl} implies that the basis $(\lPHa(H_0(\uv)),\uw)$ is always positive for the orientation given by the sequence \eqref{uorient}, which proves the claim.

Thus, we only need to show that the orientations on $\ker(\uvD_{u_0})$ coming from the sequences \eqref{uorient} and \eqref{uorientt} coincide to conclude, i.e. we need to show that the isomorphism $\ddet(D_{\at,0,t_0}) = \ddet(D_{u_0})$ given by the Proposition \ref{lingl} is orientation preserving. To this end, take the isomorphism $\Phi$ between the trivial complex vector bundle of rank $2$ over $\Sigma$ and $u_0^*T X$ we used to orient $\ker(\vD_{u_0})$. As noted in Remark \ref{remspin}, $\Phi$ induces an isomorphism $\Phi_{\at} : (u_{\at}^*TX,d c_X)\rightarrow (\TC^2,c_{\TC})$ which preserves the appropriate structures, so that we can also use $\Phi_{\at}$ to orient $\ker(\vD_{\at,0,t_0})$. Using Remark \ref{remlin}, we obtain a commutative square
\[
\xymatrix{
\ddet(\Phi_{\at}^*D_{\at,0,t_0}) \ar[d]_{\Phi_{\at}}\ar[r] & \ddet(\Phi^*D_{u_0})\ar[d]^{\Phi}\\
\ddet(D_{\at,0,t_0})\ar[r] & \ddet(D_{u_0}),
}
\]
where the horizontal isomorphisms come from Proposition \ref{lingl}. Applying the homotopy invariance for the top isomorphism, we have another square
\[
\xymatrix{
\ddet(\DB_{\TC^2,\at})\ar[d]\ar[r] & \ddet(\DB_{\TC^2})\ar[d]\\
\ddet(\Phi_{\at}^*D_{\at,0,t_0}) \ar[r] & \ddet(\Phi^*D_{u_0}),
}
\]
where the vertical isomorphisms are given by the trivialisations of the determinant bundles over an homotopy joining the pair $(\Phi_{\at}^*D_{\at,0,t_0},\Phi^*D_{u_0})$ to the pair $(\DB_{\TC^2,\at},\DB_{\TC^2})$. The result then follows from Lemma \ref{trivgl}.
\end{proof}

\bibliographystyle{plain}
\bibliography{realcurves}

\begin{thebibliography}{10}

\bibitem{atiyah}
Michael~F. Atiyah.
\newblock Riemann surfaces and spin structures.
\newblock {\em Ann. Sci. \'Ecole Norm. Sup. (4)}, 4:47--62, 1971.

\bibitem{barth}
W.~Barth, C.~Peters, and A.~Van~de Ven.
\newblock {\em Compact complex surfaces}, volume~4 of {\em Ergebnisse der
  Mathematik und ihrer Grenzgebiete (3) [Results in Mathematics and Related
  Areas (3)]}.
\newblock Springer-Verlag, Berlin, 1984.

\bibitem{bauer}
Stefan Bauer.
\newblock Almost complex 4-manifolds with vanishing first {C}hern class.
\newblock {\em J. Differential Geom.}, 79(1):25--32, 2008.

\bibitem{bishurt}
Indranil Biswas, Johannes Huisman, and Jacques Hurtubise.
\newblock The moduli space of stable vector bundles over a real algebraic
  curve.
\newblock {\em Math. Ann.}, 347(1):201--233, 2010.

\bibitem{bryanleung}
Jim Bryan and Naichung~Conan Leung.
\newblock The enumerative geometry of {$K3$} surfaces and modular forms.
\newblock {\em J. Amer. Math. Soc.}, 13(2):371--410 (electronic), 2000.

\bibitem{ceyhan}
{\"O}zg{\"u}r Ceyhan.
\newblock On moduli of pointed real curves of genus zero.
\newblock In {\em Proceedings of {G}\"okova {G}eometry-{T}opology {C}onference
  2006}, pages 1--38. G\"okova Geometry/Topology Conference (GGT), G\"okova,
  2007.

\bibitem{chen}
Xi~Chen.
\newblock A simple proof that rational curves on {$K3$} are nodal.
\newblock {\em Math. Ann.}, 324(1):71--104, 2002.

\bibitem{article}
R\'emi Cr\'etois.
\newblock D\'eterminant des op\'erateurs de {C}auchy-{R}iemann r\'eels et
  application \`a l'orientabilit\'e d'espaces de modules de courbes r\'eelles.
\newblock preprint arXiv:1207.4771, 2012.

\bibitem{art}
R{\'e}mi Cr{\'e}tois.
\newblock Automorphismes r\'eels d'un fibr\'e et op\'erateurs de
  {C}auchy-{R}iemann.
\newblock {\em Math. Z.}, 275(1-2):453--497, 2013.

\bibitem{earle}
Clifford~J. Earle.
\newblock On holomorphic families of pointed {R}iemann surfaces.
\newblock {\em Bull. Amer. Math. Soc.}, 79:163--166, 1973.

\bibitem{fkps}
Flaminio Flamini, Andreas~Leopold Knutsen, Gianluca Pacienza, and Edoardo
  Sernesi.
\newblock Nodal curves with general moduli on {$K3$} surfaces.
\newblock {\em Comm. Algebra}, 36(11):3955--3971, 2008.

\bibitem{floermonop}
A.~Floer.
\newblock Monopoles on asymptotically flat manifolds.
\newblock In {\em The {F}loer memorial volume}, volume 133 of {\em Progr.
  Math.}, pages 3--41. Birkh\"auser, Basel, 1995.

\bibitem{georzing}
Penka Georgieva and Aleksey Zinger.
\newblock The moduli space of maps with crosscaps: the relative signs of the
  natural automorphisms.
\newblock preprint arXiv:1308.1345v2, 2013.

\bibitem{hutchtaubes}
Michael Hutchings and Clifford~Henry Taubes.
\newblock Gluing pseudoholomorphic curves along branched covered cylinders.
  {II}.
\newblock {\em J. Symplectic Geom.}, 7(1):29--133, 2009.

\bibitem{Johnson}
Dennis Johnson.
\newblock Spin structures and quadratic forms on surfaces.
\newblock {\em J. London Math. Soc. (2)}, 22(2):365--373, 1980.

\bibitem{kharlr}
Viatcheslav Kharlamov and Rare{\c{s}} R{\u{a}}sdeaconu.
\newblock Counting real rational curves on {$K3$} surfaces.
\newblock preprint arXiv:1311.7621, 2013.

\bibitem{kirbysieb}
R.~C. Kirby and L.~C. Siebenmann.
\newblock Some theorems on topological manifolds.
\newblock In {\em Manifolds--{A}msterdam 1970 ({P}roc. {N}uffic {S}ummer
  {S}chool)}, Lecture Notes in Mathematics, Vol. 197, pages 1--7. Springer,
  Berlin, 1971.

\bibitem{konts}
Maxim Kontsevich.
\newblock Enumeration of rational curves via torus actions.
\newblock In {\em The moduli space of curves ({T}exel {I}sland, 1994)}, volume
  129 of {\em Progr. Math.}, pages 335--368. Birkh\"auser Boston, Boston, MA,
  1995.

\bibitem{litian}
Jun Li and Gang Tian.
\newblock Virtual moduli cycles and {G}romov-{W}itten invariants of general
  symplectic manifolds.
\newblock In {\em Topics in symplectic {$4$}-manifolds ({I}rvine, {CA}, 1996)},
  First Int. Press Lect. Ser., I, pages 47--83. Int. Press, Cambridge, MA,
  1998.

\bibitem{lockhartm}
Robert~B. Lockhart and Robert~C. McOwen.
\newblock Elliptic differential operators on noncompact manifolds.
\newblock {\em Ann. Scuola Norm. Sup. Pisa Cl. Sci. (4)}, 12(3):409--447, 1985.

\bibitem{MDS}
Dusa McDuff and Dietmar Salamon.
\newblock {\em {$J$}-holomorphic curves and symplectic topology}, volume~52 of
  {\em American Mathematical Society Colloquium Publications}.
\newblock American Mathematical Society, Providence, RI, 2004.

\bibitem{pardon}
John Pardon.
\newblock An algebraic approach to virtual fundamental cycles on moduli spaces
  of {$J$}-holomorphic curves.
\newblock preprint arXiv:1309.2370, 2013.

\bibitem{shev}
Vsevolod Shevchishin.
\newblock Pseudoholomorphic curves and the symplectic isotopy problem.
\newblock preprint arXiv:math/0010262, 2000.

\bibitem{wang}
Shuguang Wang.
\newblock On orientability and degree of {F}redholm maps.
\newblock {\em Michigan Math. J.}, 53(2):419--428, 2005.

\bibitem{wangs}
Shuguang Wang.
\newblock Orientability of real parts and spin structures.
\newblock {\em JP J. Geom. Topol.}, 7(1):159--174, 2007.

\bibitem{ww}
Katrin Wehrheim and Chris~T. Woodward.
\newblock Orientations for pseudoholomorphic quilts.
\newblock preprint available at
  http://math.mit.edu/{$\sim$}katrin/papers/orient.pdf, 2007.

\bibitem{wel1}
Jean-Yves Welschinger.
\newblock Invariants of real symplectic 4-manifolds and lower bounds in real
  enumerative geometry.
\newblock {\em Invent. Math.}, 162(1):195--234, 2005.

\bibitem{zinger}
Aleksey Zinger.
\newblock The determinant line bundle for {F}redholm operators: Construction,
  properties, and classification.
\newblock preprint arXiv:1304.6368, 2013.

\end{thebibliography}
\end{document}